\documentclass[11pt]{article}
\usepackage{amssymb,amsbsy,latexsym,amsmath,bbm,epsfig,psfrag,amsthm,mathrsfs}
\usepackage[swedish,english]{babel}
\usepackage{graphicx,tikz,esint}
\usepackage[T1]{fontenc}
\usepackage[latin1]{inputenc}
\usepackage{amsfonts}
\usepackage{amsmath}
\usepackage{amssymb}
\usepackage{epsf}
\usepackage{graphics}
\usepackage{graphicx}
\usepackage{latexsym}
\usepackage{psfrag}
\usepackage{relsize}
\usepackage{verbatim}
\usepackage{hyperref}
\usepackage{pgfplots}
\usepackage{float}
\usetikzlibrary{positioning}
\usetikzlibrary{arrows}
\usetikzlibrary{decorations.pathreplacing}
\usepackage[all]{xy}
\usepackage{url}

\theoremstyle{plain}

\newtheorem{Lem}{Lemma}
\numberwithin{Lem}{section}
\newtheorem{Prop}{Proposition}
\numberwithin{Prop}{section}
\newtheorem{Thm}{Theorem}
\numberwithin{Thm}{section}

\numberwithin{Cor}{section}

\numberwithin{Con}{section}

\theoremstyle{definition}
\newtheorem{Def}{Definition}
\numberwithin{Def}{section}

\numberwithin{hyp}{section}

\numberwithin{conj}{section}
\newtheorem{ex}{Example}
\numberwithin{ex}{section}

\theoremstyle{remark}
\newtheorem{rem}{\bf{Remark}}
\numberwithin{rem}{section}

\numberwithin{equation}{section}

\newcommand{\dv}{\partial}
\newcommand{\Om}{\Omega}
\newcommand{\dbar}{\partial_{\overline{z}}}

\newcommand{\R}{{\mathbb R}}
\newcommand{\C}{{\mathbb C}}

\newcommand{\Cs}{\mathscr{C}}

\newcommand{\Cc}{\mathcal{C}}

\newcommand{\EE}{\mathcal{E}}
\newcommand{\LL}{\mathcal{L}}
\newcommand{\A}{\mathcal{A}}
\newcommand{\BB}{\mathcal{S}}

\newcommand{\gp}{\text{\tiny{$\triangle$}}}
\newcommand{\li}{\,\,\llcorner\,\,}
\newcommand{\ri}{\,\,\lrcorner\,\,}

\newcommand{\Dp}{\mathcal{D}^{\tiny+}}
\newcommand{\Dm}{\mathcal{D}^{\tiny-}}

\newcommand{\gm}{\text{\scriptsize{$\triangledown$}}}

\begin{document}

\vspace{.4cm}

\title{\sffamily Second Order Elliptic Equations and Hodge-Dirac Operators}
\author{Erik Duse}
\date{}
\maketitle

\begin{abstract}
In this paper we show how a second order scalar uniformly elliptic equation on divergence form with measurable coefficients and Dirichlet boundary conditions can be transformed into a first order elliptic system with half-Dirichlet boundary condition. This first order system involves Hodge-Dirac operators and can be seen as a natural generalization of the Beltrami equation in the plane and we develop a theory for this equation, extending results from the plane to higher dimension. The reduction to a first order system applies both to linear as well as quasilinear second order equations and we believe this to be of independent interest. Using the first order system, we give a new representation formula of the solution of the Dirichlet problem both on simply and finitely connected domains. This representation formula involves only singular integral operators of convolution type and Neumann series there of, for which classical Calderón-Zygmund theory is applicable. Moreover, no use is made of any fundamental solution or Green's function beside fundamental solutions of constant coefficient operators. Remarkably, this representation formula applies also for solutions of the fully non-linear first order system. We hope that the representation formula could be used for numerically solving the equations. Using these tools we give a new short proof of Meyers' higher integrability theorem. Furthermore, we show that the solutions of the first order system are Hölder continuous with the same Hölder coefficient as the solutions of the second order equations. Finally, factorization identities and representation formulas for the higher dimensional Beurling-Ahlfors operator are proven using Clifford algebras, and certain integral estimates for the Cauchy transform is extended to higher dimensions. 
\end{abstract}

\tableofcontents

\newpage

\section{\sffamily Introduction}

\subsection{\sffamily Introduction}

The aim of this paper is to show that contrary to common belief, there exists a natural extension of complex analytic and quasiconformal methods to dimensions greater than 2 in the study linear elliptic partial differential equations. Using these methods we will derive global representation formulas for solutions of the Dirichlet problem for scalar linear uniformly elliptic partial differential equations with measurable coefficients on arbitrary simply connected Lipschitz domains. In particular, this will prove existence and uniqueness for weak solutions and also derive local interior estimates of higher integrability, Hölder continuity and differentiability almost everywhere.

Now what do we mean by complex analytic and quasiconformal methods?  To set the stage we are interested in the Dirichlet problem
\begin{align}\label{eq:DirPro}
\left\{
    \begin{array}{ll}
     \text{div} A(x)\nabla u(x)=0, & x\in \Om,\\
      u(x)=\phi(x), & x\in \dv \Om.  
      \end{array} \right.
\end{align}
where $\Om$ is a simply connected Lipschitz domain, $A(x)\in L^{\infty}(\Om,\LL(\R^n))$ and $\phi(x)\in W^{1/2,2}(\dv\Om)$. We say that $u$ is a weak solution of \eqref{eq:DirPro} if for all $\varphi\in C^\infty_0(\Om)$ we have 
\begin{align}\label{def:weaksol}
\int_{\Om}\langle A(x)\nabla u(x),\nabla \varphi(x)\rangle dx=0. 
\end{align}

In the plane we can introduce a and $A$-harmonic conjugate $v$ of $u$ by first defining 
\begin{align*}
B(z)=A(z)\nabla u(z)
\end{align*}
where $\star$ is the \emph{Hodge star map} given by the counterclockwise rotation 
\begin{align*}
\star=\begin{bmatrix}
   0 & -1\\
   1&0
\end{bmatrix}: \R^2\to \R^2
\end{align*}
and $z=x+iy$.

If we set $E(z)=\nabla u(z)$  where $u$ is a solution to \eqref{eq:DirPro} in the plane, then $E$ is a curl free vector field whereas $B$ is divergence free. Since the Hodge star map interchanges curl and divergence free vector fields, it follows that $\star B(z)$ is curl free and since $\Om$ is simply connected by assumption, $\star B$ potential $v$ such that $\nabla v(z)=\star A(z)\nabla u(z)$ unique up to a constant.  Moreover, the conjugate $v$ solves the second order linear elliptic equation 
\begin{align*}
\text{div}\, A^\star(z)\nabla v(z)=0
\end{align*}
in $\Om$, where 
\begin{align*}
 A^\star(z)=\star^tA^{-1}(z)\star.
\end{align*}

Using $v$ we define the complex valued function $f(z):=u(z)+iv(z)$. Now we define the pair of fields
\begin{align*}
\mathcal{F}^+(z)&=E(z)+B(z)=(I+A(z))E(z),\\
\mathcal{F}^-(z)&=E(z)-B(z)=(I-A(z))E(z).
\end{align*}
This gives the algebraic equation 
\begin{align}\label{eq:AlgComplex}
\mathcal{F}^-(z)=\mathcal{M}(z)\mathcal{F}^+(z)
\end{align}
relating the fields $\mathcal{F}^+$ and $\mathcal{F}^-$, where $\mathcal{M}(z)=(I-A(z))(I+A(z))^{-1}$ is the Cayley transform of $A(z)$. On the other hand a direct computation gives
\begin{align*}
2f_{\overline{z}}&=u_x-v_y+i(v_x-u_y),\\
2\overline{f_{z}(z)}&=u_x+v_y+i(u_y-v_x)
\end{align*}
and 
\begin{align*}
2\mathcal{F}^+(z)&=(u_x-v_y,v_x-u_y),\\
2\mathcal{F}^-(z)&=(u_x+v_y,u_y-v_x),
\end{align*}
so we have $f_{\overline{z}}=\mathcal{F}^-(z)$ and $\overline{f_{z}(z)}=\mathcal{F}^-(z)$. Combining this with the algebraic equation \eqref{eq:AlgComplex}, we get the Beltrami equation 
\begin{align}\label{eq:Beltrami}
f_{\overline{z}}=\mathcal{M}(z)\overline{f_{z}(z)}. 
\end{align}
Since any linear map $L\in \LL(\C)$ can be written on the form 
\begin{align*}
Lz=\alpha z+\beta \overline{z}
\end{align*}
for two complex numbers $\alpha,\beta\in \C$, simple algebra shows that \eqref{eq:Beltrami} can be written as an $\R$-linear Beltrami equation on the standard form 
\begin{align}\label{eq:Beltrami2}
f_{\overline{z}}=\mu(z)f_z(z)+\nu(z)\overline{f_{z}(z)}. 
\end{align}
The advantage of working with the first order system \eqref{eq:Beltrami2} over the second order equation \eqref{eq:DirPro} is that one can use the full machinery of complex potential operators and quasiconformal methods to study \eqref{eq:Beltrami2}. It is therefore natural to ask if something similar can be done in higher dimension. We will show in this paper that this is indeed the case. However, it will also be shown that the higher dimensional case is in many respects more subtle and difficult. 

A first key step towards a generalization of equation \eqref{eq:AlgComplex} to higher dimension with the purpose of studying scalar elliptic linear as well as non-linear partial differential equations was taken by Iwaniec and Sbordone in their paper \cite{IS}, where they introduced the notion of \emph{quasiharmonic fields}. To explain this concept we first observe that the vector field $B(x)=A(x)\nabla u(x)=0$ is divergence free in the sense of distributions whenever $u$ is a weak solution to $\text{div}A(x)\nabla u(x)=0$. Of course the field $E(x)=\nabla u(x)$ is curl free in the sense of distributions, and therefore the pair $(B,E)$ form a so called div-curl couple in the terminology of compensated compactness, e.g. \cite{Tat}.
If a symmetric endomorphism field $A(x)$ (i.e. $A^\ast(x)=A(x)$) satisfies the symmetric ellipticity bounds,
\begin{align*}
K^{-1}\vert v\vert^2\leq \langle A(x)v,v\rangle \leq K\vert v\vert^2\,\,\, \text{ for a.e. $x\in \Om$}
\end{align*}
then one can show that the double sided condition is equivalent to the one-sided inequality 
\begin{align*}
\vert v\vert^2+\vert A(x)v\vert^2\leq (K+K^{-1})\langle A(x)v,v\rangle \,\,\, \text{ for a.e. $x\in \Om$}. 
\end{align*}
Letting $\mathcal{K}=K+K^{-1}$ this inequality implies that any weak solution $u$ give rise to the inequality 
\begin{align}\label{def:QausiHarm}
\vert E(x)\vert^2+\vert B(x)\vert^2\leq \mathcal{K}\langle E(x),B(x)\rangle
\end{align}
for the div-curl couple $(B,E)$. Generalising this in \cite{IS}, one say that any div-curl couple $(B,E)\in L^2(\Om, \R^n\times \R^n)$ is a $\mathcal{K}$-quasiharmonic field if $\text{div }B(x)=0$ and $\text{curl }E(x)=0$
in the sense of distribution and $(B,E)$ satisfies \eqref{def:QausiHarm}. Just as in the plane one introduces the fields $\mathcal{F}^\pm=E\pm B$, which again gives the algebraic relation 
\begin{align*}
\mathcal{F}^-(x)=(I-A(x))(I+A(x))^{-1}\mathcal{F}^+(x). 
\end{align*}
 
 The novelty of the present paper is that using Hodge decompositions one can impose a gauge condition so that the there exists a multivector field $F$ with the property that 
 \begin{align*}
\mathcal{F}^-(x)=\Dm F(x),\,\,\,\, \mathcal{F}^+(x)=\Dp F(x),
 \end{align*}
 where $(\Dm,\Dp)$ is a pair of Hodge-Dirac operators, one formally self-adjoint and the other formally skew-adjoint, taking the role of the Wirtinger derivatives $(\dv_{\overline{z}},\overline{\dv_z})$ in the plane. More precisely, the Hodge-Dirac operators are related to the exterior derivative on multivectors (i.e. differential forms) through 
 \begin{align*}
\mathcal{D}^\pm=d\pm \delta,
 \end{align*}
 where $\delta=-d^\ast$ is the interior derivative. The benefit of introducing the potential field $F$ is that this allows one to use powerful integral operator techniques form Clifford and harmonic analysis to study the original Dirichlet problem. As a by product, we will see that this also allows one to generalize much of the quasiconformal methods in the plane used to to study second order elliptic equations to higher dimensions.

It is interesting to note that Clifford analysis techniques have also been used successfully in \cite{Wu1,Wu2} to prove global well-posedness of the water wave problem in 3 dimensions. 

\subsection{\sffamily Main Results}

\begin{Thm}\label{thm:Main1}
Let $\Om \subset \R^n$ with $n>2$ be a bounded simply connected $C^2$-domain. Furthermore let $A\in L^\infty(\Om, \LL(\R^n)$ be a normal matrix that satisfies the ellipticity bounds 
\begin{align}\label{eq:NormalBounds}
\lambda \vert v\vert^2\leq \langle A(x)v,v\rangle, \quad \Lambda \vert v\vert^2\leq \langle A^{-1}(x)v,v\rangle
\end{align}
for all most every $x\in \Om$ and every $v\in \R^n$. Let $u\in W^{1,2}(\Om)$ be any weak solution of 
\begin{align}\label{eq:DirPro1}
\left\{
    \begin{array}{ll}
     \text{div} \,A(x)\nabla u(x)=0, & x\in \Om,\\
      u(x)=\phi(x)\in W^{1/2,2}(\dv \Om), & x\in \dv \Om.  
      \end{array} \right.
\end{align}
Define the bivector field $v\in W^{1,2}(\Om,\Lambda^2 \R^n)$ to be the unique solution to the gauge condition equation
\begin{align}\label{eq:GaugeCond1}
\left\{
    \begin{array}{ll}
    \delta v(x)=A(x)\nabla u(x), & x\in \Om,\\
        dv(x)=0, & x\in \Om,\\
      v_T(x)=0, & x\in \dv \Om, 
      \end{array} \right.
\end{align}
where $v_T$ denotes the tangential trace. Define the multivector field $F=u+v\in W^{1,2}(\Om,\Lambda^{(0,2)} \R^n)$ and let 
\begin{align*}
\mathcal{M}(x)= (I-A(x))\circ (I+A(x))^{-1}
\end{align*}
denote the Cayley transform of $A$ and let $\widehat{\mathcal{M}}(x)$ denote the exterior extension of $\mathcal{M}(x)$ defined in Definition \ref{def:ExteriorExt}. Then $F$ solves the half-Dirichlet  problem for the Dirac-Beltrami equation 
\begin{align}\label{eq:DB1}
\left\{
    \begin{array}{ll}
    \Dm F(x)=\widehat{\mathcal{M}}(x)\Dp F(x) & x\in \Om,\\
      F_T(x)=\phi(x)\in W^{1/2,2}(\dv \Om), & x\in \dv \Om.  
      \end{array} \right.
\end{align}
Conversely, let $F\in W^{1,2}(\Om,\Lambda^{(0,2)} \R^n)$ be a solution of \eqref{eq:DB1} for some $\widehat{\mathcal{M}}(x)$ being the exterior extension of a linear map $\mathcal{M}\in L^{\infty}(\Om, \LL(\R^n))$ satisfying the uniform ellipticity condition 
\begin{align*}
\Vert \Vert \mathcal{M}(x)\Vert\Vert_{\infty}=M<1.
\end{align*} 
Here $ \Vert \mathcal{M}(x)\Vert$ denotes the operator norm of the linear map $\mathcal{M}(x)$. Then the function $u=\langle F\rangle_0\in W^{1,2}(\Om,\R)$ solves \eqref{eq:DirPro1} with 
\begin{align*}
A(x)=(I-\mathcal{M}(x))\circ (I+\mathcal{M}(x))^{-1}. 
\end{align*}
\end{Thm}

\begin{rem}
The same results also applies to a Lipschitz domain $\Om$, the only difference is that the field $F$ need not belong to $W^{1,2}(\Om,\Lambda^{(0,2)})$. In this case $F$ belongs to the partial Sobolev space
\begin{align*}
F\in W^{1,2}_{d,\delta}(\Om,\Lambda^{(0,2)})=\{F\in \mathscr{D}'(\Om,\Lambda^{(0,2)}): F,dF,\delta F\in L^2(\Om,\Lambda)\}. 
\end{align*}
Alternatively we may always change domain using a Lipschitz homeomorphism from a Lipschitz domain to a $C^2$-domain. For a similar result for class of uniformly non-linear elliptic second order equations see subsection \ref{subsec:NonLinDir}. 
\end{rem}

For the corresponding statements in the case when $\Om$ is finitely connected or we have an inhomogeneous equation see Theorem \ref{thm:MultiDB} and Theorem \ref{thm:DBIn} respectively. 
We now give solvability results for the general Dirac--Beltrami equation for a general field $F\in W^{1,2}_{d,\delta}(\Om,\Lambda)$ not necessarily associated to a scalar second order equation where the solvability cannot be deduced from the solvability of the scalar equation. 
\begin{Thm}[Solvability of the Dirac--Beltrami equation]
\label{thm:Solvability}
Let $\Om \subset \R^n$ be a bounded simply connected Lipschitz domain and let $\mathcal{M}\in L^{\infty}(\Om, \LL(\Lambda \R^n))$ satisfying the uniform ellipticity condition 
\begin{align*}
\Vert \Vert \mathcal{M}(x)\Vert\Vert_{\infty}=M<1.
\end{align*} 
Then the equation 
\begin{align*}
\left\{
    \begin{array}{ll}
    \Dm F(x)=\mathcal{M}(x)\Dp F(x) & x\in \Om,\\
      F_T(x)=\phi(x)\in W^{1,1/2}(\dv \Om), & x\in \dv \Om.  
      \end{array} \right.
\end{align*}
has a solution $F\in W^{1,2}_{d,\delta}(\Om,\Lambda^{(1,2)})$ for every $\phi\in W^{1,1/2}(\Om)$, unique up to elements in $\mathcal{H}_T(\Om)$. 
\end{Thm}

\begin{Thm}\label{thm:RepSol}
The unique solution to the half-Dirichlet problem \eqref{eq:DB1} is given by 
\begin{align}\label{eq:RepForm1}
F(x)=\mathcal{C}_T^+\circ \Pi_{\mathcal{M}}\Dp H(x)+H(x)
\end{align}
where 
\begin{align*}
\Pi_{\mathcal{M}}&=(I-\BB_T\circ \mathcal{M})^{-1}\circ \BB_T\mathcal{M}
\end{align*}
and 
\begin{align*}
H(x)=\mathcal{E}^-\circ (I+\mathcal{H}_{TN}^-)\phi. 
\end{align*}
and $\BB_T$ is the tangential Beurling-Ahlfors transform from Definition \ref{def:BTan}, $\Cc^+_T$ is the tangential Cauchy transform from Definition \ref{def:CTan}, $\mathcal{H}^-_{TN}$ is the tangential to normal Hilbert transform from Definition \ref{def:HTan} and $\mathcal{E}^-$ is the Cauchy integral from Definition \ref{def:Cint}. In addition, we get a representation formula for the unique solution to the Dirichlet problem \ref{eq:DirPro1} through
\begin{align}
u(x)=\langle F(x)\rangle_0. 
\end{align}
\end{Thm}

\begin{rem}
For the solvability and representation formula for the boundary value problem of the fully nonlinear Dirac--Beltrami equation see Theorem \ref{thm:NonLinDBRep}. 
\end{rem}

\begin{Thm}
Let $\Om \subset \R^n$ be a bounded domain and assume that $n\geq 3$ is odd. Let $\mathcal{M}\in L^{\infty}(\LL(\Lambda \R^n))$ and set $M=\Vert \Vert \mathcal{M}(x)\Vert\Vert_{L^\infty}(\Om)$. If 
\begin{align*}
M\Vert \BB\Vert_{2n}<1,
\end{align*}
where $\Vert \BB\Vert_{2n}$ denotes the norm of the Beurling-Ahlfors transform on $L^{2n}(\Om,\Lambda \R^n)$, then any solution $F\in W^{1,2}_{d,\delta}(\Om,\Lambda \R^n)$ of the Dirac--Beltrami equation
\begin{align*} 
\Dm F(x)=\mathcal{M}(x)\Dp F(x),
\end{align*}
is locally Hölder continuous with Hölder coefficient $\alpha=1/2$. 
\end{Thm}

\begin{rem}
Note that in this generality the Hölder continuity of $F$ does not follow from De Giorgi-Nash-Moser theory. 
\end{rem}

\subsection{\sffamily Relation to Similar Results in the Literature}

The resolution of Kato's square root conjecture in the seminal paper \cite{AHLMT} have given rise to a new developments in the solvability and well-posedness for elliptic systems. In particular, this allowed for the bounded $\mathcal{H}^\infty$-calculus of unbounded operators as developed by A. McIntosh and collaborators to be used in \cite{AAM,AR}, to establish solvability and well-posedness for very general boundary values of general second order elliptic \emph{systems} of the form 
\begin{align}\label{eq:System}
\LL^\alpha u(t,x)=\sum_{i,j=0}^n\sum_{\beta=1}^m\dv_i\bigg(A_{i,j}^{\alpha,\beta}(t,x)\dv_j u^\beta(t,x)\bigg)=0, \quad \alpha=1,2,..,m,
\end{align}
in the upper-half space $\R^{n+1}_+=\{(t,x)\in \R\times \R^{n}: t>0\}$, where $\dv_0=\dv_t$, and $\dv_j=\frac{\dv}{\dv x_j}$, $j=1,2,...,n$. It also allowed for the study regularity up to the boundary for very general boundary conditions. Here $A\in L^{\infty}(\R^{n+1}_+,\LL(\C^{(1+n)m}))$ and satisfying the strict accretivity condition $(3)$ in \cite{AR}. Here complex valued vector fields are allowed. At this level of generality solutions need not be continuous and De Giorgi-Nash theory does not apply. Moreover, in \cite{AAM,AR} representation formulas for the Dirichlet problem for \eqref{eq:System} are derived using functional calculus for unbounded \emph{non-selfadjoint} bisectorial operators. The method in \cite{AAM,AR} is as in this work to reduce \eqref{eq:System} to a first order system. This equation is however very different from the Dirac--Beltrami equation \eqref{eq:DB1} in the present paper, and the techniques used in this paper are very different. In particular we do not need to use any advanced functional calculus for the representation formula \eqref{eq:RepForm1} beyond Neumann series. Moreover, our methods also apply to finitely connected domains and nonlinear equations which are beyond the techniques in \cite{AR}. On the other hand, the methods in \cite{AAM,AR} answer different questions, in particular solvability and boundary regularity in much more general spaces than $W^{1,2}$, and applies to second order systems and boundary value problems to which the methods in this paper cannot presently be used.

\subsection*{Acknowledgements}
Erik Duse was supported by the Knut and Alice Wallenberg Foundation KAW grant 2016.0416. The author thanks Björn Gustafsson for providing the proof of Lemma 7.3.


\section{\sffamily Notational Conventions}

We will let $\mathscr{H}^k$ denote the $k$-dimensional Hausdorff measure normalized so that $\mathscr{H}^n=\mathscr{L}^n$ where $\mathscr{L}^n$ is the Lebesgue measure of $\R^n$. Furthermore, for any Lipschitz hypersurface $\Sigma\subset \R^n$ we will always let $\sigma=\mathscr{H}^{n-1}\lfloor \Sigma$ denote the surface measure of $\Sigma$ without explicitly referring to $\Sigma$, i.e., we will omit any subscript $\sigma_\Sigma$.

In what follows we will let $\sigma_{n-1}$ denote the $n-1$-dimensional volume of the unit sphere $\mathbb{S}^{n-1}=\{x\in \R^n: \vert x\vert=1\}$ in $\R^n$ and $\omega_{n}$ the volume of the unit ball $\mathbb{B}^{n}=\{x\in \R^n: \vert x\vert<1\}$ in $\R^n$. 

If $V$ is a finite dimensional euclidean vector space we let $\LL(V)$ denote the space of linear operators on $V$. If $A\in \LL(V)$ we let $\Vert A\Vert$ denote the operator norm of $A$ and $\vert A\vert$ the Hilbert-Schmidt norm of $A$, i.e., 
\begin{align*}
\Vert A\Vert&=\sup_{\vert x\vert=1}\vert Ax\Vert,\\
\vert A\vert&=\sqrt{\langle A,A\rangle}=\sqrt{\text{tr}(A^\ast A)}. 
\end{align*}

Define the Hilbert space
\begin{align*}
L^2(\dv \Om, \Lambda):=L^2(\dv \Om, \sigma;\Lambda \R^n)
\end{align*}
with the $L^2$-inner product 
\begin{align*}
\langle f,g\rangle=\int_{\dv \Om}\langle f(x),g(x) \rangle d\sigma(x),
\end{align*}
Similarly, we let
\begin{align*}
L^2(\dv \Om, \Lambda^k):=L^2(\dv \Om, \sigma;\Lambda^k\R^n ), \quad L^2(\dv \Om, \Lambda^{(i_1,i_2,...,i_p)}):=L^2(\dv \Om, \sigma;\bigoplus_{l=1}^p,\Lambda^{i_l}\R^n ).
\end{align*}
We also write
\begin{align*}
W^{1,k}(\dv \Om, \Lambda):=W^{1,k}(\dv \Om, \Lambda\R^n),\quad W^{1,k}(\dv \Om, \Lambda^p):=W^{1,k}(\dv \Om, \Lambda^p\R^n). 
\end{align*}
for the classical Sobolev spaces. Since the letter $H$ will be used to denote de Rham and Hardy spaces in this paper, in order to avoid confusion we will not use the notation $H^1(\Om,\Lambda)=W^{1,2}(\Om,\Lambda)$ and $H^{1/2}(\dv \Om,\Lambda)=W^{1/2,2}(\dv \Om,\Lambda)$ for the trace space. 


\section{\sffamily Algebraic preliminaries}

\subsection{\sffamily Exterior and Clifford Algebra}

In this section we give a brief introduction to the relevant algebraic structures that will be used later on. For a more complete account we refer the reader to \cite{R}.

Let $V$ be an $n$-dimensional real inner product vector space with inner product denoted by $\langle \cdot,\cdot \rangle$. We denote by $\Lambda V$ the exterior algebra of $V$ and by abuse of notation we still denote by $\langle \cdot,\cdot \rangle$ the \emph{induced} euclidean inner product on $\Lambda V$ given by the Gram determinant 
\begin{align*}
\langle u_1\wedge...\wedge u_k,v_1\wedge...\wedge v_k\rangle=\det[\langle u_j,v_j\rangle]_{i,j=1}^k
\end{align*}
for $u_1,...,u_k,v_1,...,v_k\in V$. 

Let $\overline{n}:=\{1,2,...,n\}$. For an ordered subset $s=\{s_1<s_2<...<s_k\}\subset \overline{n}$ we denote by $\vert s\vert=k$ its cardinality. If $\{e_j\}_{j=1}^n$ is an ON-basis for $V$, we denote by $\{e_s\}_{s\subset {\bar{n}}}$ the induced $ON$-basis for $\Lambda V$, with 
\begin{align*}
e_s:=e_{s_1}\wedge....\wedge e_{s_k}. 
\end{align*}

Following \cite{R}, we define the left and right \emph{adjoint product} of $\wedge$, called \emph{left interior} and \emph{right interior multiplication}, denoted $\ri$ and $\li$ respectively, so that for all $\Theta, w_1,w_2\in \Lambda V$  we have 
\begin{align*}
\langle \Theta \ri w_1,w_2\rangle &=\langle  w_1,\Theta \wedge w_2\rangle, \\
\langle w_1\li \Theta ,w_2\rangle &=\langle  w_1,w_2 \wedge \Theta \rangle 
\end{align*}

Let $q$ be the associated quadratic form to the inner product $\langle\cdot,\cdot\rangle$. We define the Clifford algebra $\Delta V^q$  to be the associative algebra generated by the relations 
\begin{align*}
e_i\gp e_j+e_j\gp e_i&=2\langle e_i,e_j\rangle\\
\end{align*}

As vector spaces there exists a \emph{canonical isomorphism} (independent of basis) such that $\Lambda V\cong \Delta V$, see Proposition 1.2 in \cite[p. 10]{LM}. In what follows we will therefore just write $\Lambda V$ for the underlying vector space which we call the \emph{exterior space}, but equipped with different products. Moreover, for any orthogonal multivectors $e_s,e_t\in \Lambda \R^n$, $e_s\wedge e_t=e_s\gp e_t$ and for any vector $v\in V$ and multivector $w\in \Lambda V$ we have the Riesz identities 
\begin{align}\label{eq:RId1}
v\gp w&=v\ri w+v\wedge w,\\
\label{eq:RId2}
w\gp v&=w\li v+w\wedge v,
\end{align}
connecting the Clifford product to the wedge products and the left and right interior multiplication. Since $\Lambda V=\bigoplus_{k=0}^n \Lambda^k V$ is a graded vector space, we denote by $\langle w\rangle_k$ the orthogonal projection of $\Lambda V$ onto $\Lambda^k V$ so that any multivector $w\in \Lambda V$ can be written $\langle w\rangle=\sum_{k=0}^n\langle w\rangle_k$. Furthermore, we set
\begin{align*}
\langle w\rangle_{ev}&=\sum_{\substack{j=0\\ 2j\leq n}}\langle w\rangle_{2j},\\
\langle w\rangle_{odd}&=\sum_{\substack{j=0\\ 2j\leq n}}\langle w\rangle_{2j+1}.
\end{align*}

We define the main involution and anti-involution by
\begin{align*}
\widehat{w}&=\sum_{j=0}^n(-1)^j\langle w\rangle_j\\
\overline{w}&=\sum_{j=0}^n(-1)^{j(j-1)/2}\langle w\rangle_j,
\end{align*}
The main involution is both an exterior algebra as well as a Clifford algebra automorphism, i.e. for all $w_1,w_2\in \Lambda V$, $\widehat{w_1\wedge w_2}=\widehat{w_1}\wedge \widehat{w_2}$ and $\widehat{w_1\gp w_2}=\widehat{w_1}\gp \widehat{w_2}$. The main anti-involution is both an exterior algebra as well as a Clifford algebra anti-automorphism, i.e. for all $w_1,w_2\in \Lambda V$, $\widehat{w_1\wedge w_2}=\widehat{w_2}\wedge \widehat{w_1}$ and $\widehat{w_1\gp w_2}=\widehat{w_2}\gp \widehat{w_1}$. Using the involution we have the following useful commutation relations
\begin{align*}
v\wedge w=\widehat{w}\wedge v, \,\, v\ri w=-\widehat{w}\li v, \text{ for all $v\in V$ and $w\in \Lambda W$}.
\end{align*}
 Moreover, the anti-involution has the following useful properties with respect to the inner product 
 \begin{align*}
 \langle w_1\gp w_2,w_3\rangle=\langle w_2,\overline{w_1}\gp w_3\rangle=\langle w_1,w_3\gp \overline{w_2}\rangle.
 \end{align*} 
Thus, the anti-involution expresses the adjoint operation of Clifford multiplication. Furthermore, we let $\Lambda^{ev}V$ and $\Lambda^{od}V$ correspond to the orthogonal decomposition of $\Lambda V$ into the eigenspaces corresponding to the eigenvalues $+1$ and $-1$ under the main involution. 
\begin{Def}[Hodge Duality]
Let $V$ be an $n$-dimensional Euclidean vector space and let $e_V$ denote the volume form of $V$, which in an ON-basis $\{e_j\}_j$ of $V$ is given by 
\begin{align*}
e_V=e_1\wedge ...\wedge e_n.
\end{align*}
Then, for any $w\in \Lambda V$ we define the \emph{left and right Hodge star map} by
\begin{align*}
\star w&:=e_V \li w=e_V\gp \overline{w},\\
w\star&:=w\ri e_V=\overline{w}\gp e_V. 
\end{align*}
\end{Def}

The Hodge star maps have the useful properties that $\star(w\star)=(\star w)\star=w$, $\star \star w=\widehat{w}=w\star \star$ and $(w_1\wedge (\star w_2))\star=\langle w_1,w_2\rangle$ for any $w,w_1,w_2\in \Lambda V$.

\begin{Lem}\label{lem:Norm}
If $v\in V$ and $w\in \Lambda V$, then $\vert v\gp w\vert=\vert w\gp v\vert=\vert v\vert \vert w\vert$.
\end{Lem}

\begin{proof}
\begin{align*}
\vert v\gp w\vert^2=\langle v\gp w,v\gp w\rangle=\langle w,v\gp v\gp w\rangle=\langle w,\vert v\vert^2 w\rangle=\vert v\vert^2\vert w\vert^2. 
\end{align*}
\end{proof}

Note that in general it is {\bf not} true that for any multivectors $w_1,w_2$ we have $\vert w_1\gp w_2\vert \leq \vert w_1\vert \vert w_2\vert$. What is true is that we always have $\vert w_1\gp w_2\vert \leq 2^n\vert w_1\vert \vert w_2\vert$ where $n$ is the dimension of $V$.

We now assume that the inner product $\langle \cdot,\cdot \rangle_+$ on $V$ is euclidean. We associate the anti-euclidean inner product $\langle \cdot,\cdot \rangle_-=-\langle \cdot,\cdot \rangle_+$ on $V$. We denote the two Clifford algebras with respect to the euclidean and anti-euclidean inner product by $\Delta V^+$ and $\Delta V^-$ and denote the associated Clifford products by $\gp$ and $\gm$ respectively and call them positive and negative Clifford multiplication. Note that for any vectors $v_1,...,v_k,u_1,...,u_k\in V$ we 
\begin{align*}
\langle v_1\wedge...\wedge v_k,u_1\rangle_-&=\det[\langle v_j,u_k \rangle_-]_{j,k}=\det[-\langle v_j,u_k \rangle_+]_{j,k}\\&=\det[\langle \widehat{v_j},u_k \rangle_+]_{j,k}=\langle \widehat{v_1\wedge...\wedge v_k},u_1\rangle_+. 
\end{align*}
Therefore, by multilinearity we have for all $w_1,w_2\in \Lambda V$
\begin{align*}
\langle w_1,w_2\rangle _-=\langle \widehat{w_1},w_2\rangle _+=\langle w_1,\widehat{w_2}\rangle _+. 
\end{align*}

We now prove the following useful anti-commutation relations for positive and negative Clifford multiplication with respect to vectors. 

\begin{Prop}
For all $u,v\in V$ and $w\in \Lambda V$ we have 
\begin{align}\label{eq:AntiCom1}
u\gp (v\gm w)=-v\gm (u\gp w).
\end{align}
and 
\begin{align}\label{eq:AntiCom2}
v\gp w=\widehat{w}\gm v. 
\end{align}
\end{Prop}

\begin{proof}
We first prove \eqref{eq:AntiCom2}. We use the Riesz identities for positive and negative Clifford multiplication 
\begin{align*}
v\gp w&=v\ri w+v\wedge w, \quad v\gm w=-v\ri w+v\wedge w&
w\gp v&=w\li v+w\wedge v, \quad w\gm v=-w\li v+ w\wedge v. 
\end{align*}
By Proposition 2.6.3 in \cite{R}, 
\begin{align*}
v\wedge w=\widehat{w}\wedge v, \quad v\ri w=-\widehat{w}\li v 
\end{align*}
Thus, 
\begin{align*}
v\gp w&=v\ri w+v\wedge w=-\widehat{w}\li v +\widehat{w}\wedge v=\widehat{w}\gm v.
\end{align*}
We now prove \eqref{eq:AntiCom1}. We use the anti-commutation relations 
from 
\begin{align*}
u\ri(v\wedge w)+v\wedge(u\ri w)=\langle u,v\rangle w
\end{align*}
from Theorem 2.8.1 in \cite{R} and expand using the Riesz identities. This gives using Proposition 2.6.3 in \cite{R}, 
\begin{align*}
u\gp (v\gm w)&=-u\ri (v\ri w)+u\ri(v\wedge w)-u\wedge(v\ri w)+u\wedge(v\wedge w)\\
&=-(v\wedge u)\ri w-v\wedge (u\ri w)+\langle u,v\rangle w+v\ri (u\wedge w)-\langle u,v\rangle w-v\wedge(u\wedge w)\\
&=(u\wedge v)\ri w-v\wedge (u\ri w)+v\ri (u\wedge w)-v\wedge(u\wedge w)\\
&=v\ri (u\ri w)-v\wedge (u\ri w)+v\ri (u\wedge w)-v\wedge(u\wedge w)\\
&=v\ri (u\gp w)-v\wedge (u\gp w)\\
&=-v\gm(u\gp w).
\end{align*}
\end{proof}

\subsection{\sffamily Relation between Bivectors and Complex Structures}

For the space of bivectors $\Lambda^2V$ there is a direct connection to geometry. In particular we have ;
\begin{Thm}[Chap. 6.5 \cite{LS}]
Let $V$ be a euclidean vector space. Then there exists a Lie algebra homomorphism 
\begin{align*}
\Lambda^2V\cong \mathfrak{so}(V),
\end{align*}
where the Lie bracket in $\Lambda^2 V$ is given by the Clifford commutator, i.e., for any $b_1,b_2\in \Lambda^2V$
\begin{align*}
[b_1,b_2]_\gp =b_1\gp b_2-b_2\gp b_1. 
\end{align*}
\end{Thm}

In particular, for any skew symmetric linear map $A\in \LL(V)$ there exists a unique bivector $b$ such that for any $v\in V$
\begin{align*}
Av=b\li v. 
\end{align*}

If we now assume that $\text{dim}V=2n$ and that we are given an ON-eigenbasis $\{e_1^+,...,e_n^+,e_1^-,...,e_n^-\}$ of $V$ we can form the bivectors $J_k:=e_k^+\wedge e_k^-$  and the bivector 
\begin{align*}
J:=\sum_{k=1}^ne_k^+\wedge e_k^-=\sum_{k=1}^nJ_k. 
\end{align*}

It is easy to any such $J$ becomes a complex structure on $V$. In particular, we have $\overline{J}=-J$. Furthermore, in dimension two, $J=e_1\wedge e_2$ coincides with the volume form and under the algebra isomorphism 
$\Delta^{ev}V\cong \C$ we have the identification $J\cong i$.

\subsection{\sffamily Exterior Extensions of Linear Maps and Grassmann Duality}

\begin{Def}\label{def:ExteriorExt}
For a linear operator $T\in \LL(V)=\text{Hom}(V,V)$, we define the \emph{exterior extension} of $T$, denoted $\widehat{T}$, to be unique linear operator $\widehat{T}\in \LL(\Lambda V)$ which is an exterior algebra homomorphism such that 
\begin{itemize}
\item[(i)] $\widehat{T}(1)=1$.
\item[(ii)] $\widehat{T}\vert_V=T$.
\item[(iii)] $\widehat{T}(\Lambda^k V)\subseteq \Lambda^kV$ for all $k=0,1,...,n$. 
\item[(iv)] For any $w_1,w_2\in \Lambda V$, $\widehat{T}(w_1\wedge w_2)=\widehat{T}(w_1)\wedge \widehat{T}(w_2)$.
\end{itemize}
The restriction of $\widehat{T}$ to $\Lambda^kV$ will be denoted by $\widehat{T}^k$. 
\end{Def}

Note that $\text{tr}(\widehat{T}^n)=\det(T)$, and that the characteristic polynomial of $T$ can be written according to 
\begin{align*}
\text{ch}_T(z)=\sum_{k=0}^n(-1)^k\text{tr}(\widehat{T}^k)z^{n-k}. 
\end{align*}

\begin{Lem}\label{lem:Refl}
Let $V$ be a euclidean vectors space and let $\mathbf{R}(v)$ denote the orthogonal reflection in the hyperplane orthogonal to the vector $v$. Then for any $w\in \Lambda V$
\begin{align}\label{ReflFormula}
\widehat{\mathbf{R}(v)}w=v\gp \widehat{w}\gp v^{-1}
\end{align}
Moreover, for any $T\in O(V)$, there exists a \emph{rotor} $q\in \text{Pin}(V)\subset \Delta V$ such that for any $w\in \Lambda V$,
\begin{align}\label{Rot}
\widehat{\mathbf{T}}w=q\gp \widehat{w}\gp q^{-1},
\end{align}
\end{Lem}
where $\text{Pin}(V)$ is the pin group, i.e., the double cover of $O(V)$. If $T\in SO(V)$, then $q\in \text{Spin}(V)$, where $\text{Spin}(V)$ is the spin group, i.e., the double cover of $SO(V)$.
\begin{proof}
See Proposition 4.1.10 in \cite{R}.
\end{proof}

Explicitly the Pin and the Spin groups can be represented within the Clifford algebra in the following way. Again assume that $V$ is a euclidean vector space and let $\Delta V$ denote the euclidean Clifford algebra with respect to $V$. Define the \emph{Clifford cone} $\widehat{\Delta}V$ to consist of all multivectors $q\in \Delta V$ such that there exists finitely many vectors $v_1,v_2,...,v_k$ such that 
\begin{align*}
q=v_1\gp v_2\gp ...\gp v_k.
\end{align*}
In particular, by Proposition 4.1.5 in \cite{R}, $q\in \widehat{\Delta}V$ if and only if $\widehat{q}\gp v\gp q^{-1}\in V$ for all $v\in V$. Then 
\begin{align*}
\text{Pin}(V)&=\{q\in \widehat{\Delta}V: \langle q,q\rangle=\pm 1\},\\
\text{Spin}(V)&=\{q\in \text{Pin}(V): \Delta^{ev}V\},\\
\end{align*}
where $\Delta^{ev}V$ denotes the even subalgebra of $\Delta V$. 

\begin{Def}\label{def:Project}
Let $\nu \in V$ be a unit vector.  Define the linear operators $\mathbf{P}_T^\nu$ and $\mathbf{P}_N^\nu$ by 
\begin{align*}
\mathbf{P}_T^\nu w&:=\nu\ri (\nu \wedge w),\\
\mathbf{P}_N^\nu w&:=\nu\wedge (\nu \ri w),
\end{align*}
for all $w\in \Lambda V$. 
\end{Def}

In particular we note that $\mathbf{P}_T^\nu+\mathbf{P}_N^\nu=I$, and that for any $u\in V$, $\mathbf{P}_T^\nu u$ is the orthogonal projection onto the linear space spanned by $\nu$. 

\begin{Def}[Grassmann dual map]
Let $T\in \LL(\Lambda V)$. The \emph{Grassman dual linear map} $T^c\in \LL(\Lambda V)$ is defined by
\begin{align}\label{def:DualL}
T^c(w):=T(\star w)\star
\end{align}
for all $w\in \LL(\Lambda V)$.
\end{Def}
From the definition the following properties hold:
\begin{itemize}
\item[(i)] $(T^c)^c=T$.
\item[(ii)] $(T^c)(e_V)=e_V$.
\item[(iii)] $(S\circ T)^c=S^c\circ T^c$
\item[(iv)] $I^c=I$.
\item[(v)] If $T(\Lambda^pV)\subset \Lambda^qV$ then $T^c(\Lambda^{\text{dim}V-p}V)\subset \Lambda^{\text{dim}V-q}V$
\item[(vi)] If $T(\Lambda^pV)\subset \Lambda^pV$ for all $p=0,1,...,\text{dim}V$, then $(T^\ast)^c=(T^c)^\ast$. 
\end{itemize}
Here $I$ denotes the identity map. Using the exterior extension and the Grassmann dual the cofactor formula can be expressed as
\begin{align*}
(\widehat{T})^c\circ \widehat{T^\ast}=\det(T)I
\end{align*}
for $T\in \LL(V)$, see Proposition 3.7 \cite{LS}.

\begin{Def}[Self-dual linear maps]
A linear map $T\in \LL(V)$ is said to be \emph{self-dual} if 
\begin{align*}
(\widehat{T})^c=\widehat{T}.
\end{align*}
\end{Def}

\begin{Lem}\label{lem:Conformal}
If $T\in \LL(V)$ is self-dual then $T$ is conformal. 
\end{Lem}

\begin{proof}
The cofactor formula and self-duality implies 
\begin{align*}
(\widehat{T})^c\circ  \widehat{T^\ast}=\widehat{T}\circ \widehat{T^\ast}=\widehat{T\circ T^\ast}=\det(T)I\Longrightarrow T\circ T^\ast=\det(T)I\vert_{V}. 
\end{align*}
\end{proof}

For more background on the Grassmann dual we advice the reader to consult \cite[Chapter 3]{LS} and \cite{BBR}.

\section{\sffamily Analytical preliminaries}

\subsection{\sffamily Tangential and Normal Trace Maps.}

\begin{Def}
A bounded domain $\Om\subset \R^n$ will be called a \emph{Lipschitz domain} or \emph{strongly Lipschitz domain} if for each point $p\in \dv \Om$ there exists a Lipschitz function $\phi_p: \R^{n-1}\to \R$ and a neighbourhood $V_p$ of $p$ and a neighbourhood $U_p$ of $0$ in $\R^{n-1}$ and a rotation $R_p\in O(\R^n)$ such that 
\begin{align*}
\dv \Om\cap U_p=R_p(\{(y,\phi_p(y)): y\in U_p\}). 
\end{align*}
\end{Def}
These are the only Lipschitz domains that will be considered in this paper.

\begin{Def}
Let $\Om \subset \R^n$ be a Lipschitz domain. The \emph{tangential and normal trace maps} $\gamma_T:W^{1,2}(\Om,\Lambda)\to L^2(\dv \Om,\Lambda)$ and $\gamma_N: W^{1,2}(\Om,\Lambda)\to L^2(\dv \Om,\Lambda)$ respectively are defined by 
\begin{align*}
\gamma_TF(x)&:=\mathbf{P}_T^{\nu(x)}\circ \gamma F(x),\\
\gamma_NF(x)&:=\mathbf{P}_N^{\nu(x)}\circ \gamma F(x),
\end{align*}
for $\mathscr{H}^{n-1}$-a.e. $x\in \dv \Om$, where $\nu$ is the outward pointing unit normal on $\dv \Om$ and $\gamma$ is the full Sobolev trace map. 
\end{Def}

\begin{rem}
Sometimes we will write $F_T$ and $F_N$ respectively instead of $\gamma_TF$ and $\gamma_NF$. 
\end{rem}

There are other equivalent ways of defining the tangential and normal trace maps. For a comparison between the different definitions we refer to the discussion in \cite[Chap. 3]{Dac} 

\begin{Def}
Let $\Om\subset \R^n$ be a bounded (strongly) Lipschitz domain. The \emph{nontangential maximal function} $u^\ast: \dv \Om \to \R$ of a function $u$, possibly vector valued, is defined by 
\begin{align*}
u^\ast(x):=\sup \{\vert u(y)\vert: y\in \Om,\,\text{ such that }\, \vert x-y\vert\leq 2\,\text{dist}(y,\dv \Om)\}
\end{align*}
\end{Def}

\subsection{\sffamily Hodge-Dirac Operators on Euclidean Domains}

\begin{Def}[Hodge-Dirac operators]
The Hodge-Dirac operators are defined by 
\begin{align}
\Dp=d+\delta,\,\,\,, \Dm=d-\delta,
\end{align}
where $d$ is the exterior derivative and $\delta$ is the interior derivative, i.e., $\delta=-d^\ast$, where $d^\ast$ is the formal adjoint of $d$. 
\end{Def}

The symbols $\sigma_{\Dp}(x,\xi),\sigma_{\Dm}(x,\xi)\in \LL(V,\LL(\Lambda V))$ of $\Dp$ and $\Dm$ respectively is given by the linear maps 
\begin{align*}
\sigma_{\Dp}(x,\xi)w=\xi \gp w, \,\,\, \sigma_{\Dm}(x,\xi)w=\xi \gm w
\end{align*}
for all $w\in \Lambda V$ and $\xi \in V$. In particular, $\sigma_{\Dp}(x,\xi)$ is self-adjoint and is skew-adjoint $\sigma_{\Dp}(x,\xi)$. 

We consider the Hodge-Dirac operators on $\Om$ on a smooth bounded domain, initially defined on $C_0^{\infty}(\Om,\Lambda \R^n)$ and viewed as unbounded operators $\mathcal{D}^\pm:L^2(\Om,\Lambda \R^n)\to L^2(\Om,\Lambda \R^n)$. We recall that the maximal extension $\mathcal{D}_{max}^\pm=((\mathcal{D}^{\pm})^\ast)^\ast$, where $\mathcal{D}^{\pm}$ acts on $C_0^{\infty}(\Om,\Lambda \R^n)$. In particular the domain of $\mathcal{D}_{max}^\pm$ is given by
\begin{align*}
\text{dom}(\mathcal{D}_{max}^\pm)=\{F\in \mathscr{D}'(\Om,\Lambda \R^n): F,\mathcal{D}_{max}^\pm F\in L^2(\Om,\Lambda \R^n) \}. 
\end{align*}

It follows trivially that $W^{1,2}(\Om,\Lambda \R^n)\subset \text{dom}(\mathcal{D}_{max}^\pm)$. On the other hand it is shown in Theorem \cite[Theorem 3.2]{BB} that the trace map $\gamma$ satisfies 
\begin{align*}
\gamma(\text{dom}(\mathcal{D}_{max}^\pm))\cap W^{-1/2,2}(\dv \Om,\Lambda)\neq \varnothing.
\end{align*}
Hence, the inclusion is strict since $\gamma(W^{1,2}(\R^n,\Lambda \R^n))=W^{1/2,1}((\dv \Om,\Lambda)$ by the trace theorem for Sobolev spaces, see \cite{Ding}. 

Finally a direct computation shows that the fundamental solution of $\Dp$ and $\Dm$ are given respectively by $\pm E(x)$, where 
\begin{align*}
E(x)=\frac{1}{\sigma_{n-1}}\frac{x}{\vert x\vert^{n}}.
\end{align*}

When working with Dirac operators as well as the exterior and interior derivative, it is many times useful for computations to introduce the formal nabla symbol
\begin{align*}
\nabla:=\sum_{j=1}^ne_j\dv_j. 
\end{align*}

Using the formal nabla symbol we have for any $F\in C^{\infty}(\Om,\Lambda)$
\begin{align*}
dF(x)=\nabla \wedge F(x),\,\,\, \delta F(x)=\nabla \ri F(x), \,\,\, \Dp F(x)=\nabla \gp F(x), \,\,\, \Dm F(x)=\nabla \gm F(x). 
\end{align*}

In the special case when $F=(f_1,f_2,...,f_n)$ is a vector field, the equation $\Dp F(x)=0$ is equivalent to the system 
\begin{align*}
\left\{
    \begin{array}{ll}
    \sum_{j=1}^n\frac{\dv f_j}{\dv x_j}=0, & \\ \\
        \frac{\dv f_i}{\dv x_j}=\frac{\dv f_j}{\dv x_i},& i\neq j, 
      \end{array} \right.
\end{align*}
used in \cite{SW} to study boundary values of harmonic functions. 

For any domain $\Om \subset \R^n$ one can define an $\Lambda \R^n$ valued dualities $(\cdot,\cdot)_{\gp}$ and $(\cdot,\cdot)_{\gm}$, defined by 
\begin{align*}
(F,G)_\gp:=\int_{\Om}\overline{F(x)}\gp G(x)dx,\,\,\,(F,G)_\gm:=\int_{\Om}\overline{F(x)}\gm G(x)dx.
\end{align*}

In \cite{DB}, it is shown the the Riesz representation theorem extends to such generalized dualities. Moreover, the usual inner product is contained in the generalized duality using the algebraic identity 
\begin{align*}
\langle w_1,w_2\rangle =\langle \overline{w_1}\gp w_2\rangle_0
\end{align*}
for multivectors, which yield $\langle F,G\rangle=\langle(F,G)_\gp\rangle_0$. Using the generalized duality we also have the vector valued integration by parts formula
\begin{align}
\label{partIntDirac}
(G,\Dp F)_\gp=-(\Dp G,F)_\gp+\int_{\dv \Om} \overline{G(y)}\gp \nu(y)\gp F(y)d\sigma(y). 
\end{align}
on a smooth domain $C^1$-domain with outwards point unit normal $\nu$ and any $F,G\in C^1(\overline{\Om},\Lambda)$. In addition $\mathcal{D}^{\pm}$ satisfies the following versions of Stokes' theorem 
\begin{align}\label{eq:StokesD}
\int_{\Om}\Dp F(x)dx=\int_{\dv \Om}\nu(y)\gp F(y)d\sigma(y),\quad \int_{\Om}\Dm F(x)dx=\int_{\dv \Om}\nu(y)\gm F(y)d\sigma(y),
\end{align}
valid for piecewise smooth bounded domains $\Om$ and $F\in W^{1,2}(\Om,\Lambda)$. Finally, we may express the partial derivatives $\dv_jF$ according 
\begin{align*}
\dv_j F(x)=\frac{1}{2}(e_j \gp \Dp F(x)-e_j\gm \Dm F(x)). 
\end{align*}

We now study how the Hodge-Dirac operators transform under the action of the Hodge star maps. 
\begin{Lem}\label{lem:StarD}
If $F\in W^{1,2}(\Om,\Lambda)$ then 
\begin{align*}
\Dp F(x)\star&=\Dm (\widehat{F(x)}\star),\\
\Dm F(x)\star&=\Dp (\widehat{F(x)}\star). 
\end{align*}
\end{Lem}

\begin{proof}
Using the algebraic identities from \cite[Proposition 2.6.8]{R}
\begin{align*}
(\Theta_1\wedge \Theta_2)\star&=\Theta_2\ri (\Theta_1\star),\\
(\Theta_1\li \Theta_2)\star&=\Theta_2\wedge (\Theta_1\star).
\end{align*}
This gives using identities valid for $F\in W^{1,2}(\Om,\Lambda)$ we get
\begin{align*}
(\nabla \wedge F(x))\star&=( \widehat{F(x)}\wedge \nabla)\star=\nabla \ri(\widehat{F(x)}\star),\\
(\nabla \ri F(x))\star&=-( \widehat{F(x)}\li \nabla)\star=-\nabla \wedge (\widehat{F(x)}\star).
\end{align*}
Hence,
\begin{align*}
(\Dm F(x))\star &=(\nabla \wedge F(x))\star-(\nabla \ri F(x))\star=\nabla \ri(\widehat{F(x)}\star)+\nabla \wedge (\widehat{F(x)}\star)=\Dp (\widehat{F(x)}\star)\\
(\Dp F(x))\star &=(\nabla \wedge F(x))\star+(\nabla \ri F(x))\star=\nabla \ri(\widehat{F(x)}\star)-\nabla \wedge (\widehat{F(x)}\star)=\Dm (\widehat{F(x)}\star). 
\end{align*}

\end{proof}

On the other hand in this paper we will mostly be concerned with multivector fields $F$ of the form $F=u+v$, where $u\in W^{1,2}(\Om,\R)$ and $v\in W^{1,2}(\Om,\Lambda^2\R^n)$ where in addition $dv=0$. In this case 
$\overline{F}=u-v$, and 
\begin{align*}
\Dp \overline{F}=du -\delta v=\Dm F,\,\,\, \Dm \overline{F}=du+\delta v=\Dp F,
\end{align*}
which are analogous to the identities $\dbar \overline{f}=\overline{\dv f}$ and $\dv \overline{f}=\overline{\dbar f}$ for complex valued fields in the plane. 

\begin{rem}
When acting on smooth fields $F\in C^\infty(\Om,\Lambda \R^n)$, one easily checks that the Hodge-Dirac operators $\mathcal{D}^\pm$ satisfies the (anti-)commutation relations 
\begin{align*}
\Dp \Dm+\Dm \Dp=0,\quad (\Dp)^2=-(\Dm)^2=\Delta,\quad\mathcal{D}^\pm \Delta=\Delta \mathcal{D}^\pm
\end{align*}
where $\Delta=d\delta+\delta d$ is the Hodge--Laplacian. A variation of the Hodge--Dirac operators was used in E. Witten's paper \cite{Witten} on his supersymmetric proof of the Morse inequalities. He used the ``supersymmetry'' operators 
\begin{align*}
Q_1=\Dm, \quad Q_2=i\Dp
\end{align*}
\end{rem}

\subsection{\sffamily Cauchy-Pompeiu Formula for the Hodge-Dirac Operators}\label{sec:Cauchy}

\begin{Thm}
If $F\in C^1(\Om,\Lambda \R^n)\cap C^0(\overline{\Om},\Lambda \R^n)$ then the Cauchy-Pompieu formula reads
\begin{align}\label{eq:GenCF}
F(x)&=\frac{1}{\sigma_{n-1}}\int_{ \Om}\frac{y-x}{\vert y-x\vert^n}\gp \Dp F(y)dy+\frac{1}{\sigma_{n-1}}\int_{\dv \Om}\frac{y-x}{\vert y-x\vert^n}\gp \nu(y)\gp F(y)d\sigma(y),\\
&=\frac{1}{\sigma_{n-1}}\int_{ \Om}\frac{x-y}{\vert x-y\vert^n}\gm \Dm F(y)dy+\frac{1}{\sigma_{n-1}}\int_{\dv \Om}\frac{x-y}{\vert x-y\vert^n}\gm \nu(y)\gm F(y)d\sigma(y).
\end{align}.  
Moreover, if $\Dp F=0$ in $\Om$ then the Cauchy integral formula 
\begin{align}
F(x)=\int_{\dv \Om}E(y-x)\gp \nu(y)\gp \gamma(F)(y)dy
\end{align}
hold for \emph{every} $x\in \Om$. 
\end{Thm}

For a proof see Theorem 8.1.8 in \cite{R}.

\begin{Def}\label{def:Cint}
The Cauchy integrals $\EE^\pm(f)$ of $f\in L^2(\dv \Om,\Lambda)$ is defined as
\begin{align*}
\EE^+(f)(x)&:=2\text{p.v.}\int_{\dv \Om}E(y-x)\gp \nu(y)\gp f(y)d\sigma(y),\\
\EE^-(f)(x)&:=2\text{p.v.}\int_{\dv \Om}E(x-y)\gm \nu(y)\gm f(y)d\sigma(y)
\end{align*}
\end{Def}

The boundedness of the operators $\EE^\pm$ on $L^2(\dv \Om,\Lambda)$ on Lipschitz hypersurfaces $\dv \Om$ follows from the seminal work in \cite{CMM}. For more details see \cite{GM}.

\begin{Thm}[Sokhotski-Plemelj jump formulas]
Let $f\in L^2(\dv \Om,\Lambda)$. Let $\Om_+=\Om$ and $\Om_-=\R^n\setminus \overline{\Om}$. Then 
\begin{align*}
\lim_{\substack{y\to x\in \dv \Om\\ y\in \Om^\pm}}\int_{\dv \Om}E(y-x)\gp \nu(y)\gp f(y)d\sigma(y)=\frac{1}{2}(I\pm\EE^{+})(f)(x). 
\end{align*}
\end{Thm}
A proof of this fact follows from an easy modification to the vector valued setting of the argument given in \cite{To}.

We have the Hardy space splitting of $L^2(\dv \Om,\Lambda)$ according to 
\begin{align*}
L^2(\dv \Om,\Lambda)=H^2_+(\dv \Om)\oplus H^2_-(\dv \Om). 
\end{align*}

In particular $f\in H^2_+(\dv \Om)$ if and only if 
\begin{align*}
\frac{1}{2}(I+\EE^{+})(f)(x)=f(x)\text{ for a.e. $x\in \dv \Om$},
\end{align*}
that is if and only if $f$ solves the singular integral equation 
\begin{align*}
\frac{1}{\sigma_{n-1}}\text{p.v.}\int_{\dv \Om}\frac{y-x}{\vert y-x\vert^n}\gp \nu(y)\gp f(y)d\sigma(y)=\frac{1}{2}f(x).
\end{align*}


\section{\sffamily Topological preliminaries}

\subsection{\sffamily Harmonic Fields and De Rham Cohomology}

Consider the Hilbert space $L^2(\dv \Om, \Lambda)$ and the linear subspaces $L^2(\dv \Om, \Lambda^k)$ consisting of homogeneous multivector fields of degree $k$. At each point $x\in \dv \Om$, the tangent bundle $T_x\Om$ has an orthogonal decomposition as $T_x\Om=T_x\dv \Om\oplus N_x\dv \Om$. This induces an orthogonal splitting of the exterior tangent bundle at $x$ according to 
\begin{align*}
\Lambda T_x \Om=\Lambda (T_x\dv  \Om\oplus N_x\dv  \Om)\cong \Lambda T_x\dv\Om \widehat{\otimes}\Lambda N_x\dv \Om. 
\end{align*}
Since $ \Lambda N_x\dv \Om=\text{span}\{1,\nu(x)\}$, we see that 
\begin{align*}
\Lambda T_x \Om\cong \Lambda T_x\dv\Om \widehat{\otimes}1\oplus\Lambda T_x\dv\Om \widehat{\otimes}\nu(x). 
\end{align*}
Therefore, any $f\in L^2(\dv \Om, \Lambda)$ can be written uniquely as 
\begin{align*}
f(x)=f_T(x)+f_N(x)=f_T(x)+\nu(x)\wedge g_T(x)
\end{align*}
with $f_T,g_T\in L^2(\dv \Om, \Lambda T\dv \Om)$ and $f_N\in \nu\wedge  L^2(\dv \Om, \Lambda T\dv \Om)$ and where  
\begin{align*}
f_T(x)=\nu(x)\ri (\nu(x)\wedge f(x)), \,\, f_N(x)=\nu(x)\wedge (\nu(x)\ri f(x)). 
\end{align*}

Moreover $f_T$ and $f_N$ satisfies the following identities for $\sigma$-a.e. $x\in \dv \Om$
\begin{align*}
\nu(x)\wedge f(x)=\nu(x)\wedge f_T(x), \,\,\, \nu(x)\ri f(x)=\nu(x)\ri f_N(x). 
\end{align*}

We define 
\begin{align*}
L^2_T(\dv \Om,\Lambda):=L^2(\dv \Om, \sigma,\Lambda T\dv \Om),\,\, L^2_N(\dv \Om,\Lambda):=\nu(x)\wedge L^2(\dv \Om, \sigma,\Lambda T\dv \Om).
\end{align*}
In particular, $f\in L^2_T(\dv \Om,\Lambda)$ if and only if for $\sigma$-a.e. $x$ we have $\nu(x)\ri f(x)=0$ and  $f\in L^2_N(\dv \Om,\Lambda)$ if and only if for $\sigma$-a.e. $x$ we have $\nu(x)\wedge f(x)=0$

\begin{Def}
A $k$-multivector field $\omega \in C^{\infty}(\Om, \Lambda^k\R^n)$ is called an \emph{harmonic field} if $d\omega=\delta \omega=0$ in $\Om$. The space of harmonic $k$-multivector fields in $\Om$ is denoted by $\mathcal{H}^k(\Om)$. A $k$-multivector field $\omega \in C^{\infty}(\overline{\Om}, \Lambda^k\R^n)$ such that $\omega_T=0$ on $\dv \Om$ is called a \emph{tangential harmonic $k$-multivector} and the space of tangential harmonic $k$-multivector fields in $\Om$ is denoted by $\mathcal{H}_T^k(\Om)$. Similarly, A $k$-multivector field $\omega \in C^{\infty}(\overline{\Om}, \Lambda^k\R^n)$ such that $\omega_N=0$ on $\dv \Om$ is called a \emph{normal harmonic $k$-multivector} and the space of tangential harmonic $k$-multivector fields in $\Om$ is denoted by $\mathcal{H}_N^k(\Om)$.
\end{Def}

The Hodge-star map induces an isomorphism $\star \mathcal{H}_T^k(\Om)=\mathcal{H}_N^{n-k}(\Om)$. Moreover we have the de Rham--Hodge isomorphisms 
\begin{align*}
\mathbf{H}^k(\Om)&\cong \mathbf{H}_{dR,abs}^k(\Om)\cong \mathcal{H}_N^{k}(\Om)\\
\mathbf{H}^k_{rel}(\Om)&\cong \mathbf{H}_{dR,rel}^k(\Om)\cong \mathcal{H}_T^{k}(\Om)
\end{align*}

Here, $\mathbf{H}^k(\Om)$ denotes the absolute cohomology groups and $\mathbf{H}^k_{rel}(\Om)$ denotes the relative cohomology groups. These are also denoted by $\mathbf{H}^k_{rel}(\Om)=\mathbf{H}^k(\Om,\dv \Om)$ in the literature. Therefore, normal and tangential boundary conditions are sometimes called absolute and relative boundary conditions in the literature, e.g. \cite{Sw,CDGM}. 

We will also set 
\begin{align*}
\mathcal{H}_T(\Om)&=\bigoplus_{k=0}^n\mathcal{H}_T^k(\Om)=\{\omega\in C^{\infty}(\overline{\Om},\Lambda \R^n): d\omega=\delta \omega=0, \,\, \omega_T=0\},\\
\mathcal{H}_N(\Om)&=\bigoplus_{k=0}^n\mathcal{H}_N^k(\Om)=\{\omega\in C^{\infty}(\overline{\Om},\Lambda \R^n): d\omega=\delta \omega=0, \,\, \omega_N=0\}. 
\end{align*}

\begin{Lem}
\begin{align*}
\mathcal{H}_T(\Om)&=\{F\in C^{\infty}(\overline{\Om},\Lambda \R^n): , \,\, \Dp F=0, F_T=0\}=\{F\in C^{\infty}(\overline{\Om},\Lambda \R^n): , \,\, \Dm F=0, F_T=0\},\\
\mathcal{H}_N(\Om)&=\{F\in C^{\infty}(\overline{\Om},\Lambda \R^n): , \,\, \Dp F=0, F_N=0\}=\{F\in C^{\infty}(\overline{\Om},\Lambda \R^n): , \,\, \Dm F=0, F_N=0\},\\
\end{align*}
\end{Lem}

\begin{proof}
If $\Dp F(x)=0$ in $\Om$, then by integration by parts formula \eqref{eq:IntPart1}
\begin{align*}
0&=\int_{\Om}\langle \Dp F(x),\Dp F(x)\rangle dx=\int_{\Om}\langle (d+\delta) F(x),(d+\delta) F(x)\rangle dx\\
&=\int_{\Om} (\vert d F(x)\vert^2+\vert \delta F(x)\vert^2)dx+2\int_{\Om} \langle d F(x),\delta F(x)\rangle dx\\
&=\int_{\Om} (\vert d F(x)\vert^2+\vert \delta F(x)\vert^2)dx-2\int_{\Om} \langle d^2 F(x), F(x)\rangle dx+2\int_{\dv \Om} \langle  \nu(x)\ri F(x), dF(x)\rangle dx\\
&=\int_{\Om} (\vert d F(x)\vert^2+\vert \delta F(x)\vert^2)dx.
\end{align*}
Consequently, $dF(x)=\delta F(x)=0$. 
\end{proof}

In the case that $\Om\subset \R^n$ is a simply connected domain, so in particular $\Om$ is diffeomorphic to the open unit ball $\mathbb{B}_1(0)=\{x\in \R^n: \vert x\vert <1\}$, then by Poincaré lemma
\begin{align*}
\mathbf{H}^k(\Om)=\left\{
    \begin{array}{ll}
   0 & \text{if $1\leq k\leq n$},\\
        \R & \text{if $k=0$}  
      \end{array} \right.,
\,\,
\mathbf{H}^k_{rel}(\Om)=\left\{
    \begin{array}{ll}
   0 & \text{if $0\leq k\leq n-1$},\\
        \R & \text{if $k=n$}  
      \end{array} \right..
\end{align*}
Moreover one can show that $\mathcal{H}_N(\Om)=\{c:c\in \R\}$ and $\mathcal{H}_T(\Om)=\{ce_{\mathbf{n}}:c\in \R\}$, where $e_{\mathbf{n}}=e_1\wedge e_2\wedge...\wedge e_n$ is the positively oriented volume $n$-field for $\R^n$. 
Thus, the tangential boundary value problem \eqref{eq:HDT1} always has solution whenever $\langle F_T\rangle_n=0$. However, since $\langle F\rangle_n= \langle F_N\rangle_n$, this is always true. Therefore, for a simply connected domain, the tangential boundary value problem for the Hodge-Dirac operator is always solvable. Moreover, since $\langle F\rangle_0=\langle F_T\rangle_0$ it follows that $F_N\perp \mathcal{H}_T(\dv \Om)$ automatically, and hence the solution is unique as well. 

Moreover, by the Poincaré duality we have the isomorphisms between the homology and cohomology groups 
\begin{align*}
\mathbf{H}_k(\Om)\cong \mathbf{H}^{n-k}(\Om,\dv \Om),\,\,\, \mathbf{H}_k(\Om,\dv \Om)\cong \mathbf{H}^{n-k}(\Om). 
\end{align*}

In dimension 3 we have in particular (see \cite[p. 423]{CDG}) 
\begin{align*}
\text{dim}\,\mathbf{H}_0(\Om)&= \text{number of components of $\Om$}\\
\text{dim}\,\mathbf{H}_1(\Om)&=\text{total genus of $\dv \Om$}\\
\text{dim}\,\mathbf{H}_2(\Om)&=\text{number of components of $\dv \Om$}-\text{number of components of $\Om$}\\
\text{dim}\,\mathbf{H}_3(\Om)&=0. 
\end{align*}


\section{\sffamily  Reduction of Second Order Equation to Dirac-Beltrami Equation}


\subsection{\sffamily The Dirichlet Problem}

We now consider the Dirichlet problem \eqref{eq:DirPro} and assume that the endomorphism field $A\in L^\infty(\Om, \LL(\R^n))$ is normal, i.e., $A^\ast(x)A(x)=A(x)A^\ast(x)$ for a.e. $x\in \Om$ and satisfies the uniform ellipticity bounds
\begin{align}\label{eq:EBound1}
\lambda\vert v\vert^2\leq \langle A(x)v,v\rangle,\quad \Vert A(x)\Vert\leq \Lambda\,\, \text{ for a.e. $x$},
\end{align}
or in the case if $A$ is symmetric, i.e., $A^\ast(x)=A(x)$, 
\begin{align*}
\lambda\vert v\vert^2\leq \langle A(x)v,v\rangle \leq \Lambda\vert v\vert^2
\end{align*}
for all $v\in \R^n$ and a.e. $x\in \Om$. Consider weak solution $u\in W^{1,2}(\Om)$ of the equation 
\begin{align}\label{eq:AEq}
\text{div}A(x)\nabla u(x)=0
\end{align}
 in $\Om$. Equation \eqref{eq:AEq} is the Euler-Lagrange equation for the energy functional 
\begin{align*}
\mathcal{E}_A[u]=\int_{\Om}\langle A(x)\nabla u(x),\nabla u(x)\rangle dx, 
\end{align*}
and call
\begin{align*}
\mathscr{E}_A[u]=\langle A(x)\nabla u(x),\nabla u(x)\rangle
\end{align*}
the energy density. The weak Euler-Lagrange equations for $\mathcal{E}_A[\cdot]$ shows that $A(x)\nabla u(x)$ is a weakly divergence free vector field for any weak solution $u$. The integration by parts formula \eqref{eq:IntPart1} shows that the energy density $\mathscr{E}_A[u]=\langle A(x)\nabla u(x),\nabla u(x)\rangle$ is a \emph{null-lagrangian}. Indeed,
\begin{align*}
\mathcal{E}_A[u]&=\int_{\Om}\langle A(x)\nabla u(x),\nabla u(x)\rangle dx=\int_{\dv \Om}\langle u(y)\nu(y),A(y)\nabla u(y)\rangle d\sigma(y). 
\end{align*}
We also notice that \emph{conormal derivative} $\dv_{\nu_A}u(x):=\langle \nu(x), A(y)\nabla u(x)\rangle$ is nothing but the Neumann data for $\text{div}A\nabla u$, and so the energy of the solution can be expressed a the pairing of the Dirichlet and Neumann data. If we let $\Lambda_{DN}$ denote the \emph{Dirichlet to Neumann map}, the energy in terms of the boundary data becomes 
\begin{align*}
\mathcal{E}_A[\phi]=\int_{\dv \Om}\phi(x)\Lambda_{DN}(\phi)(x)d\sigma(x). 
\end{align*}

\subsection{\sffamily Hodge Theory and Gauge Conditions}\label{sec:GaugeCond}

By assumption $\nabla u\in L^2(\Om,\R^n)$. Since $u$ is a weak solution get the div-curl couple $(B,E)$ defined by
\begin{align*}
E(x)=\nabla u(x),\,\,\, B(x)=A(x)\nabla u(x)
\end{align*}
where $dE(x)=0$ and $\delta B(x)=0$ in the sense of distributions. 
We have the $L^2$-estimate 
\begin{align*}
\Vert B\Vert_{L^2}=\int_{\Om}\vert B(x)\vert^2dx=\int_{\Om}\vert A(x)\nabla u(x)\vert^2dx\leq \Vert A\Vert_\infty\Vert \nabla u\Vert_{L^2}\leq \Lambda_E\Vert E\Vert_{L^2}. 
\end{align*}

Since $\delta B(x)=0$ we can always find a $v\in W^{1,2}(U,\Lambda^2\R^n)$ locally in a simply connected open set $U\Subset \Om$ such that $\delta v(x)=B(x)$. Of course such $v$ is not unique and it is natural to try and impose the exact gauge condition $d v(x)=0$ so that $\Dp v=(d+\delta )v=B$ locally. If such condition can be imposed, then this determines a unique $v$ up to harmonic bivector fields, i.e. $h\in W^{1,2}(U,\Lambda^2\R^n)$ such that $dh(x)=\delta h(x)=0$. However, without boundary conditions this space is is still infinite dimensional. A natural choice is to impose either tangential or normal boundary conditions on $v$. If $U$ is simply connected this determines a unique $v$ as Theorem \ref{thm:Hodge} shows. If on the other hand we want to find a \emph{global} bivector field $v$ there are topological restrictions for the existence of such $v$. 

We have the following theorem adapted to the Euclidean setting:
\begin{Thm}[Theorem 3.2.5 in \cite{Sw}]
\label{thm:Hodge}
Let $\Om\subset \R^n$ be a smooth ($C^2$) compact domain. Given $\chi\in W^{s,p}(\Om,\Lambda^{k+1}\R^n)$, $\rho\in W^{s,p}(\Om,\Lambda^{k-1}\R^n)$ and $\psi\in W^{s+1-1/p,p}(\Om,\Lambda^{k}\R^n)\vert_{\dv \Om}$, with $s\geq 0$ and $p>1$ the boundary value problem 
\begin{align*}
\left\{
    \begin{array}{ll}
    d \omega(x)=\chi(x), & x\in \Om,\\
        \delta \omega(x)=\rho(x), & x\in \Om,\\
      \omega_T(x)=\psi_T, & x\in \dv \Om.  
      \end{array} \right.
\end{align*}
is solvable if and only if $\delta \rho=0$ in the sense of distributions and 
\begin{align*}
(\rho,\kappa)=\int_{\Om}\langle\rho(x),\kappa(x) \rangle dx=0, \,\, \forall \, \kappa\in \mathcal{H}_T^{k-1}(\Om),
\end{align*}
and $d\chi=0$ in the sense of distributions, $\chi_T=d\psi_T$, and 
\begin{align*}
(\chi,\lambda)=\int_{\dv \Om}\langle\psi_T(x),\lambda_N(x) \rangle dx, \,\, \forall \, \lambda\in \mathcal{H}_T^{k+1}(\Om),
\end{align*}
The solution is unique up to an arbitrary harmonic field $h\in \mathcal{H}_T^k(\Om)$. Moreover, we have the Sobolev norm estimate
\begin{align}\label{eq:Sobolev}
\Vert \omega\Vert_{W^{s+1,p}(\Om)}\leq C(\Vert \chi\Vert_{W^{s,p}(\Om)}+\Vert \rho\Vert_{W^{s,p}(\Om)}+\Vert \psi\Vert_{W^{s+1-1/p,p}(\dv \Om)}). 
\end{align}
\end{Thm}

In the case of interest for us $k=2$, $s=0$ and $p=2$. We now \emph{make the choice} to set $\chi=0$. Then of course $\chi=0\in L^2(\Om, \Lambda^3\R^n)$, $\rho=B\in  L^2(\Om, \Lambda^1\R^n)$ and $\psi=0\in W^{1/2,2}(\Om, \Lambda^2\R^n)\vert_{\dv\Om}$ and $d\psi=0\in  W^{1,-1/2}(\Om, \Lambda^2\R^n)\vert_{\dv\Om}$. This gives us the boundary value problem for $v$
\begin{align}\label{eq:GaugeCond}
\left\{
    \begin{array}{ll}
    \delta v(x)=B(x), & x\in \Om,\\
        dv(x)=0, & x\in \Om,\\
      v_T(x)=0, & x\in \dv \Om.  
      \end{array} \right.
\end{align}
According to Theorem 3.2.5 this problem has a solution if and only if 

\begin{align*}
0=(\chi,\lambda)=\int_{\dv \Om}\langle\psi_T(x),\lambda_N(x) \rangle dx\,\, \forall \, \lambda\in \mathcal{H}_T^{2}(\Om),
\end{align*}
which is automatically true since $\psi_T=0$ by choice and 
\begin{align*}
(\rho,\kappa)=(B, \kappa)=0,\,\, \forall \, \kappa\in \mathcal{H}_T^{1}(\Om).
\end{align*}

This is not necessarily true unless  $\mathbf{H}_{rel}^1(\Om)\cong\mathcal{H}_T^{1}(\Om)=\{0\}$. We therefore assume that $\mathbf{H}_{rel}^1(\Om)=\{0\}$. This is equivalent to assuming that $\Om$ is simply connected. From now on we therefore assume this. In particular \eqref{eq:GaugeCond} is equivalent to the inhomogenous problem for the Hodge--Dirac operator 
\begin{align}\label{eq:GaugeCond2}
\left\{
    \begin{array}{ll}
    \Dp v(x)=B(x), & x\in \Om,\\
        v_T(x)=0, & x\in \dv \Om.  
      \end{array} \right.
\end{align}
This equation in turn can be reduced to the corresponding homogenous equation by the Cauchy transform (see section \ref{sec:Cauchy}). Applying Theorem 1.7 in \cite{MMMT} for the homogenous Hodge--Dirac equation on Lipschitz domains one can extend the Theorem 3.2.5 in \cite{Sw} to Lipschitz domains. Alternatively one can directly prove Theorem 3.2.5 for Lipschitz domains using Hodge decompositions for multivector fields with zero tangential trace as in Chapter 10.3 Theorem  10.3.1 in \cite{R}. We omit the details.  Note however that in the case of Lipschitz domain the estimate \eqref{eq:Sobolev} need not hold, see section \ref{sec:Elliptic}. Instead, we have the estimate 
\begin{align*}
\Vert v\Vert_{W^{1,2}_{d,\delta}(\Om)}\leq C\Vert B\Vert_{L^2(\Om)}. 
\end{align*}

Hence we have a unique $v$ up to an arbitrary harmonic field in $h\in \mathcal{H}_T^2(\Om)$. However, since $\mathcal{H}_T^2(\Om)=\{0\}$, this $v$ is unique. 
We call $v$ the \emph{tangential $A$-harmonic conjugate} to $u$. The additional differential condition $dv(x)=0$ in \eqref{eq:GaugeCond} can be viewed as imposing a gauge condition on the auxiliary field $v$ much in the same way as one can impose gauge conditions on the electromagnetic potential in electrodynamics. Moreover, by Theorem \ref{thm:Hodge} $v\in W^{1,2}(\Om)$ in the case when $\Om$ is a $C^2$-domain. We now define the multivector field 
\begin{align*}
F(x)=u(x)+v(x)
\end{align*}
which in view of Theorem \ref{thm:Hodge} lies in $W^{1,2}_{d,\delta}(\Om,\Lambda^0\R^n\oplus \Lambda^2\R^n)$. We note that 
\begin{align*}
\overline{F(x)}=u(x)-v(x),
\end{align*}
so that $u(x)=\frac{1}{2}(F(x)+\overline{F(x)})$ and $v(x)=\frac{1}{2}(F(x)-\overline{F(x)})$. Since the anti-involution can be thought of as generalization of complex conjugation we see that we can think of $u$ and $v$ as the``'real'' and ``imaginary'' part of $F$ in analogy with the complex case. Furthermore, $F_T(x)=u(x)=\phi(x)$, $x\in \dv \Om$. 
We now relate the div-curl couple $(B,E)$ to $F$. We have 
\begin{align*}
E(x)+B(x)&=du(x)+\delta v(x)=(d+\delta)u(x)+(d+\delta)v(x)=\Dp F(x),\\
E(x)-B(x)&=du(x)-\delta v(x)=(d-\delta)u(x)+(d-\delta)v(x)=\Dm F(x). 
\end{align*}

And the other hand, the algebraic relation between $E$ and $B$ gives
\begin{align*}
\Dp F(x)=E(x)+B(x)&=(I+A(x))E(x)=(I+A(x)^{-1})B(x),\\
\Dm F(x)=E(x)-B(x)&=(I-A(x))E(x)=-(I-A(x)^{-1})B(x).
\end{align*}
Hence, 
\begin{align*}
\Dm F(x)&=(I-A(x))\circ (I+A(x))^{-1}\Dp F(x)\\
\end{align*}
We set $\mathcal{M}(x):=(I-A(x))\circ (I+A(x))^{-1}$, the \emph{Cayley transform} of $A$. Then we get the \emph{Beltrami equation}
\begin{align*}
\Dm F(x)=\mathcal{M}(x)\Dp F(x). 
\end{align*}

Define 
\begin{align}
\label{def:DistK}
\mathscr{M}(x):=\Vert \mathcal{M}(x)\Vert. 
\end{align}
Then ellipticity bound \eqref{eq:NormalBounds} together with Lemma \ref{lem:EllipticityBound} implies that 
\begin{align}\label{def:M}
M:=\Vert \mathscr{M}(x)\Vert_{L^\infty(\Om)}<1. 
\end{align}

If in addition $\Lambda=1$ and we set $L:=\lambda^{-1}$, then $M=(L-1)/(L+1)$. If we furthermore set $\mathcal{K}:=2(1+M^2)/(1-M^2)$, then $\mathcal{K}=(L^2+1)/(4L)$ and any solution $F$ satisfies the distortion inequality 
\begin{align*}
\vert dF(x)\vert^2+\vert \delta F\vert^2\leq \mathcal{K}\langle dF(x),\delta F(x)\rangle. 
\end{align*}

Thus $F\in W^{1,2}_{d,\delta}(\Om,\Lambda^0\R^n\oplus \Lambda^2\R^n)$ solves the half-Dirichlet boundary value problem 
\begin{align*}
\left\{
    \begin{array}{ll}
    \Dm F(x)=\mathcal{M}(x)\Dp F(x)  & x\in \Om,\\
        F_T(x)=\phi(x), & x\in \dv \Om.  
      \end{array} \right.
\end{align*} 

Here we see the analogy with classical linear Beltrami equation in the plane $\overline{\dv}f(z)=\mathcal{M}(z)\overline{\dv f(z)}$, and we can think of the pair of Hodge-Dirac operators $(\Dm,\Dp)$ as a generalization of the pair of complex derivatives $(\overline{\dv},\dv)$.

\subsection{\sffamily Converse Direction}\label{subsec:Conv}

Assume now that we have a multivector field $F\in W^{1,2}_{d,\delta}(\Om, \Lambda^0\oplus \Lambda^2)$ with tangential boundary value $F_T=\phi$. It then follows that $\Dp F,\Dm F\in L^2(\Om, \Lambda^1\oplus \Lambda^3)$. Thus in particular $\langle \mathcal{D}^\pm \rangle_3$ need not be zero. However, in the Beltrami equation $\Dm F(x)=\mathcal{M}(x)\Dp F(x)$, the measurable field $\mathcal{M}\in L^{\infty}(\Om, \LL(\R^n))$, so $\mathcal{M}(x)$ is a linear operator on vectors, and so the action on $\mathcal{M}$ on tri-vectors are not a priori defined and the equation $\Dm F(x)=\mathcal{M}(x)\Dp F(x)$ does not make sense unless $\mathcal{M}$ is extended as a measurable field 
$\widetilde{\mathcal{M}}\in L^{\infty}(\Om, \LL(\Lambda \R^n))$. The most natural extension is to take $\widetilde{\mathcal{M}}=\widehat{\mathcal{M}}$, the exterior extension of $\mathcal{M}$. In particular, the exterior extension is grade preserving, so $\widehat{\mathcal{M}}(\Lambda^p\R^n)\subset \Lambda^p\R^n$, and if $\mathcal{M}(x)$ is invertible, so is $\widehat{\mathcal{M}}(x)$ with inverse $(\widehat{\mathcal{M}}(x))^{-1}=\widehat{\mathcal{M}^{-1}}(x)$. 
Moreover, by Proposition \ref{prop:CayleyExt} in the Appendix, if $\Vert \mathcal{M}\Vert<1$ then $\Vert \widehat{\mathcal{M}}\Vert\leq \Vert \mathcal{M}\Vert<1$.  

In order not to overburden the notation, we will keep writing $\mathcal{M}(x)$ instead of $\widehat{\mathcal{M}}(x)$. With this in mind the Beltrami equation $\Dm F(x)=\mathcal{M}(x)\Dp F(x)$ is well defined. Splitting $F=\langle F(x)\rangle_0+\langle F(x)\rangle_2=u(x)+v(x)$, we find that 
\begin{align*}
\Dp F(x)=d u(x)+\delta v(x)+d v(x),\,\, \Dm F(x)=d u(x)-\delta v(x)+d v(x).
\end{align*}
Since $\mathcal{M}$ is grade preserving by construction, the Beltrami equation $\Dm F(x)=\mathcal{M}(x)\Dp F(x)$ reduces to the system
\begin{align}
\label{eq:Conv1}
d u(x)-\delta v(x)&=\mathcal{M}(x)(d u(x)+\delta v(x)),\\
\label{eq:Conv2}
dv(x)&=\mathcal{M}(x)dv(x). 
\end{align}
Assuming that $M<1$ so that the Dirac-Beltrami operator is elliptic, this implies 
\begin{align*}
\vert dv(x)\vert \leq M\vert dv(x)\vert<\vert dv(x)\vert
\end{align*}
which implies that $dv (x)=0$. Hence $\mathcal{D}^\pm F(x)$ are vector fields $v$ is a closed bivector field. Solving for $\delta v$ in equation \eqref{eq:Conv1} gives
\begin{align*}
\delta v(x)=(I+\mathcal{M}(x))^{-1}\circ(I-\mathcal{M}(x))(du(x))=A(x)du(x).
\end{align*}
Since $\delta^2 v(x)=0$, it follows that $u$ solves the scalar second order (elliptic) equation 
\begin{align*}
\delta A(x)du(x)=0. 
\end{align*}
Thus, the scalar part of $F$ with $F$ as above solves a linear $A$-harmonic equation. On the other hand we may from \eqref{eq:Conv1} solve for $d u(x)$ instead giving 
\begin{align*}
d u(x)&=(I-\mathcal{M}(x))^{-1}\circ (I+\mathcal{M}(x))(\delta v(x))\\
&=A(x)^{-1}\delta v(x). 
\end{align*}
Since $d^2 u(x)=0$ in the sense of distributions it follows that $v$ solves the equation 
\begin{align}\label{eq:Adual}
d A(x)^{-1}\delta v(x)=0.
\end{align}
We call \eqref{eq:Adual} the \emph{A-harmonic conjugate equation} of $\text{div} A(x)\nabla u(x)=0$. 

To conclude, we have proven the following theorem:

\begin{Thm}\label{thm:conv}
Let $\Om\subset \R^n$ be a simply connected Lipschitz domain. Let $F\in W^{1,2}_{d,\delta}(\Om,\Lambda^{(0,2)})$ be the unique solution to the half-Dirichlet problem
\begin{align*}
\left\{
    \begin{array}{ll}
    \Dm F(x)=\widehat{\mathcal{M}}(x)\Dp F(x)  & x\in \Om,\\
        F_T(x)=\phi(x), & x\in \dv \Om.  
      \end{array} \right.
\end{align*} 
where $\widehat{\mathcal{M}}(x)$ is the exterior extension of a some $\mathcal{M}\in L^\infty(\Om, \LL(\R^n))$ satisfying $\Vert \Vert \mathcal{M}(x)\Vert\Vert_{L^\infty}=m<1$. 
Then $u=\langle F\rangle_0\in W^{1,2}(\Om)$ is the unique weak solution of the Dirichlet problem 
\begin{align*}
\left\{
    \begin{array}{ll}
    \text{div}\,A(x)\nabla u(x)=0  & x\in \Om,\\
        u(x)=\phi(x), & x\in \dv \Om.  
      \end{array} \right.,
\end{align*} 
where $A(x)=(I-\mathcal{M}(x))\circ (I+\mathcal{M}(x))^{-1}$. 
\end{Thm}

\begin{rem}
In the special case when $\mathcal{M}(x)=\mu(x)I_{\R^n}$, for some scalar function $\mu$, we may well take the extension $\widetilde{\mathcal{M}}(x)=\mu(x)I_{\Lambda \R^n}$ instead of $\widehat{\mathcal{M}}(x)$. 
\end{rem}

Finally, we recall Theorem 1 in \cite{IS}. Define the scalar function 
\begin{align*}
\lambda(x):=\left\{
    \begin{array}{ll}
    \displaystyle \frac{\vert \Dm F(x)\vert}{\vert \Dp F(x)\vert}  & \text{if $\Dp F(x)\neq 0$}\\
        0 & \text{if $\Dp F(x)= 0$}
      \end{array} \right.,
\end{align*} 
and the multivector field (being either a unit vector or 1)
\begin{align*}
a(x):=\left\{
    \begin{array}{ll}
     \displaystyle \frac{\Dm F(x)-\lambda(x)\Dp F(x)}{\vert \Dm F(x)-\lambda(x)\Dp F(x) \vert}  & \text{if $\Dm F(x)-\lambda(x)\Dp F(x)\neq 0$}\\
        1 & \text{if $\Dm F(x)-\lambda(x)\Dp F(x)=0$}
      \end{array} \right. .
\end{align*} 
Then as $\Dp F(x)$ and $\Dm F(x)$ are vectors and by Theorem 1 in \cite{IS} and Lemma \ref{lem:Refl} this can be written as 
\begin{align*}
\Dm F(x)=-\lambda(x)a(x)\gp \Dp F(x) \gp a(x)^{-1}.
\end{align*}
In particular, in view of Theorem 1 in \cite{IS} we may always assume that $\mathcal{M}(x)$ is a conformal linear map. Moreover, if we define the conformal map $\mathcal{C}(x)\in \LL(\R^n)$ through 
\begin{align*}
\mathcal{C}(x)v=-\lambda(x)a(x)\gp v \gp a(x)^{-1}
\end{align*}
for all $v\in \R^n$, then 
\begin{align*}
\Vert \mathcal{C}(x)\Vert=\lambda(x)=\frac{\vert \Dm F(x)\vert}{\vert \Dp F(x)\vert}\leq \frac{\Vert  \mathcal{M}(x)\Vert\vert \Dp F(x)\vert}{\vert \Dp F(x)\vert}=\Vert \mathcal{M}(x)\Vert.
\end{align*}
and we see that the ellipticity bounds are preserved.

\subsection{\sffamily Comparison to the two dimensional case}

In dimension two we recall the algebraic isomorphism $\Delta^{ev}\R^2\cong \Lambda^0\R^2\oplus \Lambda^2\R^2\cong \C$ of the even subalgebra of the euclidean Clifford algebra $\Delta \R^2$ with the complex numbers. In particular we can identify the imaginary unit $i$ with the euclidean volume form $e_{12}$. Moreover, the gauge condition $dv(x)=0$ for the $A$-harmonic conjugate of $u$ is \emph{automatically} satisfied and therefore vacuous in dimension two, contrary to higher dimensions. Moreover, the Hodge star map can be expressed using Clifford algebra by 
\begin{align*}
\star w=e_{12}\gp w
\end{align*}
for any multivector $w\in \Lambda V$. Hence, using $v=\star(v\star)$  and setting $w=v\star$, we have 
\begin{align*}
F=u+v=u+\star(v\star)=u+e_{12}w\cong u+iw,
\end{align*}
where $w$ is a scalar valued field. Hence, the usual stream function $v$ in complex notation can be identified with $v\star$ in the framework of this paper. 

\subsection{\sffamily Multiply connected domains and inhomogeneous equations}\label{sec:MultiD}

Consider the Dirichlet problem for the Laplace equation 
\begin{align}\label{eq:DirProLaplace}
\left\{
    \begin{array}{ll}
     \Delta u(z)=0 & z\in \Om\\
      u(z)=\phi(z) & z\in \dv \Om 
      \end{array} \right.
\end{align}
on a bounded smooth domain $\Om \subset \C$ in the plane. If $\Om$ is simply connected it is well-known that $u$ is the real of a holomorphic function $f$, i.e., $f=u+iv$, where $v$ is the harmonic conjugate of $v$. In this case the Dirichlet problem for Laplace equation becomes equivalent to the boundary value problem 
\begin{align}\label{eq:Holo1}
\left\{
    \begin{array}{ll}
     \overline{\dv} f(z)=0 & z\in \Om\\
      \mathfrak{R}e[f(z)]=\phi(z) & z\in \dv \Om 
      \end{array} \right.
\end{align}
for holomorphic functions. If on the other hand $\Om$ is not simply connected there need not exist a global harmonic conjugate $v$ of $u$, the two boundary value problems are not equivalent. If one would like to avoid the use of multivalued holomorphic functions in multiply connected domains there exist a nice substitute called the logarithmic conjugation theorem.
\begin{Thm}\label{thm:ComplexLog}
Let $\Om$ be a finitely connected domain with $K_1,K_2,...,K_n$ denoting the bounded components of the complement of $\Om$. For each $j$, let $a_j$ be in $K_j$. If $u$ is a real valued harmonic function on $\Om$, then there exists a holomorphic function $f$ on $\Om$ and real numbers $c_1,c_2,...,c_n$ such that 
\begin{align*}
u(z)=\mathfrak{R}e[f(z)]+c_1\log\vert z-a_1\vert+c_2\log\vert z-a_2\vert+...+c_n\log\vert z-a_n\vert
\end{align*} 
for every $z\in \Om$. 
\end{Thm}

For a proof see \cite{Ax}. Using this theorem we see that the Dirichlet problem for Laplace's equation in a multiply connected domain is equivalent to boundary value problem 

\begin{align}\label{eq:Holo2}
\left\{
    \begin{array}{ll}
     \overline{\dv} f(z)=0 & z\in \Om\\
      \mathfrak{R}e[f(z)]=\phi(z)-\sum_{j=1}^nc_j\log\vert z-a_j\vert, & z\in \dv \Om. 
      \end{array} \right.
\end{align}
Here the constants $c_j$ are determined by the period conditions. The usefulness of Theorem \ref{thm:ComplexLog} is that it reduces the boundary value problem for Laplace equation on a multiply connected domain to a boundary value problem for the Cauchy-Riemann operator. Unfortunately, the proof of Theorem \ref{thm:ComplexLog} does not generalize to higher dimension. Instead we will proceed by a different route, again using Hodge theory. 

Thus let $\Om \subset \R^n$ be a finitely connected Lipschitz domain and we consider the Dirichlet problem \eqref{eq:DirPro} on $\Om$. Then $\mathcal{H}_T^{1}(\Om)$ is finite dimensional. Moreover we have the unique orthogonal splitting
\begin{align*}
A\nabla u=B=B_T^\perp+\omega_T,
\end{align*}
where $B_T^\perp\in \mathcal{H}_T^{1}(\Om)^\perp$ and $\omega_T\in \mathcal{H}_T^{1}(\Om)$. We can now apply Theorem \ref{thm:Hodge}, to get a bivector field $v$ such that
\begin{align*}
\left\{
    \begin{array}{ll}
    d v(x)=0, & x\in \Om,\\
        \delta v(x)=B_T^\perp(x), & x\in \Om,\\
      v_T(x)=0 & x\in \dv \Om.  
      \end{array} \right.
\end{align*}

This $v$ is unique up to bivectors in $ \mathcal{H}_T^{2}(\Om)$. We define 
\begin{align*}
F(x)=u(x)+v(x).
\end{align*}
Then 
\begin{align*}
\Dp F(x)&=du(x)+\delta v(x), \quad \Dm F(x)=du(x)-\delta v(x). 
\end{align*}
On the other hand we consider the div-curl pair $(B,E)=(A(x)\nabla u(x),\nabla u(x))$. We have 
\begin{align*}
E(x)\pm B(x)=\nabla u(x)\pm A(x)\nabla u(x)=du(x)\pm \delta v(x)\pm \omega_T(x)=\mathcal{D}^\pm F(x)\pm \omega_T(x). 
\end{align*}

Thus, 
\begin{align*}
\mathcal{D}^+ F(x)+ \omega_T(x)=(I+A(x))E(x), \quad \mathcal{D}^-F(x)-\omega_T(x)=(I-A(x))E(x).
\end{align*}

This finally gives 
\begin{align*}
\mathcal{D}^-F(x)-\omega_T(x)&=(I-A(x))\circ(I+A(x))^{-1}(\mathcal{D}^+ F(x)+ \omega_T(x))\\
&\Longleftrightarrow\\
\mathcal{D}^-F(x)-M(x)\mathcal{D}^+ F(x)&=(I+M(x))\omega_T(x).  
\end{align*}

This leads to the boundary value problem

\begin{align}\label{eq:MultiDom}
\left\{
    \begin{array}{cl}
    \mathcal{D}^-F(x)-M(x)\mathcal{D}^+ F(x)=(I+M(x))\omega_T(x), & x\in \Om,\\
       F_T=\phi(x), & x\in \dv \Om.  
      \end{array} \right.
\end{align} 

Now $\omega_T$ can be determined form the orthogonality criterion, namely that $\delta v\perp \mathcal{H}_T^1(\Om)$. Since $\delta v=\delta F$ this gives us the finite linear equation system 
\begin{align*}
\int_{\Om}\langle \delta F(x),\kappa(x) \rangle dx=0, \,\, \forall \, \kappa\in \mathcal{H}_T^{1}(\Om)
\end{align*}
which uniquely determines $\omega_T$. Solutions of \eqref{eq:MultiDom} are unique up to elements in $ \mathcal{H}_T^2(\Om)$. By imposing the orthogonality condition that $F\perp \mathcal{H}_T^2(\Om)$ gives us a unique solution. To conclude we have proven 
\begin{Thm}\label{thm:MultiDB}
Let $\Om\subset \R^n$ be a bounded finitely connected domain. To every solution of \eqref{eq:DirPro} on $\Om$ there exists a unique solution $F\in W^{1,2}_{d,\delta}(\Om,\Lambda^{(0,2)})$ of the boundary value problem 
\begin{align}\label{eq:MultiDB}
\left\{
    \begin{array}{ll}
    \Dm F(x)=\mathcal{M}(x)\Dp F(x)+(I+\mathcal{M}(x))\omega_T&  x\in \Om,\\
    F_T=\phi& x\in \dv \Om\\
    F\perp \mathcal{H}_T^2(\Om)\\ 
    \omega_T\in \mathcal{H}_T^1(\Om)\\ 
    \int_{\Om}\langle \delta F(x),\kappa(x) \rangle dx=0, \,\, \forall \, \kappa\in \mathcal{H}_T^{1}(\Om)\\
    \end{array} \right. .
\end{align}
\end{Thm}

We now consider the inhomogeneous Dirichlet problem 

\begin{align}\label{eq:DirProIn}
\left\{
    \begin{array}{ll}
     \text{div} A(x)\nabla u(x)=\text{div}\,G(x), & x\in \Om,\\
      u(x)=\phi(x), & x\in \dv \Om.  
      \end{array} \right.
\end{align}
on a simply connected Lipschitz domain $\Om$ and vector field $G\in L^2(\Om,\R^n)$. 

We now use Theorem \ref{thm:Hodge} to define the bivector field $v$ as the unique solution of 

\begin{align*}
\left\{
    \begin{array}{ll}
     dv(x)=0,&  x\in \Om,\\
     \delta v(x)=\mathcal{A}(x)\nabla u(x)-G(x)&, x\in \Om\\
      v_T(x)=0, & x\in \dv \Om
      \end{array} \right. .
\end{align*}

We now define
\begin{align*}
F(x)=u(x)+v(x)
\end{align*}
Then 
\begin{align*}
\Dp F(x)&=du(x)+\delta v(x)=du(x)+\mathcal{A}(x)du(x)-G(x),\\
\Dm F(x)&=du(x)-\delta v(x)=du(x)-\mathcal{A}(x)du(x)+G(x).
\end{align*}

Hence,
\begin{align*}
du(x)=(I+\mathcal{A}(x))(\Dp F(x)+G(x))
\end{align*}
which gives 
\begin{align*}
\Dm F(x)&=(I-\mathcal{A}(x))du(x)+G(x)=(I-\mathcal{A}(x))(I-\mathcal{A}(x))(\Dp F(x)+G(x))+G(x)\\
&=\mathcal{M}(x)\Dp F(x)+(I+\mathcal{M}(x))G(x)
\end{align*}

In addition $F_T=\phi$. Thus we have proven the following theorem.

\begin{Thm}\label{thm:DBIn}
Let $\Om \subset \R^n$ be a bounded simply connected Lipschitz domain. To each solution of \eqref{eq:DirProIn} there exists a unique solution 
$F\in W^{1,2}_{d,\delta}(\Om,\Lambda^{(0,2)})$ of the boundary value problem 
\begin{align}\label{eq:DBIn}
\left\{
    \begin{array}{ll}
    \Dm F(x)=\mathcal{M}(x)\Dp F(x)+(I+\mathcal{M}(x))G(x) &  x\in \Om,\\
     F_T(x)=\phi, & x\in \dv \Om
      \end{array} \right. .
\end{align}
for every $G\in L^{2}(\Om,\R^n)$. 
\end{Thm}

\subsection{\sffamily Nonlinear uniformly elliptic equations}\label{subsec:NonLinDir}

\begin{Def}\label{def:Afield}
Let $\A: \Om\times \R^n \to \R^n$ be a mapping that satisfies the conditions:
\begin{itemize}
\item[(i)] $\A: x\mapsto \A(x,\xi)$ is measurable for every $\xi \in \R^n$.
\item[(ii)] $\A: \xi \mapsto \A(x,\xi)$ is continuous for almost every $x\in \Om$.
\item[(iii)] $\A(x,0)=0$ for all $x\in \Om$.
\item[(iv)] $\A$ satisfies the uniform ellipticity bound 
\begin{align*}
\vert \xi-\zeta\vert^2+\vert \A(x,\xi)-\A(x,\zeta)\vert^2\leq K\langle  \A(x,\xi)- \A(x,\zeta),\xi-\zeta\rangle \text{ for a.e. $x\in \Om$ and all $\xi,\zeta\in \R^n$}
\end{align*}
for some $K\geq 2$. 
\end{itemize}
\end{Def}
We want to solve Dirichlet problem for the nonlinear uniformly elliptic equation 
\begin{align}\label{eq:NonlinD}
\left\{
    \begin{array}{ll}
\text{div}\,\A(x,\nabla u(x))=0 &  x\in \Om,\\
          u(x)=\phi(x) & x\in \dv \Om
      \end{array} \right. .
\end{align}
with structure field $\A$ as above for $u\in W^{1,2}(\Om)$ and $\phi\in W^{1/2,2}(\dv \Om)$.

\begin{Def}\label{def:NonlinC}
Define the Cayley transform 
\begin{align*}
\mathcal{M}(x,\xi):=(I-\A(x,\xi))\circ (I+\A(x,\xi))^{-1}
\end{align*}
for each fixed $x$. 
\end{Def}
By Lemma \ref{lem:NLC1} $\mathcal{M}(x,\cdot): \R^n \to \R^n$ is well defined for a.e. $x$. Moreover, $\mathcal{M}$ is measurable in $x$ and by Lemma \ref{lem:NLC2} $\mathcal{M}$ is $\sqrt{\frac{K-2}{K+2}}$-Lipschitz continuous in $\xi$ for a.e. $x$.
We can now argue precisely as in the linear case. If $\Om$ is simply connected we define the bivector field $v$ to be a solution of 
\begin{align*}
dv(x)=0,\,\,\, \delta v(x)=\A(x,\nabla u(x)), \,\,\, v_T=0, 
\end{align*}
and if $\Om$ is multiply connected we do suitable modification as in subsection \ref{sec:MultiD}. Then an identical argument proves that the multivector field $F=u+v$ solves the boundary value problem 
\begin{align}\label{eq:NonlinDB}
\left\{
    \begin{array}{ll}
\Dm F(x)=\mathcal{M}(x,\Dp F(x)) &  x\in \Om,\\
          F_T(x)=\phi(x) & x\in \dv \Om
      \end{array} \right. .
\end{align}

For the converse direction we note that $\Dp F:\Om \to \Lambda^1\R^n\oplus \Lambda^3\R^n$. Thus $\mathcal{M}(x,\Dp F(x))$ is not a priori well-defined. We therefore extend $\mathcal{M}:\Om \times \R^n \to \R^n$ to a mapping $\overline{\mathcal{M}}:\Om \times \Lambda\R^n \to \Lambda\R^n$ by defining 
\begin{align*}
\overline{\mathcal{M}}(x,w)=\mathcal{M}(x,\langle w\rangle_1). 
\end{align*}
It the follows that the extended map is still $\sqrt{\frac{K-2}{K+2}}$-Lipschitz. Indeed,
\begin{align*}
\vert \overline{\mathcal{M}}(x,w_1)-\overline{\mathcal{M}}(x,w_2)\vert&=\vert \mathcal{M}(x,\langle w_1\rangle_1)-\mathcal{M}(x,\langle w_2\rangle_1)\vert\\
&\leq \sqrt{\frac{K-2}{K+2}}\vert \langle w_1\rangle_1-\langle w_2\rangle_1\vert=\sqrt{\frac{K-2}{K+2}}\vert \langle w_1-w_2\rangle_1\vert\\
&\leq \sqrt{\frac{K-2}{K+2}}\vert  w_1-w_2\vert. 
\end{align*}
By abuse of notation we will write $\mathcal{M}$ instead of $\overline{\mathcal{M}}$. An identical argument as in section \ref{subsec:Conv}, shows $u=\frac{1}{2}(F+\overline{F})$ solves \eqref{eq:NonlinD}. We conclude this section by mentioning that by rewriting the second order equation \eqref{eq:NonlinD} as a nonlinear Beltrami equation as in \eqref{eq:NonlinDB} allowed the authors of \cite{ACFJK} to prove an improved Hölder regularity for solutions of \eqref{eq:NonlinD}.


\subsection{\sffamily Ellipticity of the Dirac--Beltrami equation and Null-Lagrangian}\label{sec:Elliptic}

We now consider the ellipticity of Dirac--Beltrami operator $\mathcal{P}_\mathcal{M}=\Dm-\mathcal{M}(x)\Dp$. We will consider the more general case when $\mathcal{P}_\mathcal{M}$ acts on general multivector fields $F: \Om \to \Lambda \R^n$ rather than just special case when $F: \Om \to \R\oplus \Lambda^2\R^n$ considered in the previous section. Following \cite{BB}, the symbol $\sigma_{\mathcal{M}}(x,\xi):=\sigma_{\mathcal{P}_\mathcal{M}}(x,\xi)$ acts on multivectors $w\in \Lambda \R^n$ according to 
\begin{align*}
\sigma_\mathcal{M}(x,\xi)w=\xi\gm w-\mathcal{M}(x)\xi\gp w,
\end{align*}
where $x,\xi\in \R^n$ and where $\mathcal{M}\in L^{\infty}(\Om,\LL(\Lambda \R^n))$ such that $\Vert \Vert \mathcal{M}\Vert\Vert_{L^\infty(\Om)}<1$. Assume that for some $\xi\neq 0$ there exists an $w\neq 0$ such that $\sigma_\mathcal{M}(x,\xi)w=0$. Then 
\begin{align*}
\xi \gm w=\mathcal{M}(x)\xi \gp w.
\end{align*}
Hence, 
\begin{align*}
\vert \xi \gm w\vert=\vert \xi\vert \vert w\vert \leq \Vert \mathcal{M}(x)\Vert \xi\vert\vert w\vert\Longrightarrow  \Vert \mathcal{M}(x)\Vert\geq 1. 
\end{align*}

Thus, $\mathcal{P}_\mathcal{M}$ is elliptic on a domain $U\subset \R^n$ whenever $\sup_{x\in U} \Vert \mathcal{M}(x)\Vert<1$. By Lemma \ref{lem:AccOp} in the appendix, this is equivalent to $A(x)$ being strictly positive, i.e., 
\begin{align*}
\langle A(x)\xi,\xi\rangle>0
\end{align*}
for all $\xi\in \R^n$ and all $x\in U$.

The Beltrami operator $\mathcal{P}_\mathcal{M}$ can be naturally defined on  $F\in \text{dom}_{max}(\Dp)\cap\text{dom}_{max}(\Dm)$, where $\text{dom}_{max}$ is the maximal domain of  
an unbounded operator, (e.g. \cite{BB}). In particular,
\begin{align*}
W_{d,\delta}^{1,2}(\Om,\Lambda)&=\text{dom}_{max}(\Dp)\cap\text{dom}_{max}(\Dm)\\&=\{F\in \mathscr{D}'(\Om,\Lambda \R^n): F, \Dp F,\Dm F\in L^2(\Om,\Lambda \R^n)\}\\
&=\{F\in \mathscr{D}'(\Om,\Lambda \R^n): F, d F,\delta F\in L^2(\Om,\Lambda \R^n)\}. 
\end{align*} 
Without any assumptions on traces of $F$, one has for Lipschitz domains in general the strict inclusion $W^{1,2}(\Om,\Lambda)\subset W_{d,\delta}^{1,2}(\Om,\Lambda)$. This should be contrasted with the two dimensional case. Indeed, by the identity 
\begin{align*}
\Vert Df(z)\Vert=\vert f_{\bar{z}}(z)\vert+\vert f_z(z)\vert
\end{align*}
for a mapping $f:\Om \to \C$ it follows that 
\begin{align*}
W^{1,2}(\Om,\C)=\{f\in \mathscr{D}'(\Om,\C): f, f_{\bar{z}},f_z\in L^{2}(\Om,\C)\}. 
\end{align*}
On the other hand we have for any $F\in W^{1,2}(\Om,\Lambda)$ with $\gamma_TF=0$ and $\Om$ a bounded $C^2$-domain then by Gaffney's inequality with traces (see Corollary 2.1.6 in \cite{Sw}) there exists a constant $C>0$ depending only $\Om$ such that
\begin{align*}
\Vert F\Vert_{W^{1,2}(\Om,\Lambda)}^2\leq C\int_{\Om}(\vert dF(x)\vert^2+\vert \delta F(x)\vert^2+\vert F(x)\vert^2)dx
\end{align*}
Thus, if $F\in W_{d,\delta}^{1,2}(\Om,\Lambda)$ and $\gamma_TF\in W^{1/2,2}(\dv \Om,\Lambda)$, we can write $F=F^T_0+\Phi$, with $\Phi\in W^{1,2}(\Om,\Lambda)$ and $F^T_0\in  W_{d,\delta}^{1,2}(\Om,\Lambda)$ with $\gamma_TF^T_0=0$. Hence, by the Gaffney inequality 
\begin{align*}
W^{1,2}(\Om,\Lambda)=\{F\in W_{d,\delta}^{1,2}(\Om,\Lambda): \gamma_TF\in W^{1/2,2}(\dv \Om,\Lambda)\}. 
\end{align*}
If however $\Om$ is only a Lipschitz domain we have in general only inclusion, see example 10.3.2 in \cite{R}.

We recall that the ellipticity condition \eqref{eq:NormalBounds} by Lemma \ref{lem:EllipticityBound} implies the distortion inequality

\begin{align*}
\vert E(x)\vert^2+\vert B(x)\vert^2\leq  \mathcal{K}\langle E(x),B(x)\rangle. 
\end{align*}
for the div-curl couple $(B,E)=(\delta F,d F)$ and some $\mathcal{K}\geq 2$. 
In \cite{IS}, $\langle dF(x),\delta F(x)\rangle$ is called the Jacobian of the div-curl pair $(B,E)$ and denoted by $J(x,\mathcal{F})$. However, since $\mathscr{E}_A[u](x)=\langle dF(x),\delta F(x)\rangle$ and to avoid confusion, we will instead refer to the \emph{null-Lagrangian}, e.g., \cite{BCO, IL},  as the the energy density of the field $F$, and write $\mathscr{E}(x,F)$. Here we recall that $\langle dF(x),\delta F(x)\rangle$ is a null-Lagrangian as 
\begin{align*}
\int_{\Om} \langle dF(x),\delta F(x)\rangle dx=0
\end{align*}
for all $F\in C^\infty_0(\Om,\Lambda \R^n)$. Using that $2(\vert dF(x)\vert^2+\vert \delta F(x)\vert^2)=\vert \Dp F(x)\vert^2+\vert \Dm F(x)\vert^2$
and $4\langle dF(x),\delta F(x)\rangle=\vert \Dp F(x)\vert^2-\vert \Dm F(x)\vert^2$, the distortion inequality \eqref{eq:Distort} can equivalently be expressed as 
\begin{align}\label{eq:Distort}
\vert \Dp F(x)\vert^2+\vert \Dm F(x)\vert^2\leq \frac{1}{2}\mathcal{K}(\vert \Dp F(x)\vert^2-\vert \Dm F(x)\vert^2). 
\end{align}
We note that in dimension 2, using that $\Vert DF(x)\Vert^2=\vert \dv F(x)\vert^2+\vert \overline{\dv} F(x)\vert^2$ and $ \det(DF(x))=J(x,F)=\vert \dv F(x)\vert^2-\vert \overline{\dv} F(x)\vert^2$ the inequality \eqref{eq:Distort} can be written as
\begin{align*}
\Vert DF(x)\Vert^2\leq \mathcal{K}J(F,x),
\end{align*}
which is the classical differential inequality characterizing $\mathcal{K}$-quasiregular mappings in the plane. 
Finally we note that in the case when $\mathcal{M}(x)=\mathcal{M}$ is constant in the domain $\Om$, then for any smooth $U\Subset \Om$ it follows form \eqref{eq:StokesD} that the following version of Stokes' theorem holds.
\begin{align}\label{eq:StokesDB}
\int_{U}\mathcal{P}_{\mathcal{M}}F(x)dx=\int_{\dv U}\sigma_{\mathcal{M}}(\nu(y))F(y)d\sigma(y)
\end{align}
for any $F\in W^{1,2}(\Om,\Lambda)$. \newline

When considering the general Dirac-Beltrami operator one should also distinguish the special case when the linear operator $\mathcal{M}(x)\in \LL(\Lambda \R^n)$ is grade preserving for almost all $x\in \Om$, i.e., when 
\begin{align*}
\mathcal{M}(x)(\Lambda^k \R^n)\subset \Lambda^k \R^n
\end{align*}
for all $k=0,1,2,...,n$ and almost all $x\in \Om$. Set $\mathcal{M}_k(x)=\mathcal{M}_k(x)\vert_{\Lambda^k \R^n}$ and let $F_k(x)=\langle F(x)\rangle_k$ for all $k=0,1,2,...,n$ so that $F=\sum_{k=0}^nF_k$. Then for every $k=0,1,2,...,n$ (where $F_j$=0 if $j<0$ or $j>n$)
\begin{align*}
\langle \mathcal{D}^\pm F\rangle_k=dF_{k-1}+\delta F_{k+1},
\end{align*}
which gives the coupled system of equations
\begin{align*}
dF_{k-1}-\delta F_{k+1}=\mathcal{M}_k(dF_{k-1}+\delta F_{k+1})
\end{align*}
for $k=0,1,2,...,n$, or equivalently 
\begin{align*}
\delta F_{k}&=(I+\mathcal{M}_{k+1})^{-1}(I-\mathcal{M}_{k+1})dF_{k-2}\\
d F_{k}&=(I-\mathcal{M}_{k+1})^{-1}(I+\mathcal{M}_{k+1})\delta F_{k+2}\
\end{align*}
Defining $\mathcal{A}_k(x)=(I+\mathcal{M}_k(x))^{-1}(I-\mathcal{M}_k(x))$ and using that $d^2=\delta^2=0$, we find that each component $F_k$ satisfies the equations
\begin{align}\label{eq:OverS}
\delta \mathcal{A}_{k+1}(x) dF_{k}(x)&=0\\
d \mathcal{A}_{k-1}(x)^{-1} \delta F_{k}(x)&=0
\end{align}
which form an \emph{overdetermined elliptic system}. In the special case when $\mathcal{M}(x)\equiv 0$, the Dirac--Beltrami equation reduces to the homogeneous Hodge--Dirac equation 
\begin{align*}
\Dm F(x)=0.
\end{align*}
Since $(\Dm)^2=-\Delta$, it follows that $\Delta F_k=(d\delta+\delta d)F_k=0$ for each $k=0,1,2,...,n$. However, in view of \eqref{eq:OverS}, we in fact have that 
\begin{align*}
\delta dF_{k}(x)&=0\\
d \delta F_{k}(x)&=0
\end{align*}
or each $k=0,1,2,...,n$.

\subsection{\sffamily Hodge Duality}

Assume that we have a self-dual Beltrami operator $\mathcal{M}(x)$, i.e., that $\mathcal{M}^c(x)=\mathcal{M}(x)$. Then by definition any solution 
\begin{align*}
\Dm F(x)=\mathcal{M}(x)\Dp F(x)
\end{align*}
 is also a solution of 
\begin{align*}
\Dm F(x)=\mathcal{M}^c(x)\Dp F(x)\Longleftrightarrow \Dm F(x)=\star \mathcal{M}(x)(\Dp F(x)\star)\Longleftrightarrow \Dm F(x)\star = \mathcal{M}(x)(\Dp F(x)\star)
\end{align*}

Thus, by Lemma \ref{lem:StarD}, any self-dual equation also solves
\begin{align*}
\Dp (\widehat{F(x)}\star)=\mathcal{M}(x)\Dm (\widehat{F(x)}\star). 
\end{align*}

Note that if $F:\Om \to \Lambda^{ev}V$, then $\widehat{F(x)}=F(x)$ and so the self-dual equations become 
\begin{align*}
\Dp (F(x)\star)=\mathcal{M}(x)\Dm (F(x)\star). 
\end{align*}

In dimension two the Hodge star map coincides with the standard complex structure $J$ of $\R^2$. In this case, the self-duality of $\mathcal{M}(x)$ is equivalent to the condition that 
\begin{align*}
J^{-1}\mathcal{M}(x)J=\mathcal{M} \Longleftrightarrow \mathcal{M}(x)J- J\mathcal{M}(x)=0 \text{ for a.e. $x\in \Om$},
\end{align*}
i.e., $\mathcal{M}(x)$ is $\C$-linear. In this case, Beltrami equation reduces to the \emph{anti}-$\C$-linear Beltrami equation 
\begin{align*}
\dbar f(z)=\mu(z)\overline{\dv f(z)}
\end{align*}

In subsection \ref{subsec:Conv} we saw that any solution of $\Dm F(x)=\mathcal{M}(x)\Dp F(x)$ with $F\in W^{1,2}(\Om,\Lambda^{(0,2)})$ also solved an equation of the form 
\begin{align*}
\Dp F(x)=\mathcal{C}(x)\Dm F(x)
\end{align*}
where $\mathcal{C}(x)$ is a conformal map for a.e. $x$. In view of the discussion in this subsection it follows that $F$ also solves its dual equation 
\begin{align*}
\Dp (F(x)\star)=\mathcal{C}(x)\Dm (F(x)\star).
\end{align*}


\section{\sffamily The Cauchy Transform on $\R^n$.}\label{sec:Cauchy}

\subsection{\sffamily Dirac Operators and the Cauchy Transforms}

\begin{Def}[Cauchy transform]
\label{def:Ctransform}
The \emph{Cauchy transforms} $\Cc^\pm$ are acting $C^\infty_0(\R^n,\Lambda)$ are defined by
\begin{align}
\label{eq:C1}
\Cc^+\Phi(x)=\frac{1}{\sigma_{n-1}}\int_{\R^n}\frac{y-x}{\vert y-x\vert^n}\gp \Phi(y)dy,\\
\label{eq:C2}
\Cc^-\Phi(x)=\frac{1}{\sigma_{n-1}}\int_{\R^n}\frac{x-y}{\vert x-y\vert^n}\gm \Phi(y)dy,
\end{align}
for $\Phi\in C^\infty_0(\R^n,\Lambda)$. 
\end{Def}

We also recall the Riesz potentials 
\begin{align*}
\mathcal{I}_\alpha \Phi(x)=\frac{1}{\omega_n}\int_{\R^n}\frac{\Phi(y)}{\vert y-x\vert^{n-\alpha}}dy. 
\end{align*} 
where $0<\alpha <n$. We then clearly have the bound 
\begin{align*}
\vert \Cc^+\Phi(x)\vert\leq \vert \mathcal{I}_1\vert \Phi\vert (x)\vert
\end{align*}
for all $\Phi\in  C^\infty_0(\R^n,\Lambda \R^n)$ and all $x\in \R^n$. 

The Hardy--Littlewood--Sobolev inequality for a Riesz potentials \cite[p. 119]{S} shows that for $1<p<n$ and $p^\ast=\frac{np}{n- p}<+\infty$ there exists a constant $A_p$ such that 
\begin{align*}
\Vert \mathcal{I}_1 f(x)\Vert_{p^\ast}\leq A_p\Vert f\Vert_p, 
\end{align*}
for every $f\in L^{p}(\R^n)$. This immediately implies the following theorem:

\begin{Thm}\label{thm:DiracC}
For all $1<p<n$ such that $p^\ast=\frac{np}{n- p}<+\infty$, the Cauchy transforms $\Cc^{\pm}$ extends to continuous maps $\Cc^{\pm}:L^p(\R^n,\Lambda \R^n)\to L^{p^\star}(\Om,\Lambda \R^n)$. In particular there exists an absolute constant $A_p=A(p,n)$ such that for all $f\in L^{p^\star}(\Om,\Lambda \R^n)$.
\begin{align*}
\Vert \Cc^\pm F\Vert_{p^\star}\leq A_p\Vert F\Vert_{p}.
\end{align*}
\end{Thm}

\begin{Thm}
Let $\Om$ be a bounded domain in $\R^n$. Let $G\in L^2(\R^n,\Lambda)$. Then the weak derivatives satisfy
\begin{align*}
\Dp \Cc^+G(x)=G(x)\,\,\, \text{and}\,\,\, \Dm \Cc^-G(x)=G(x). 
\end{align*}
\end{Thm}

\begin{proof}
First assume that $F\in C^\infty_0(\Om)$. Then by the Cauchy-Pompieu formula \eqref{eq:GenCF} that the solution to $\Dp F(x)=G(x)$ is given by $F(x)=\Cc^+G(x)$. Since $C^\infty_0(\Om)$ is dense in $L^2(\Om,\Lambda)$, the result follows by a standard approximation argument.
\end{proof}

\subsection{\sffamily Hölder estimates for the Cauchy Transforms}

\begin{Lem}
If $n=2k$ then 
\begin{align*}
&\frac{\xi-x}{\vert \xi-x\vert^n}-\frac{\xi-y}{\vert \xi-y\vert^n}=\frac{(\xi-x)\big[(\xi-y)^{n-1}-(\xi-x)^{n-1}\big](\xi-y)}{\vert \xi-x\vert^n\vert y-\xi\vert^n}
\end{align*}

If $n=2k+1$ then 

\begin{align*}
&\frac{\xi-x}{\vert \xi-x\vert^n}-\frac{\xi-y}{\vert \xi-y\vert^n}=\frac{(\xi-x)\big[\vert \xi-y\vert(\xi-y)^{n-2}-\vert \xi-x\vert(\xi-x)^{n-2}\big](\xi-y)}{\vert \xi-x\vert^n\vert y-\xi\vert^n}.
\end{align*}

\end{Lem}

\begin{proof}

If $n=2k$, then $\vert \xi-x\vert^n=(\vert \xi-x\vert^2)^k=(\xi-x)^{2k}$. Consequently,

\begin{align*}
&(\xi-x)\vert \xi-y\vert^n-\vert \xi-x\vert^n(\xi-y)=(\xi-x)(\xi-y)^{2k-1}(\xi-y)+(\xi-x)(\xi-x)^{2k-1}(\xi-y)\\&=(\xi-x)\big[(\xi-y)^{2k-1}-(\xi-x)^{2k-1}\big](\xi-y).
\end{align*}

If $n=2k+1$, then $\vert \xi-x\vert^n=\vert \xi-x\vert (\vert \xi-x\vert^2)^k=\vert \xi-x\vert (\xi-x)^{2k}$. Consequently,

\begin{align*}
&(\xi-x)\vert \xi-y\vert^n-\vert \xi-x\vert^n(\xi-y)=(\xi-x)\vert \xi-y\vert(y-\xi)^{2k-1}(\xi-y)-(\xi-x)\vert \xi-x\vert(\xi-x)^{2k-1}(\xi-y)\\&=(\xi-x)\big[\vert \xi-y\vert(\xi-y)^{2k-1}-\vert \xi-x\vert(\xi-x)^{2k-1}\big](\xi-y).
\end{align*}

\end{proof}

\begin{Lem}\label{lem:IneqKernel}
If $n=2k$, $n>2$, we have the inequality 
\begin{align*}
\bigg\vert \frac{\xi-x}{\vert \xi-x\vert^n}-\frac{\xi-y}{\vert \xi-y\vert^n}\bigg\vert\leq c_n\sum_{j=1}^{n-1}\frac{\vert x-y\vert^{j}\vert \xi-x\vert^{n-1-j}}{\vert \xi-x\vert^{n-1}\vert y-\xi\vert^{n-1}}
\end{align*}
If $n=2k+1$, $n>2$, we have the inequality 
\begin{align*}
\bigg\vert \frac{\xi-x}{\vert \xi-x\vert^n}-\frac{\xi-y}{\vert \xi-y\vert^n}\bigg\vert\leq c_n\sum_{j=1}^{n-1}\bigg(\delta_{\vert \xi-x\vert>\vert \xi-y\vert}\frac{\vert x-y\vert^{j}\vert \xi-x\vert^{n-1-j}}{\vert \xi-x\vert^{n-1}\vert y-\xi\vert^{n-1}}+\delta_{\vert \xi-x\vert<\vert \xi-y\vert}\frac{\vert x-y\vert^{j}\vert \xi-y\vert^{n-1-j}}{\vert \xi-x\vert^{n-1}\vert y-\xi\vert^{n-1}}\bigg).
\end{align*}
\end{Lem}

\begin{proof}
First assume that $n=2k$. Set $\eta=\xi-x$. Then 

\begin{align*}
(\xi-y)^{n-1}-(\xi-x)^{n-1}=(\eta+x-y)^{n-1}+\eta^{n-1}=\eta^{n-1}+R-\eta^{n-1},
\end{align*}
where $R$ is sum of products, each containing at least one factor $(x-y)$. Moreover, we have the estimate 
\begin{align*}
\vert R\vert\leq c_n\sum_{j=1}^{n-1}\vert x-y\vert^{j}\vert \eta\vert^{n-1-j}=c_n\sum_{j=1}^{n-1}\vert x-y\vert^{j}\vert \xi-x\vert^{n-1-j}
\end{align*}
for some constant $c_n$ depending on $n$. Thus,

\begin{align*}
\big\vert(\xi-x)\big[(\xi-y)^{n-1}-(\xi-x)^{n-1}\big](\xi-y)\big\vert\leq c_n\sum_{j=1}^{n-1}\vert\xi-x\vert\vert x-y\vert^{j}\vert \xi-x\vert^{n-1-j}\vert\xi-y\vert. 
\end{align*}

Now assume that $n=2k+1$. If $\vert \xi-y\vert>\vert \xi-x\vert$ set $\eta=\xi-x$, and we get 

\begin{align*}
&\vert \xi-y\vert^{n-1}-\vert \xi-x\vert^{n-1}=\vert \eta+x-y\vert^{n-1}-\vert \eta\vert^{n-1}\leq \big(\vert \eta\vert+\vert x-y\vert\big)^{n-1}-\vert \eta\vert^{n-1}\\
&\leq c_n\sum_{j=1}^{n-1}\vert x-y\vert^j\vert \eta\vert^{n-1-j}=c_n\sum_{j=1}^{n-1}\vert x-y\vert^j\vert \xi-x\vert^{n-1-j}.
\end{align*}
If $\vert \xi-y\vert<\vert \xi-x\vert$ set $\zeta=\xi-y$, and we get 
\begin{align*}
&\vert \xi-y\vert^{n-1}-\vert \xi-x\vert^{n-1}=\vert \zeta\vert^{n-1}-\vert \zeta+y-x\vert^{n-1}\leq \big(\vert \zeta\vert+\vert x-y\vert\big)^{n-1}-\vert \zeta\vert^{n-1}\\
&\leq c'_n\sum_{j=1}^{n-1}\vert x-y\vert^j\vert \zeta\vert^{n-1-j}=c'_n\sum_{j=1}^{n-1}\vert x-y\vert^j\vert \xi-y\vert^{n-1-j}.
\end{align*}
\end{proof}

\begin{Prop}\label{prop:CauchyHölder}
Let $n>2$ and assume that $p>n$. If $\Phi\in L^p(\R^n,\Lambda \R^n)$ then $\Cc^+\Phi$ is Hölder continuous with Hölder constant $\displaystyle\alpha=1-\frac{n}{p}$.
\end{Prop}

\begin{rem}
This is an extension of Theorem 4.3.13 in \cite{AIM} for the Cauchy transform in the plane to higher dimension. The only subtlety in proof compared to the plane is the more complicated estimate in Lemma \ref{lem:IneqKernel}.
\end{rem}

\begin{proof}
We may assume that $\vert x-y\vert<1$. By Lemma \ref{lem:IneqKernel}, 
\begin{align*}
&\vert \Cc^+\Phi(x)-\Cc^+\Phi(y)\vert=\bigg\vert \frac{1}{\omega_n}\int_{\R^n}\bigg(\frac{\xi-x}{\vert \xi-x\vert^n}-\frac{\xi-y}{\vert \xi-y\vert^n}\bigg)\gp \Phi(\xi)d\xi\bigg\vert\\
&\leq \frac{c_n}{\omega_n}\sum_{j=1}^{n-1}\int_{\R^n}\bigg(\frac{\vert x-y\vert^{j}\vert \xi-x\vert^{n-1-j}}{\vert \xi-x\vert^{n-1}\vert y-\xi\vert^{n-1}}+\frac{\vert x-y\vert^{j}\vert \xi-y\vert^{n-1-j}}{\vert \xi-x\vert^{n-1}\vert y-\xi\vert^{n-1}}\bigg)\vert  \Phi(\xi)\vert d\xi\\
&\leq \frac{c_n}{\omega_n}\sum_{j=1}^{n-1}\vert x-y\vert^j\int_{\R^n}\bigg(\frac{1}{\vert \xi-x\vert^{j}\vert y-\xi\vert^{n-1}}+\frac{1}{\vert \xi-x\vert^{n-1}\vert y-\xi\vert^{j}}\bigg)\vert  \Phi(\xi)\vert d\xi.
\end{align*}
We now apply Hölder's inequality with Hölder conjugate pair chosen so that $(n-1)q<n$. Then we get 
\begin{align*}
&\vert \Cc^+\Phi(x)-\Cc^+\Phi(y)\vert\\&\leq \frac{c_n}{\omega_n}\Vert \Phi\Vert_{L^p}\sum_{j=1}^{n-1}\vert x-y\vert^{j}\bigg\{\bigg(\int_{\R^n}\frac{1}{\vert \xi-x\vert^{qj}\vert \xi-y\vert^{q(n-1)}}d\xi\bigg)^{1/q}+\bigg(\int_{\R^n}\frac{1}{\vert \xi-x\vert^{q(n-1)}\vert y-\xi\vert^{qj}} d\xi\bigg)^{1/q}\bigg\}.
\end{align*}

Since 
\begin{align*}
&\int_{\R^n}\frac{1}{\vert \xi-x\vert^{qj}\vert \xi-y\vert^{q(n-1)}}d\xi=\int_{\R^n}\frac{1}{\vert \eta\vert^{qj}\vert \eta+x-y\vert^{q(n-1)}}d\eta\\
&=\int_{\R^n}\frac{\vert x-y\vert^n}{\vert x-y\vert^{qj+q(n-1)}\vert \zeta\vert^{qj}\vert \zeta+\frac{x-y}{\vert x-y\vert}\vert^{q(n-1)}}d\zeta\\
&=\vert x-y\vert^{n-qj-q(n-1)}\int_{\R^n}\frac{1}{\vert \zeta\vert^{qj}\vert \zeta+\frac{x-y}{\vert x-y\vert}\vert^{q(n-1)}}d\zeta
\end{align*}
and similarly for the other integral, and since  
\begin{align*}
I:=\sup_{x\neq y}\int_{\R^n}\frac{1}{\vert \zeta\vert^{qj}\vert \zeta+\frac{x-y}{\vert x-y\vert}\vert^{q(n-1)}}d\zeta<+\infty
\end{align*}
we find that 

\begin{align*}
&\vert \Cc^+\Phi(x)-\Cc^+\Phi(y)\vert\\&\leq \frac{2Ic_n}{\omega_n}\Vert \Phi\Vert_{L^p}\sum_{j=1}^{n-1}\vert x-y\vert^{j}\vert x-y\vert^{(n-qj-q(n-1))/q}\\
&\leq  \frac{2nIc_n}{\omega_n}\Vert \Phi\Vert_{L^p}\vert x-y\vert^{1-n/p}
\end{align*}

\end{proof}

\subsection{\sffamily BMO estimates for the Cauchy Transform}

In what follows we set $B_1(0)=\{x\in \R^n: \vert x\vert<1\}$ and recall that $n\vert B_1(0)\vert=\omega_{n-1}$

\begin{Lem}\label{lem:CauchyBall}
\begin{align*}
\Cc^+(\chi_{B_1(0)})(x)=\left\{
\begin{array}{ll}
\displaystyle\frac{1}{n}x & \text{if } \vert x\vert < 1,\\ &\\ \displaystyle \frac{1}{n}\frac{x}{\vert x\vert^n} & \text{if } \vert x\vert >1 .
\end{array} \right.
\end{align*}
\end{Lem}

\begin{proof}
This follows from the well-known formula for the Newton potential of a ball and the fact that $(\Dp)^2=\Delta$. 
\end{proof}

\begin{Lem}\label{lem:BMO}
Let $\phi\in L^{p}(\R^n,\Lambda \R^n)$ with $1<p<\infty$ and set $\phi_r(x)=\phi(rx)$. Then 
\begin{align*}
\Cc^+(\phi_r)(x)=r^{-1}\Cc^+(\phi)(rx)
\end{align*}
\end{Lem}

\begin{proof}
\begin{align*}
\Cc^+(\phi_r)(x)&=\frac{1}{\sigma_{n-1}}\int_{\R^n}\frac{y-x}{\vert y-x\vert^n}\gp \phi(ry)dy=\{ry=\xi\}\\
&=\frac{1}{\sigma_{n-1}}\int_{\R^n}\frac{r^{-1}\xi-x}{\vert r^{-1}\xi-x\vert^n}\gp \phi(\xi)r^{-n}d\xi\\
&=\frac{1}{r\sigma_{n-1}}\int_{\R^n}\frac{\xi-rx}{\vert \xi-rx\vert^n}\gp \phi(\xi)d\xi.
\end{align*}
\end{proof}

\begin{Prop}\label{prop:BMO}
Let $\phi\in L^n(\R^n,\Lambda \R^n)$. Then $\Cc^\pm\phi\in BMO(\R^n,\Lambda \R^n)$ and 
\begin{align*}
\Vert \Cc^\pm \phi \Vert_{BMO}\leq C_n^\pm\Vert \phi\Vert_{L^n}. 
\end{align*}
for some constant $C_n^\pm$ depending on $n$. 
\end{Prop}

\begin{proof}
It is enough to consider $\Cc^+$ as the proof for $\Cc^-$ is analogous. 
Since functions in $BMO(\R^n,\Lambda \R^n)$ are only defined up to constants we first need to give a meaning to $\Cc^+\phi$ in $BMO$ as an equivalence class modulo constants. Therefore we are free to add a constant to the original definition that may depend on the function $\phi$ itself since this gives the same element in $BMO$. For $\phi\in L^n(\R^n,\Lambda \R^n)$ we define 
\begin{align*}
 \Cc^+ \phi(x):=\frac{1}{\sigma_{n-1}}\int_{\R^n}\bigg[\frac{y-x}{\vert y-x\vert^n}-\frac{y}{\vert y\vert^n}\chi_{B_1(0)}(y)\bigg]\gp \phi(y)dy.
\end{align*}

Then the integral converges for almost every $x$. Since $C^\infty_0(\Om,\Lambda \R^n)$ is dense in $L^n(\R^n,\Lambda \R^n)$ it is sufficient to prove the BMO-estimate for a smooth compactly supported function.

For any ball $B_r(x_0)$ set 
\begin{align*}
(\phi)_{B_r(x_0)}:=\frac{1}{\vert B_r(x_0)\vert}\int_{B_r(x_0)}\phi(x)dx. 
\end{align*}

After a translation we may assume that $x_0=0$. Set $\phi_r(x)=\phi(rx)$. By Lemma \ref{lem:BMO}, 

Since 
\begin{align*}
&\frac{1}{\vert B_r(0)\vert}\int_{B_r(0)}\vert \Cc^+\phi(x)-( \Cc^+ \phi)_{B_r(0)}\vert dx=\{r\xi=x\}\\&=\frac{1}{\vert B_r(0)\vert}\int_{B_1(0)}\vert \Cc^+\phi(r\xi )-( \Cc^+ \phi)_{B_r(0)}\vert r^nd\xi\\
&=\frac{r}{\vert B_1(0)\vert}\int_{B_1(0)}\vert \Cc^+\phi_r(x )-( \Cc^+ \phi_r)_{B_1(0)}\vert dx.
\end{align*}
We see it is enough to prove that for all $r>0$
 
\begin{align*}
\frac{r}{\vert B_1(0)\vert}\int_{B_1(0)}\vert \Cc^+\phi_r(x )-( \Cc^+ \phi_r)_{B_1(0)}\vert dx &\leq C_n\Vert \phi(x)\Vert_n\\ &\Longleftrightarrow\\
\frac{1}{\vert B_1(0)\vert}\int_{B_1(0)}\vert \Cc^+\phi_r(x )-( \Cc^+ \phi_r)_{B_1(0)}\vert dx &\leq C_n \Vert \phi_r(x)\Vert_n.
\end{align*}
for some constant $C_n$ depending only on $n$.

Now,

\begin{align*}
( \Cc^+ \phi)_{B_1(0)}&=\frac{1}{\vert B_1(0)\vert}\int_{B_1(0)}\bigg(\frac{1}{\sigma_{n-1}}\int_{\R^n}\bigg[\frac{y-x}{\vert y-x\vert^n}-\frac{y}{\vert y\vert^n}\chi_{B_1(0)}(y)\bigg]\gp \phi(y)dy\bigg)dx\\
&=\frac{1}{\vert B_1(0)\vert}\int_{\R^n}\bigg(\frac{1}{\sigma_{n-1}}\int_{B_1(0)}\frac{y-x}{\vert y-x\vert^n}dx\bigg)\gp \phi(y)dy-\frac{1}{\sigma_{n-1}}\int_{\R^n}\frac{y}{\vert y\vert^n}\chi_{B_1(0)}(y)\gp \phi(y)dy\\
&=-\frac{1}{\vert B_1(0)\vert}\int_{\R^n}\Cc^+(\chi_{B_1(0)})(x)\gp \phi(y)dy-\frac{1}{\sigma_{n-1}}\int_{ B_1(0)}\frac{y}{\vert y\vert^n}\gp \phi(y)dy.
\end{align*}
Using Lemma \ref{lem:CauchyBall}

\begin{align*}
&-\frac{1}{\vert B_1(0)\vert}\int_{\R^n}\Cc^+(\chi_{B_1(0)})(x)\gp \phi(y)dy\\&=-\frac{1}{\vert B_1(0)\vert}\int_{B_1(0)}\Cc^+(\chi_{B_1(0)})(x)\gp \phi(y)dy-\frac{1}{\vert B_1(0)\vert}\int_{\R^n\setminus B_1(0)}\Cc^+(\chi_{B_1(0)})(x)\gp \phi(y)dy\\
&=-\frac{1}{n\vert B_1(0)\vert}\int_{B_1(0)}y\gp \phi(y)dy-\frac{1}{n\vert B_1(0)\vert}\int_{\R^n\setminus B_1(0)}\frac{y}{\vert y\vert^n}\gp \phi(y)dy\\
&=-\frac{1}{\sigma_{n-1}}\int_{B_1(0)n}y\gp \phi(y)dy-\frac{1}{\sigma_{n-1}}\int_{\R^n\setminus B_1(0)}\frac{y}{\vert y\vert^n}\gp \phi(y)dy.
\end{align*}
Hence 
\begin{align*}
( \Cc^+ \phi)_{B_1(0)}&=-\frac{1}{\vert B_1(0)\vert}\int_{\R^n}\Cc^+(\chi_{B_1(0)})(x)\gp \phi(y)dy-\frac{1}{\sigma_{n-1}}\int_{ \R^n}\frac{y}{\vert y\vert^n}\gp \phi(y)dy.
\end{align*}
Thus, 
\begin{align*}
&\Cc^+\phi(x)-( \Cc^+ \phi)_{B_1(0)}\\&=\frac{1}{\sigma_{n-1}}\int_{B_1(0)}\bigg\{\frac{y-x}{\vert y-x\vert^n}+y\bigg\}\gp \phi(y)dy+\frac{1}{\sigma_{n-1}}\int_{\R^n\setminus B_1(0)}\bigg\{\frac{y-x}{\vert y-x\vert^n}+\frac{y}{\vert y\vert^n}\bigg\}\gp \phi(y)dy.
\end{align*}

Define the function 
\begin{align*}
\psi(x,y):=\left\{
\begin{array}{ll}
\displaystyle \frac{y-x}{\vert y-x\vert^n}+y& \text{if } \vert y\vert < 1,\\ &\\ \displaystyle\frac{y-x}{\vert y-x\vert^n}+\frac{y}{\vert y\vert^n} & \text{if } \vert y\vert >1 .
\end{array} \right.
\end{align*}

Then 
\begin{align*}
&\frac{1}{\vert B_1(0)\vert }\int_{B_1(0)}\vert \Cc^+\phi(x)-( \Cc^+ \phi)_{B_1(0)}\vert dx\leq \frac{1}{\sigma_{n-1}\vert B_1(0) \vert }\int_{B_1(0)}\bigg(\int_{\R^n}\vert \psi(x,y)\vert \vert \phi(y)\vert dy\bigg)dx\\
&=\frac{1}{\sigma_{n-1}\vert B_1(0)\vert }\int_{\R^n}\bigg(\int_{B_1(0)}\vert \psi(x,y)\vert dx\bigg) \vert \phi(y)\vert dy.
\end{align*}

We now estimate $\int_{B_1(0)}\vert \psi(x,y)\vert dx$. When $\vert y\vert<1$ we get
\begin{align*}
\int_{B_1(0)}\vert \psi(x,y)\vert dx\leq \int_{B_1(0)}\frac{1}{\vert y-x\vert^{n-1}}dx+\int_{B_1(0)}\vert y\vert dx\leq c_1,
\end{align*}
for some constant $c_1$ independent of $y$.  When $\vert y\vert>1$ we get
\begin{align*}
\int_{B_1(0)}\vert \psi(x,y)\vert dx&\leq \int_{B_1(0)}\frac{1}{\vert y-x\vert^{n-1}}dx+\int_{B_1(0)}\frac{1}{\vert y\vert^n} dx\\
&\leq  \frac{1}{\vert y\vert^{n-1}}\bigg(\int_{B_1(0)n}\frac{1}{\vert 1-y^{-1}\gp x\vert^{n-1}}dx+\int_{B_1(0)}dx\bigg)\\
&\leq  \frac{c_2}{\vert y\vert^{n-1}}
\end{align*}
for some constant $c_2$ independent of $y$. Let $g(y)=\int_{B_1(0)}\vert \psi(x,y)\vert dx$. This gives 
\begin{align*}
\Vert g\Vert_{L^{n/(n-1)}(\R^n)}^{(n-1)/n}=\int_{B_1(0)}c_1^{n/(n-1)}dy+c_2^{n/(n-1)}\int_{\R^n\setminus B_1(0)}\frac{1}{\vert y\vert^{n}}dy<+\infty
\end{align*}
By Hölder's inequality
\begin{align*}
&\frac{1}{\vert B_1(0)\vert }\int_{B_1(0)}\vert \Cc^+\phi(x)-( \Cc^+ \phi)_{B_1(0)}\vert dx\leq \frac{1}{\sigma_{n-1}\vert B_1(0)\vert }\int_{\R^n}g(y)\vert \phi(y)\vert dy\\
&\leq C_n\Vert \phi\Vert_{L^n}
\end{align*}
for some constant $C_n$, and the proof is complete. 
\end{proof}

\subsection{\sffamily Compactness Properties of the Cauchy Transforms}

In many applications one is more interested in the restriction of the Cauchy transforms to bounded domains $\Om\subset \R^n$. We write the restriction of $\Cc^\pm F\vert_{\Om}$ as multiplication with the characteristic function $\chi_\Om$. On bounded domains the Cauchy transforms have the following compactness properties:
\begin{Thm}\label{thm:CompactC}
Let $\Om$ be a bounded domain in $\R^n$. Then the following operators are compact.
\begin{itemize}
\item[(i)] For $n<p\leq \infty$,
\begin{align*}
\chi_\Om\circ \Cc^\pm: L^{p}(\R^n,\Lambda \R^n)\to C^{\alpha}(\Om,\Lambda \R^n),\,\,\,\, 0\leq \alpha <1-n/p
\end{align*}
\item[(ii)] For $1\leq p<n$,
\begin{align*}
\chi_\Om\circ \Cc^\pm: L^{p}(\R^n,\Lambda \R^n)\to L^{s}(\Om,\Lambda \R^n),\,\,\,\, 1\leq s<p^\star=\frac{np}{n-p}
\end{align*}
\end{itemize}
\end{Thm}

\begin{proof}
The proofs are the same as the proof of Theorem 4.3.14 in \cite{AIM} with suitable modifications. The proof of claim $(ii)$ follows from Proposition \ref{prop:CauchyHölder} and an application of Arzela-Ascoli's theorem. The proof of claim $(ii)$ follows from Theorem 1 in \cite[page 120]{S} and Rellich-Kondrachov compact embedding theorem. 
\end{proof}


\section{\sffamily The Beurling-Ahlfors and Riesz Transforms on $\R^n$.}

\subsection{\sffamily The Beurling-Ahlfors Transform on $\R^n$.}

In the ground breaking papers \cite{DS}  and \cite{IM1} the authors introduced a higher dimensional generalisation of the classical Beurling-Ahlfors transform to study quasiconformal $4$-manifolds and removable singularities of quasiregular maps in even dimensions respectively. Furthermore this operator was also used to give a local formulae for characteristic classes on even dimensional quasiconformal manifolds, see \cite{CDT}. This higher dimensional operator, denoted by $\mathcal{S}$ is intimately connected with Hodge-Morrey decompositions of differential forms on $\R^n$. Indeed, let $\omega\in L^p(\R^n, \Lambda^k \R^n)$ be a differential $k$-form. Then the Hodge-Morrey decomposition of $\omega$ is given by
\begin{align*}
\omega=d \alpha+\delta \beta,
\end{align*}
with $\alpha\in \text{ker}(\delta)\cap W^{1,p}(\R^n, \Lambda^{k-1}\R^n)$ and  $\beta\in \text{ker}(d)\cap W^{1,p}(\R^n, \Lambda^{k+1}\R^n)$ and where $\delta=d^\ast$. The Beurling-Ahlfors transform on $\mathcal{S}:L^p(\R^n, \Lambda^k \R^n)\to L^p(\R^n, \Lambda^k \R^n)$ is defined by the rule 
\begin{align*}
\mathcal{S}(\omega)=d\alpha-\delta \beta,
\end{align*}
or alternatively $\mathcal{S}=(d\delta-\delta d)\circ \Delta^{-1}$, where $\Delta$ is the Laplacian.

In \cite{IM1} Theorem 8.1 it is shown that the Beurling-Ahlfors operator has the following properties:
\begin{itemize}
\item[(i)] $\mathcal{S}(d\alpha)=d\alpha$ for all $\alpha$ with $d\alpha\in L^p(\R^n, \Lambda^k \R^n)$.
\item[(ii)] $\mathcal{S}(\delta \alpha)=-\delta\alpha$ for all  $\alpha$ with $\delta \alpha\in L^p(\R^n, \Lambda^k \R^n)$.
\item[(iii)] $\mathcal{S}$ is self-adjoint and $\mathcal{S}\circ \mathcal{S}=\text{Id}$ and $\mathcal{S}^{-1}=\mathcal{S}$. 
\item[(iv)] $\mathcal{S}$ anti-commutes with the Hodge star, i.e., 
\begin{align*}
\mathcal{S}\star=-\star \mathcal{S}.
\end{align*} 
\item[(v)] $\mathcal{S}: L^2(\R^n, \Lambda \R^n)\to  L^2(\R^n, \Lambda \R^n)$ is an isometry and $\mathcal{S}$ is bounded on $L^p(\R^n,\Lambda \R^n)$. 
\end{itemize}
 For applications it is of great interest to know the $L^p$-norm of $\mathcal{S}$. Indeed, in the paper \cite{IM1} Theorem 8.4 the authors prove the following bound 
\begin{align}\label{thm:BoundIM}
\Vert \mathcal{S}\Vert_{L^p(\R^n,\Lambda \R^n)}\leq C(n)\max\bigg\{\frac{1}{p-1},p-1\bigg\}, \,\, 1<p<+\infty,
\end{align}
where $n$ is the dimension and $C(n)$ is a polynomial in $n$ not depending on $p$. In fact the authors conjecture that $C(n)=1$ for all $n$. This conjecture remains open so far  and has been the focus of great attention over the years. It is even open in dimension $n=2$. The bound \ref{thm:BoundIM} has subsequently been improved over the years by several authors, e.g., \cite{BL},\cite{NV},\cite{PSW} using in particular Bellman function techniques and Burkholder's sharp inequality for martingale transforms on $L^p$. The best bound so far is proven in \cite{H} and given by 
\begin{align}\label{normB1}
\Vert \mathcal{S}\Vert_{L^p(\R^n,\Lambda \R^n)}\leq (1+n/2)(p-1),
\end{align}
where $p^\star:=\max\{p/(p-1),p\}$. In fact when $\mathcal{S}$ is restricted to the spaces $L^p(\R^n,\Lambda^k\R^n)$, then in \cite{H} it is shown that the $L^p$-norms are have the estimates 
\begin{align}\label{normB2}
\Vert \mathcal{S}\Vert_{L^p(\R^n,\Lambda^k \R^n)}\leq \bigg(\frac{2k(n-k)}{n}+1\bigg)(p^\star-1)\leq (2k+1)(p^\star-1). 
\end{align}

\subsection{\sffamily Factorisation formula}

Let $\mathscr{S}(\R^n, \Lambda \R^n)$ denote the Schwartz space of differential forms on $\R^n$, i.e. 
\begin{align*}
\mathscr{S}(\R^n, \Lambda \R^n):=\{\omega\in C^\infty(\R^n, \Lambda \R^n): \sup_{x}\vert x^\beta\dv^\alpha \omega(x)\ \vert\text{ for all muti-indices $\alpha$ and $\beta$}\}. 
\end{align*}

 One easily checks that $\mathscr{S}(\R^n, \Lambda \R^n)$ is a Fréchet algebra, closed under the products $\wedge,\ri,\li,\gp,\gm$. Moreover, $\mathscr{S}(\R^n, \Lambda \R^n)$ is dense in $L^p(\R^n, \Lambda \R^n)$ for $1<p<+\infty$. 
 
 The Fourier transform $\mathcal{F}$ of a multicovector field $\omega\in \mathscr{S}(\R^n, \Lambda \R^n)$ is defined as
 \begin{align*}
 \mathcal{F}(\omega)(\xi)=\frac{1}{(\sqrt{2\pi})^n}\int_{\R^n}\omega(x)e^{-i\langle x,\xi\rangle}dx. 
 \end{align*}

In Theorem 8.2 in \cite{IM1} it is shown that $\mathcal{S}$ is an isometry on $L^2(\R^n, \Lambda^k \R^n)$ for every $k$. Let $M(\xi): \R^n \to \LL(\R^n)$ be the orthogonal transformation 
\begin{align*}
M(\xi)=I-2\vert \xi\vert^{-1}\xi\otimes \xi,
\end{align*}
being the reflection in the hyperplane orthogonal to $\xi$. Let $\widehat{M}_\xi: \LL(\Lambda \R^n)\to  \LL(\Lambda \R^n)$ denote its exterior extension (in the notation in \cite{IM1} $\widehat{M}(\xi)=M_{\#}(\xi)$). In Theorem 8.10 in \cite{IM1} is shown that the Fourier multiplier of $\mathcal{S}:L^2(\R^n, \Lambda \R^n)\to L^2(\R^n, \Lambda \R^n)$ is given by 
\begin{align*}
\mathcal{F}(\mathcal{S}\omega)(\xi)&=\widehat{M}(\xi)\mathcal{F}(\omega)(\xi).
\end{align*}

Here $\widehat{M}(\xi)$ is identified with its complexification acting on $\Lambda \R^n\otimes \C$ so that for any $w_1,w_2\in \Lambda \R^n$ we have $\widehat{M}(\xi)(w_1+iw_2)=\widehat{M}(\xi)(w_1)+i\widehat{M}(\xi)(w_2)$.

\begin{Lem}\label{lem:BMulti}
The Fourier multiplier of the Beurling transform can be expressed as 
\begin{align}
\mathcal{F}(\mathcal{S}\omega)(\xi)&=-\frac{i\xi }{\vert \xi\vert}\gp \widehat{\mathcal{F}(\omega)(\xi)}\gp \frac{i\xi }{\vert \xi\vert}. 
\end{align}
\end{Lem}

\begin{rem}
Note that here $\widehat{\omega(\xi)}$ denotes the involution on $\Lambda V$ ($\Lambda V\otimes \C$) and {\bf not} the Fourier transform.
\end{rem}

We will give to different proofs.
\begin{proof}[Proof 1]
Since $M(\xi)$ is a reflection in the hyperplane orthogonal to $\xi$, Theorem 4.1.3 in \cite{R} gives
\begin{align*}
M(\xi)v=-\xi\gp v\gp \xi^{-1}
\end{align*}
for every $v\in V$. 
Lemma \ref{lem:Refl} equation \eqref{ReflFormula} and the fact that the Clifford multiplicative inverse of a vector is given by $\xi^{-1}=\xi/\vert \xi\vert^2$ give
\begin{align*}
\mathcal{F}(\mathcal{S}\omega)(\xi)&=\xi\gp \widehat{\mathcal{F}(\omega)(\xi)}\gp \xi^{-1}=\frac{\xi}{\vert \xi \vert}\gp\widehat{\mathcal{F}(\omega)}\gp \frac{\xi}{\vert \xi \vert}\\
&=-\frac{i\xi}{\vert \xi \vert}\gp\widehat{\mathcal{F}(\omega)}\gp \frac{i\xi}{\vert \xi \vert}.
\end{align*}
\end{proof}

\begin{proof}[Proof 2]
We use the fact that $\mathcal{F}(d\omega)(\xi)=i\xi\wedge \mathcal{F}(\omega)(\xi)$ and $\mathcal{F}(\delta \omega)(\xi)=-i\xi\ri \mathcal{F}(\omega)(\xi)$. Set $\eta(\xi)=-\vert \xi\vert^{-2}\mathcal{F}(\omega)(\xi)$. Thus, 
\begin{align*}
&\mathcal{F}((d\delta-\delta d)\circ \Delta^{-1}\omega)(\xi)=i\xi \wedge (-i\xi \ri \eta (\xi))+i\xi \ri(i\xi \wedge\eta (\xi) )\\
&=(\xi \ri +\xi \wedge)(-\xi \ri \eta (\xi))+(\xi \ri +\xi \wedge)(\xi \wedge \eta (\xi))=\xi\gp(-\xi \ri \eta (\xi)+\xi \wedge \eta (\xi))\\
&=\xi \gp (\widehat{\eta(\xi)}\li \xi+\widehat{\eta(\xi)}\wedge)=\xi(\gp \widehat{\eta(\xi)}\gp)\\
&=\xi(\gp -\vert \xi\vert^{-2}\widehat{\mathcal{F}(\omega)(\xi)}\gp)\\
&=-\frac{i\xi}{\vert \xi \vert}\widehat{\mathcal{F}(\omega)}\gp \frac{i\xi}{\vert \xi \vert}. 
\end{align*}
Here we have used the fact that $\xi\wedge( \xi \wedge w)=0$ and $\xi\ri( \xi \ri w)=0$ for any multivector $w$. 
\end{proof}

\begin{Def}
The scalar Riesz transforms $\mathcal{R}_\nu$ acting on $L^p(\R^n)$ for any $\nu\in \mathbb{S}^{n-1}$ are given by the singular integrals
\begin{align*}
\mathcal{R}_\nu f(x)=\frac{2}{  \sigma_{n}}\text{p.v.}\int_{\R^n}\frac{(x_\nu-y_\nu)f(y)}{\vert x-y\vert^{n+1}}dy
\end{align*}
where $x_\nu=\langle x,\nu\rangle$. 
\end{Def}

The $L^p$-norms of $\mathcal{R}_\nu$ are given by $\Vert \mathcal{R}_\nu\Vert_p=H_p$, (Theorem 12.11 in \cite[p. 304]{IM2}) for any $\nu\in \mathbb{S}^{n-1}$ and the Fourier transforms of $\mathcal{R}_\nu \omega$ are given by
\begin{align*}
\mathcal{F}[\mathcal{R}_\nu\omega](\xi)=\frac{i\xi_\nu}{\vert \xi\vert}\mathcal{F}(\omega)(\xi). 
\end{align*}

\begin{Def}
Define the Clifford-Riesz transforms $\mathcal{R}^+$ and $\mathcal{R}^-$ are defined by
\begin{align*}
\mathcal{R}^{+}f(x)&:=\frac{2}{\sigma_n}\text{p.v.}\int_{\R^n}\frac{x-y}{\vert x-y\vert^{n+1}}\gp f(y)dy,\\
\mathcal{R}^{-}f(x)&:=\frac{2}{\sigma_n}\text{p.v.}\int_{\R^n}\frac{x-y}{\vert x-y\vert^{n+1}}\gm f(y)dy,
\end{align*}
\end{Def}
These operators arise naturally when considering Hardy spaces of multivector fields with respect to the Hodge-Dirac operators. For more on these operators we refer the reader to Chapter 4 in \cite{GM} and \cite{LMcIS}

\begin{Lem}\label{lem:Riesz}
$\mathcal{R}^+ $ and $\mathcal{R}^-$ are bounded operators on $L^p(\R^n, \Lambda \R^n)$ with norm estimates
\begin{align*}
 \Vert \mathcal{R}^{\pm} \Vert_p\leq 2H_p\frac{\sigma_{n-1}}{\sigma_n}.  
\end{align*}
\end{Lem}

\begin{proof}
It is enough to consider $\mathcal{R}^+$. We apply the real method of rotation, Theorem 12.5.1 in \cite{IM2}. $\Gamma: \R^n \to \LL(\Lambda V)$ is given by 
\begin{align*}
\Gamma(x)=\frac{2}{ \sigma_{n}}\frac{x}{\vert x\vert^{n-1}}\gp ,
\end{align*}
and so for all $x\in \mathbb{S}^{n-1}$ and all $t\in \R$, we have $\vert t\vert^{n+1}\Gamma(tx)=t\Gamma(x)$. Hence by Theorem 12.5.1. in \cite{IM2} $\mathcal{R}^\pm $ are bounded in $L^p(\R^n, \Lambda \R^n)$ for all $1<p<\infty$ and 
\begin{align*}
\Vert \mathcal{R}^+ \Vert_p\leq \frac{\pi}{2} H_p\int_{\mathbb{S}^{n-1}}\vert \Gamma(x)\vert d\sigma(x),
\end{align*}
where
\begin{align*}
\vert \Gamma (x)\vert=\max\{\vert \Gamma(x)w\vert: w\in \Lambda V, \vert w\vert=1\}. 
\end{align*}
Since $x\in \mathbb{S}^{n-1}$, $2 \omega_{n}^{-1}\vert \Gamma (x)\vert=\max\{\vert x\gp w\vert: w\in \Lambda V, \vert w\vert=1\}=\max\{\vert w\vert: w\in \Lambda V, \vert w\vert=1\}=1$ by Lemma \ref{lem:Norm}. Thus,
\begin{align*}
\int_{\mathbb{S}^{n-1}}\vert \Gamma(x)\vert d\sigma(x)=\frac{2}{ \omega_{n}}\int_{\mathbb{S}^{n-1}}\frac{\vert x\vert}{\vert x\vert^{n+1}}d\sigma(x)=\frac{2\omega_{n-1}}{\omega_n}.
\end{align*}
\end{proof}

It is sometimes useful to introduce the formal Riesz transform symbol 
\begin{align*}
\mathbf{R}:=\sum_{j=1}^n\mathcal{R}_{e_j}e_j
\end{align*}
with respect to an ON-basis $\{e_j\}_j$ of $\R^n$ much in the same way as the formal nabla symbol $\nabla$. Using this symbol we can write formally 
\begin{align*}
\mathcal{R}^+F(x)=\mathbf{R}\gp F(x),\,\,\, \mathcal{R}^-F(x)=\mathbf{R}\gm F(x),
\end{align*}

In addition, because of the algebraic decomposition of the action of Clifford multiplication into exterior an interior multiplication we have 
\begin{align*}
\mathbf{R}\gp F(x)=\mathbf{R}\wedge F(x)+\mathbf{R}\ri F(x), \,\,\, \mathbf{R}\gm F(x)=\mathbf{R}\wedge F(x)-\mathbf{R}\ri F(x),
\end{align*}
where 
\begin{align*}
\mathbf{R}\wedge F(x)&=\frac{2}{\sigma_n}\text{p.v.}\int_{\R^n}\frac{x-y}{\vert x-y\vert^{n+1}}\wedge F(y)dy,\\
\mathbf{R}\ri F(x)&=\frac{2}{\sigma_n}\text{p.v.}\int_{\R^n}\frac{x-y}{\vert x-y\vert^{n+1}}\ri F(y)dy.
\end{align*}

\begin{Prop}\label{prop:BFactor}
The Beurling-Ahlfors operator factorises according to 
\begin{align*}
\mathcal{S}(\omega)=-\mathbf{R}\gp(\widehat{\omega} \gp \mathbf{R})=-(\mathbf{R}\gp \widehat{\omega})\gp \mathbf{R}
\end{align*}
for all $\omega\in L^p(\R^n, \Lambda \R^n)$ and $1<p<\infty$.
\end{Prop}

\begin{proof}
It is enough to establish the factorisation for $\omega\in \mathscr{S}(\R^n, \Lambda \R^n)$. Indeed, since the Schwartz space of multivector fields is dense in $L^{p}(\R^n,\Lambda \R^n)$ for $1<p<\infty$, and $\mathcal{S},\mathbf{R}\gp, \gp \mathbf{R}$ and $\mathcal{R}\gm$ are bounded on $L^{p}(\R^n,\Lambda \R^n)$, the result follows by continuity. The functions $x_\nu/\vert x\vert^{n+1}$ define tempered homogeneous distributions for each $\nu$. Hence, by acting component wise $\mathcal{R}_{\nu}\omega\in \mathscr{S}(\R^n, \Lambda \R^n)$ for every $\omega \in\mathscr{S}(\R^n, \Lambda \R^n)$ and consequently also $\mathbf{R}\gp \omega, \omega\gp \mathbf{R},\mathbf{R}\gm \omega \in \mathscr{S}(\R^n, \Lambda \R^n)$. Setting $\eta(\xi)=\widehat{\mathcal{F}(\omega)(\xi)}\gp i\xi/\vert \xi \vert$, we apply Theorem 5 in \cite[Chap. 3  p.73]{S}
first to $-\frac{i\xi}{\vert \xi \vert}\gp \eta$ and then to $\eta(\xi)$ and the result follows. The other formula follows similarly. 
\end{proof}

\begin{rem}
In the complex case we set $\mathcal{R}_\C=\mathcal{R}_1+i\mathcal{R}_2$. Then we have $\mathcal{S}_{\C}=\mathcal{R}_\C^2$ where 
\begin{align*}
\mathcal{S}_{\C}f(z)=-\frac{1}{\pi}\text{p.v.}\int_{\C}\frac{f(w)}{(z-w)^2}dA(w). 
\end{align*}
is the complex Beurling-Ahlfors transform. In $\R^2$ the even part of $\Lambda \R^2$, i.e. the linear subspace $\Lambda^0\R^2\oplus\Lambda^2\R^2$ is a subalgebra of the euclidean Clifford algebra $\Delta (\R^2)$ isomorphic to $\C$. In particular, Clifford multiplication by $e_1$ define a map $e_1: V\to \Lambda^0\R^2\oplus\Lambda^2\R^2$ where $e_{12}\cong i$. Thus, with this identification we have 
\begin{align*}
\mathcal{R}_\C\cong \gp (e_1\gp \mathcal{R})=(e_1\gp \mathcal{R})\gp 
\end{align*}
and so assuming that $\omega\in L^p(\R^2, \Lambda^0\R^2\oplus\Lambda^2\R^2)$ so $\widehat{\omega}=\omega$ we have 
\begin{align*}
\mathcal{S}_\C\omega&=(\mathcal{R}_\C\circ \mathcal{R}_\C)(\omega)=(e_1\gp \mathcal{R})\gp (\omega \gp (e_1\gp \mathcal{R}))\\
&=e_1\gp (\mathcal{R}\gp (\omega \gp e_1)\gp \mathcal{R}))\\
&=e_1\gp  (\mathcal{R}\gp (\omega \gp e_1)\gp \mathcal{R})\\
&=e_1\gp  (\mathcal{R}\gp (\widehat{\widehat{\omega \gp e_1}})\gp \mathcal{R})\\
&=e_1\mathcal{S}(\widehat{\omega \gp e_1}). 
\end{align*}
\end{rem}

In \cite{IM1} Theorem 8.12 an integral representation formula for the Beurling transform is given. Using Clifford algebra we can avoid the use of matrices and their minors altogether and present this formula in the more compact form as 
\begin{align*}
\BB F(x)&=\sum_{l=0}^n\frac{n-2l}{n}\langle F(x)\rangle_l\nonumber\\&+\frac{1}{\sigma_{n-1}}\text{p.v.}\int_{\R^n}\frac{(x-y)\gp F(y)\gp (x-y)^{-1}-\sum_{l=0}^nn^{-1}(2l-n)\langle F(y)\rangle_l}{\vert x-y\vert^n}dy
\end{align*}
for $F\in L^p(\R^n,\Lambda)$.
Using the linear operator $\mathbf{B}$ on multivectors defined by
\begin{align*}
\mathbf{B}w=\sum_{j=1}^ne_j\gp (e_j \gm w)=\sum_{j=1}^ne_j\gp \widehat{w} \gp e_j, 
\end{align*}
for any $w\in \Lambda \R^n$ we can express the Beurling-Ahlfors transform even more compactly as 
\begin{align}\label{eq:BInt}
\BB F(x)&=\mathbf{B}F(x)+\frac{1}{\sigma_{n-1}}\text{p.v.}\int_{\R^n}\frac{(x-y)\gp F(y)\gp (x-y)^{-1}-\mathbf{B}F(y)}{\vert x-y\vert^n}dy
\end{align}

\begin{rem}
The property of $\mathbf{B}$ used in the representation formula \eqref{eq:BInt} is that $\mathbf{B}w=(n-2k)w$ whenever $w\in \Lambda^k\R^n$. Moreover, $\mathbf{B}$ is related to the \emph{skew Euler operator} $\mathbf{E}\in \LL(\Lambda \R^n)$ given by 
\begin{align*}
\mathbf{E}w=\sum_{j=1}^ne_j\wedge(e_j\ri w),
\end{align*}
through the identity 
\begin{align*}
\mathbf{B}=2\mathbf{E}+nI. 
\end{align*}
By Theorem 5.5.2 in \cite{GW} $\mathbf{E}$ is the generator of 
\begin{align*}
\text{End}_{GL(\R^n)}(\Lambda \R^n)=\{L\in \LL(\Lambda \R^n): [L,\widehat{T}]=0\, \text{ for all $T\in \text{GL}(\R^n)$}\}.
\end{align*}
In addition, by Lemma 5.5.1 in \cite{GW} 
\begin{align*}
[\mathbf{E},T]=pT
\end{align*}
for all $T\in \LL(\Lambda \R^n)$ such that $T: \Lambda^{k}\R^n\to \Lambda^{k+p}\R^n$ for all $k$.
\end{rem}

\subsection{\sffamily Connection to the Cauchy Transforms on $\R^n$.}

We now come to the important connection between the Cauchy transforms and the Beurling--Ahlfors transform on $\R^n$.

\begin{Thm}\label{thm:DiracBA}
For any $\Phi\in L^2(\R^n,\Lambda)$, the weak derivatives satisfies 
\begin{align}
 \Dm \Cc^+\Phi=\BB\Phi,\,\,\, \Dp \Cc^-\Phi=\BB\Phi.
\end{align}
\end{Thm}

\begin{proof}
By definition of the Beurling--Ahlfors transform, if $\Dp F(x)=\Phi(x)$, then $\BB\Phi(x)=\Dm F(x)$. By Theorem \ref{thm:DiracC}, $F(x)=\Cc^+\Phi(x)$, and so $\Dm F(x)=\Dm \Cc^+\Phi(x)$. This proves the first identity. The second follow similarly. 
\end{proof}


\section{\sffamily Hodge-Dirac Equation with Normal and Tangential Boundary Conditions}

\begin{Def}[Hardy spaces]
Let $\Om$ be a bounded Lipschitz domain. The \emph{Hardy spaces} $\mathbb{H}_\pm^2(\Om)$ are defined by
\begin{align*}
\mathbb{H}_\pm^2(\Om):=\{F\in C^\infty(\Om,\Lambda): \mathcal{D}^\pm F(x)=0,\,\, F^\ast\in L^2(\dv \Om)\}
\end{align*}
endowed with the norm $\Vert F\Vert_{\mathbb{H}_\pm^2(\Om)}:=\Vert F^\ast\Vert_{L^2(\dv \Om)}$. Here $F^\ast$ denotes the nontangential maximal function. 
\end{Def}

It is known, e.g. \cite{GM}, that the pointwise nontangetial trace of any $F\in \mathbb{H}_\pm^2(\Om)$ exists for $\mathscr{H}^{n-1}$-a.e. on $\dv \Om$ and that there exists positive constants $c\leq C$ independent of $F$ such that 
\begin{align*}
c\Vert F\Vert_{\mathbb{H}_\pm^2(\Om)} \leq  \Vert F\Vert_{\dv \Om}\Vert_{L^2(\dv \Om,\Lambda)}\leq C\Vert F\Vert_{\mathbb{H}_\pm^2(\Om)}.
\end{align*}

In addition the Cauchy integral theorem 
\begin{align*}
\int_{\dv \Om}\overline{F(x)}\gp \nu(x)\gp G(x)d\sigma(x)=0
\end{align*}
hold for every $F,G\in \mathbb{H}_+^2(\Om)$ and analogously for $\mathbb{H}_-^2(\Om)$. Finally, Proposition 4.2 in \cite{Mar}, see also \cite{IMS}, implies that the following norms are all equivalent
\begin{align}\label{eq:EquivNorm}
\Vert F^\ast\Vert_{L^2(\dv \Om)}\approx \Vert F_T\Vert_{L^2(\dv \Om,\Lambda)}\approx \Vert F_N\Vert_{L^2(\dv \Om,\Lambda)}. 
\end{align}

\begin{Thm}[Theorem 1.7 in \cite{MMMT}]
\label{thm:HDT1}
The boundary value problem
\begin{align}\label{eq:HDT1}
\left\{
    \begin{array}{ll}
   \Dp F(x)=0 & x\in \Om,\\
       \nu(x)\wedge F(x)=f(x), & x\in \dv \Om.  
      \end{array} \right.
\end{align}
with $f\in L_N^2(\dv \Om,\Lambda)$ is solvable if and only if $f\in \{\mathcal{H}_T( \Om)\vert_{\dv \Om})\}^{\perp}=\{\omega\vert_{\dv \Om}:\omega\in \mathcal{H}_T(\dv \Om)\}^{\perp}$. 
\end{Thm}

\begin{Thm}[Theorem 1.7 in \cite{MMMT}]
The boundary value problem
\begin{align}\label{eq:HDT2}
\left\{
    \begin{array}{ll}
   \Dp F(x)=0 & x\in \Om,\\
       \nu(x)\ri F(x)=f(x), & x\in \dv \Om.  
      \end{array} \right.
\end{align}
with $f\in L_T^2(\dv \Om,\Lambda)$ is solvable if and only if $f\in \{\mathcal{H}_N( \Om)\vert_{\dv \Om})\}^{\perp}=\{\omega\vert_{\dv \Om}:\omega\in \mathcal{H}_N(\dv \Om)\}^{\perp}$. 
\end{Thm}

Since $\nu(x)\wedge F(x)=\nu(x)\wedge F_T(x)$ and $\nu(x)\ri F(x)=\nu(x)\wedge F_N(x)$ and so 
\begin{align*}
F_T(x)=\nu(x)\ri (\nu(x)\wedge F_T(x))=\nu(x)\ri f(x)\in L_T^2(\dv \Om,\Lambda),
\end{align*}
and 
\begin{align*}
F_N(x)=\nu(x)\wedge  (\nu(x)\ri F_N(x))=\nu(x)\wedge f(x)\in  L_N^2(\dv \Om,\Lambda).
\end{align*}
 
Therefore, the solvability condition that $f \in\{\mathcal{H}_T( \Om)\vert_{\dv \Om})\}^{\perp}$ is equivalent to \newline$F_T\in \nu(x)\ri \{\mathcal{H}_T( \Om)\vert_{\dv \Om})\}^{\perp}$ and the solvability condition that $f \in\{\mathcal{H}_N( \Om)\vert_{\dv \Om})\}^{\perp}$ is equivalent to $F_N\in \nu(x)\wedge \{\mathcal{H}_N( \Om)\vert_{\dv \Om})\}^{\perp}$.


It is clear that any solution of \eqref{eq:HDT1} is unique only up to addition of a field $h$ that solves 
\begin{align*}
\left\{
    \begin{array}{ll}
   \Dp h(x)=0 & x\in \Om,\\
      h_T(x)=0, & x\in \dv \Om, 
      \end{array} \right.
\end{align*}
i.e. $h\in \mathcal{H}_T(\Om)$. Similarly, 
any solution of \eqref{eq:HDT2} is unique only up to addition of a field $h$ that solves 
\begin{align*}
\left\{
    \begin{array}{ll}
   \Dp h(x)=0 & x\in \Om,\\
      h_N(x)=0, & x\in \dv \Om, 
      \end{array} \right.
\end{align*}
 We may get a unique solution if we impose the period conditions
 \begin{align*}
 \int_{\dv \Om}\langle F_N(x),h(x)\rangle d\sigma(x)=\wp_h^T 
 \end{align*}
for every $h\in \mathcal{H}_T(\dv \Om)$ and any choice of real numbers $\wp_h^T$ in \eqref{eq:HDT1}, and a similar if we impose the period conditions
\begin{align*}
 \int_{\dv \Om}\langle F_T(x),h(x)\rangle d\sigma(x)=\wp_h^N 
 \end{align*}
for every $h\in \mathcal{H}_N(\dv \Om)$ and any choice of real numbers $\wp_h^T$ in \eqref{eq:HDT2}.  Among these choices there are the distinguished choices of requiring either $F_N\perp \mathcal{H}_T(\dv \Om)$ or $F_T\perp \mathcal{H}_N(\dv \Om)$. Moreover in the case when $\Om  \cong B_1^n(0)=\{x\in \R^n: \vert x\vert<1\}$, then 
\begin{align*}
\mathcal{H}_T(\Om)=\{ce_V: c\in \R\}.
\end{align*}
Hence the condition that $F_N\perp \mathcal{H}_T(\dv \Om)$ is equivalent to requiring that $\int_{\dv \Om}\langle F_N(x),e_V\rangle d\sigma(x)=\int_{\dv \Om}\langle F(x),e_V\rangle d\sigma(x)=0$. 

Note that if $\Om$ is a $C^2$-domain and either $F_T\in W^{1/2,2}(\dv \Om,\Lambda)$ or $F_N\in W^{1/2,1}(\dv \Om,\Lambda)$, then the corresponding boundary value is \emph{D-elliptic} in the sense of Definition 3.7 in \cite{BB}. Moreover by Theorem 7.17 in \cite{BB2}, any solution $F$ with D-elliptic boundary conditions satisfy $F\in W^{1,2}(\Om,\Lambda)$. Thus in particular if $F\in \mathbb{H}_\pm^2(\Om)$ then $\mathcal{D}^\pm F\in L^2(\Om)$.

\subsection{\sffamily Tangential to Normal Hilbert Transform.}

\begin{Def}\label{def:HTan}
Let $f\in L_T^2(\dv \Om,\Lambda)$ and let $F$ be the unique solution to the half-Dirichlet boundary value problem 
\begin{align}\label{eq:BVu1}
\left\{
    \begin{array}{ll}
   \Dp F(x)=0 & x\in \Om,\\
     F_T=f   & x\in \dv\Om,\\
     f\in \nu(x)\ri \{\mathcal{H}_T( \Om)\vert_{\dv \Om})\}^{\perp}&x \in \dv\Om,\\
     F_N\perp \mathcal{H}_T(\dv \Om)& x\in \dv\Om.
      \end{array} \right.
\end{align}
The \emph{tangential to normal Hilbert transform} is (by \ref{eq:EquivNorm}) the bounded linear map $\mathcal{H}_{TN}^+: L_T^2(\dv \Om,\Lambda)\to  L_N^2(\dv \Om,\Lambda)$ which gives the normal trace of the unique solution $F$, i.e.,
\begin{align*}
\mathcal{H}_{TN}^+(f)=F_N. 
\end{align*}
Similarly, let $f\in L_N^2(\dv \Om,\Lambda)$ and let $F$ be the unique solution to the half-Dirichlet boundary value problem . 
\begin{align}\label{eq:BVu2}
\left\{
    \begin{array}{ll}
   \Dp F(x)=0 & x\in \Om,\\
       F_N=f  & x\in \dv\Om,\\
       f\in \nu(x)\wedge \{\mathcal{H}_N( \Om)\vert_{\dv \Om})\}^{\perp}& x \in \dv\Om,\\
     F_T\perp \mathcal{H}_N(\dv \Om)& x\in \dv\Om.
      \end{array} \right.
\end{align}
The \emph{normal to tangential Hilbert transform} is (by \ref{eq:EquivNorm}) the bounded linear map $\mathcal{H}_{NT}^+: L_N^2(\dv \Om,\Lambda)\to  L_T^2(\dv \Om,\Lambda)$ which gives the tangential trace of the unique solution $F$, i.e.,
\begin{align*}
\mathcal{H}_{TN}^+(f)=F_T. 
\end{align*}
The tangential to normal and normal to tangential Hilbert transforms $\mathcal{H}_{TN}^-$ and $\mathcal{H}_{NT}^-$ with respect to the Hodge-Dirac operator $\Dm$ are defined analogously.  
\end{Def}

\begin{rem}
By Theorem 7.17 in \cite{BB2}, $\mathcal{H}_{TN}^\pm: W^{1/2,2}_T(\dv \Om,\Lambda)\to W^{1/2,2}_N(\dv \Om,\Lambda)$ and $\mathcal{H}_{NT}^\pm: W^{1/2,2}_N(\dv \Om,\Lambda)\to W^{1/2,2}_T(\dv \Om,\Lambda)$ boundedly. 
\end{rem}

It follows from the equivalence of norms \eqref{eq:EquivNorm} that all the operators $\mathcal{H}_{NT}^\pm$ and $\mathcal{H}_{TN}^\pm$ are bounded.

We note that if $\Om$ is simply connected, then the condition that $F_N\perp \mathcal{H}_T(\dv \Om)$ is automatically satisfied if $F: \Om \to \oplus_{k=0}^{n-1}\Lambda \R^k$ and similarly $F_T\perp \mathcal{H}_T(\dv \Om)$ is automatically satisfied if $F: \Om \to \oplus_{k=1}^{n}\Lambda \R^k$.

Using the operator $\mathcal{H}_{TN}^+(f)$ and the Cauchy integral theorem we can express the solution of \eqref{eq:BVu1} according to 
\begin{align*}
F(x)=\int_{\dv \Om}E(y-x)\gp \nu(y)(f(y)+\mathcal{H}_{TN}^+(f)(y))d\sigma(y)
\end{align*}

In view of the discussion in Subsection \ref{sec:Cauchy}, solvability \eqref{eq:HDT1} is equivalent to the requirement that $\gamma F=F_T+F_N=f+F_N\in H^2_+(\dv \Om)$. This implies that any such $g=F_N$ is a solution of the singular integral equation 
\begin{align*}
(I-\EE^{+})(g)(x)=-(I-\EE^{+})(f)(x)
\end{align*}

In particular $\mathcal{H}_{TN}^+$ is equals $(I-\EE^{+})^{-1}\circ (-I+\EE^{+})$ on $L_N^2(\dv \Om,\Lambda)\setminus \mathcal{H}_T(\dv \Om)$. In particular it is an open problem whether we can express 
\begin{align*}
\mathcal{H}_{TN}^+=\sum_{l=0}^\infty(\EE^{+})^l(-I+\EE^{+})
\end{align*}
as a Neumann series when $\Om$ is a general simply connected Lipschitz domain. For more on singular boundary equations we refer the reader to \cite{FJR,McL,HW,HMOMPET}. 

\subsection{\sffamily Inhomogenous Hodge--Dirac Equation}

We now consider the inhomogenous Hodge--Dirac equation 
\begin{align}\label{eq:HDT3}
\left\{
    \begin{array}{ll}
   \Dp F(x)=G(x) & x\in \Om,\\
        F_T(x)=f(x) & x\in \dv \Om,  
      \end{array} \right.
\end{align}
on a bounded Lipschitz domain $\Om$ where $G\in L^2(\Om,\Lambda)$. Using the Cauchy transform $\Cc^+$, we have  $\Dp \Cc^+(\chi_\Om G)=\chi_\Om G$. Hence we define 
the new field 
\begin{align*}
H(x):=F(x)-\Cc^+(\chi_\Om G)(x).
\end{align*}
Then $\Dp H(x)=\Dp F(x)-\Dp\Cc^+(\chi_\Om G)(x)=0$ for $x\in \Om$. Hence $F$ solves \eqref{eq:HDT3} if and only if $H$ solves the homogeneous Hodge--Dirac equation 
\begin{align}\label{eq:HDT4}
\left\{
    \begin{array}{ll}
   \Dp H(x)=0 & x\in \Om,\\
        H_T(x)=f(x)-(\Cc^+\chi_\Om G)_T(x) & x\in \dv \Om.  
      \end{array} \right.
\end{align}

Using Theorem \ref{thm:HDT1}, we have proven:

\begin{Thm}\label{thm:InHDirac}
Let $\Om$ be a bounded Lipschitz domain $\Om$ and let $G\in L^2(\Om,\Lambda)$. 
\begin{align*}
\left\{
    \begin{array}{ll}
   \Dp F(x)=G(x) & x\in \Om,\\
        F_T(x)=f(x) \in L^2(\dv \Om)& x\in \dv \Om,  \\
        F_N-\gamma_N\Cc^+(\chi_\Om G)\perp \mathcal{H}_T(\Om)
      \end{array} \right.
\end{align*}
is uniquely solvable for $F\in W_{d,\delta}^{1,2}(\Om,\Lambda \R^n)$ if and only if 
\begin{align*}
f-(\Cc^+\chi_\Om G)_T\in \nu(x)\ri \{\mathcal{H}_T(\Om)\vert_{\dv \Om}\}^\perp. 
\end{align*}

\end{Thm}

\subsection{\sffamily Dirichlet Problem for Laplace Equation}

In the case of the Laplace equation with Dirichlet boundary condition, we have $A(x)\equiv I$, and so $\mathcal{M}(x)=\mathscr{C}(I)\equiv 0$. Thus, Theorem \ref{thm:Main1} shows that Laplace equation reduces to the homogeneous Hodge-Dirac equation 

\begin{align*}
\left\{
    \begin{array}{ll}
   \Dm F(x)=0 & x\in \Om,\\
        F_T(x)=\phi(x)& x\in \dv \Om.  
      \end{array} \right.
\end{align*}
which has the solution 
\begin{align*}
F(x)=\frac{1}{\sigma_{n-1}}\int_{\dv \Om}\frac{x-y}{\vert x-y\vert}\gm \nu(y)\gm (\phi(y)+\mathcal{H}_{TN}^-(\phi)(y))d\sigma(y). 
\end{align*}

Using that $u(x)=\frac{1}{2}(F(x)+\overline{F(x)})$, we get 
\begin{align*}
u(x)&=\frac{1}{2\sigma_{n-1}}\int_{\dv \Om}\bigg(\frac{x-y}{\vert x-y\vert^n}\gm \nu(y)\gm (\phi(y)+\mathcal{H}_{TN}^-(\phi)(y))+ \frac{y-x}{\vert y-x\vert^n}\gp \nu(y)\gp ((\phi(y)+\overline{\mathcal{H}_{TN}^-(\phi)(y)})\bigg)d\sigma(y).
\end{align*}
If we define 
\begin{align*}
f^+(x)&=\frac{1}{2\sigma_{n-1}}\int_{\dv \Om} \frac{y-x}{\vert y-x\vert^n}\gp \nu(y)\gp ((\phi(y)+\overline{\mathcal{H}_{TN}^-(\phi)(y)})d\sigma(y)\\
f^-(x)&=\frac{1}{2\sigma_{n-1}}\int_{\dv \Om} \frac{x-y}{\vert y-x\vert^n}\gm \nu(y)\gm ((\phi(y)+\mathcal{H}_{TN}^-(\phi)(y))d\sigma(y),
\end{align*}
we see that $u=f^++f^-$ is the sum of two fields such that $\Dp f^+=0$ and $\Dm f^-=0$. This is analogous to the two dimensional case, where every harmonic function on a simply connected domain can be written (uniquely up to a constant) as a sum of a holomorphic and anti-homorphic function.


\section{\sffamily Representation Formulas for the Solution to the Dirichlet Problem}

\subsection{\sffamily Reduction of Homogeneous Dirac--Beltrami Equation to an an Inhomogeneous Equation}\label{sec:1}

Let $\Om\subset \R^n$ be a smooth bounded simply connected domain. We consider the Dirac--Beltrami equation subject to tangential boundary conditions
\begin{align}\label{eq:DiracBeltrami}
\left\{
    \begin{array}{ll}
     \Dm F(x)=\mathcal{M}(x)\Dp F(x), & x\in \Om,\\
      F_T(x)=\phi(x), & x\in \dv \Om.  
      \end{array} \right.
\end{align}
where $\phi\in W^{1/2,2}(\dv \Om)$.
We now consider the  multivector field $H(x)\in W^{1,2}(\Om,\Lambda)\cap C^\omega(\Om,\Lambda)$ defined to be the solution of the homogeneous Hodge-Dirac equation
\begin{align*}
\left\{
    \begin{array}{ll}
     \Dm H=0 & x\in \Om,\\
      H_T(x)=\phi(x)& x\in \dv \Om,  
      \end{array} \right.
\end{align*}
and define the multivector field $G:=F-H$. Then 
\begin{align*}
\Dp G(x)&=\Dp F(x)-\Dp H(x),\\
\Dm G(x)&=\Dm F(x)-\Dm H(x)=\Dm F(x).
\end{align*}
Hence,
\begin{align*}
\Dm G(x)&=\Dm F(x)=\mathcal{M}(x)\Dp F(x)=\mathcal{M}(x)(\Dp G(x)+\Dp H(x))\\&=\mathcal{M}(x)\Dp G(x)+\mathcal{M}(x)\Dp H(x)
\end{align*}
and $G_T(x)=F_T(x)-H_T(x)=\phi(x)-\phi(x)=0$. Thus, $G$ solves the inhomogeneous Dirac--Beltrami equation
\begin{align*}
\left\{
    \begin{array}{ll}
    \Dm G(x)-\mathcal{M}(x)\Dp G(x)=\mathcal{M}(x)\Dp H(x)& x\in \Om,\\
      G_T(x)=0& x\in \dv \Om,  
      \end{array} \right.
\end{align*}

The representation formula for the monogenic field $H$ is given by 
\begin{align*}
H=\mathcal{E}^-\circ (I+\mathcal{H}_{TN}^-)\phi
\end{align*}

\subsection{\sffamily  Tangential and Normal Beurling--Ahlfors Operators}

\begin{Def}\label{def:BTan}
Let $\Om\subset \R^n$ be a bounded Lipschitz domain. Let $G\in L^2(\Om,\Lambda \R^n)$. Consider a solution $F$ of the inhomogeneous Hodge-Dirac equation with tangential boundary condition equal to zero
\begin{align}\label{def:tBA}
\left\{
    \begin{array}{ll}
    \Dm F(x)=G(x)& x\in \Om,\\
      F_T(x)=0& x\in \dv \Om.
     \end{array} \right.
\end{align}
\emph{The tangential Beurling--Ahlfors Operator} $\BB_T$ on $\Om$ is defined according to the rule 
\begin{align*}
\BB_T(G)=\Dp F. 
\end{align*}
 Consider a solution $F$ the inhomogeneous Hodge-Dirac equation with normal boundary condition equal to zero
\begin{align}\label{def:nBA}
\left\{
    \begin{array}{ll}
    \Dm F(x)=G(x)& x\in \Om,\\
      F_N(x)=0& x\in \dv \Om.
            \end{array} \right.
\end{align}
\emph{The normal Beurling--Ahlfors Operator} $\BB_N$ on $\Om$ is defined according to the rule 
\begin{align*}
\BB_N(G)=\Dp F. 
\end{align*}
\end{Def}

\begin{Prop}\label{prop:BAIso}
The tangential and normal Beurling--Ahlfors transforms are isometries on $L^2(\Om, \Lambda \R^n)$. 
\end{Prop}

\begin{proof}
Let $G\in L^2(\Om, \Lambda \R^n)$ and let $F$ as in \eqref{def:tBA}. Then 
\begin{align*}
\Vert G\Vert^2_{L^2(\Om, \Lambda \R^n)}&=\int_{\Om}\langle G(x),G(x)\rangle dx=\int_{\Om}\langle \Dm F(x),\Dm F(x)\rangle dx=\int_{\Om}\langle (d-\delta) F(x),(d-\delta) F(x)\rangle dx\\
&=\int_{\Om} \Big\{\vert d F(x)\vert^2+\vert \delta F(x)\vert^2 -2\langle d F(x),\delta F(x)\rangle \Big\}dx=\{\text{by \eqref{eq:IntPart1}}\}\\
&=\int_{\Om} (\vert d F(x)\vert^2+\vert \delta F(x)\vert^2)dx -2\int_{\dv \Om} \langle \nu(x)\ri F(x),\delta F(x)\rangle d\sigma(x)\\
&=\int_{\Om} (\vert d F(x)\vert^2+\vert \delta F(x)\vert^2)dx
\end{align*}
since $F_T(x)=0$ implies that $\nu(x)\ri F(x)=0$ $\sigma$-a.e. $x\in \dv \Om$. On the other hand by definition of $\BB_T$
\begin{align*}
\Vert \BB_T(G)\Vert^2_{L^2(\Om, \Lambda \R^n)}&=\int_{\Om}\langle \Dp F(x),\Dp F(x)\rangle dx=\int_{\Om}\langle (d+\delta) F(x),(d+\delta) F(x)\rangle dx\\
&=\int_{\Om} \Big\{\vert d F(x)\vert^2+\vert \delta F(x)\vert^2 +2\langle d F(x),\delta F(x)\rangle \Big\}dx\\
&=\int_{\Om} (\vert d F(x)\vert^2+\vert \delta F(x)\vert^2)dx.
\end{align*}
Hence $\Vert \BB_T(G)\Vert_{L^2(\Om, \Lambda \R^n)}=\Vert G\Vert_{L^2(\Om, \Lambda \R^n)}$. The proof for $\BB_N$ is similar using the integration by formula \eqref{eq:IntPart2} instead and that $F_N=0$ implies that $\nu(x)\wedge F(x)=0$ $\sigma$-a.e. $x\in \dv \Om$. 
\end{proof}

\begin{Prop}
The tangential and normal Beurling--Ahlfors transforms have the following integral representations
\begin{align*}
\BB_T(G)&=\BB(G\chi_\Om)-\Dp\circ \mathcal{E}^-\circ (I-\mathcal{H}_{TN}^-)\circ \gamma_T\circ \Cc^-(G\chi_\Om),\\
\BB_N(G)&=\BB(G\chi_\Om)-\Dp\circ \mathcal{E}^-\circ (I-\mathcal{H}_{NT}^-)\circ \gamma_N\circ \Cc^-(G\chi_\Om). 
\end{align*}
where $\BB$ is the Beurling-Ahlfors operator on $\R^n$. 
\end{Prop}

Observe that $\text{supp}(\BB_T(G))\nsubseteq \Om$ in general.

\begin{proof}
We only consider the tangential Beurling--Ahlfors transform. The proof for the normal Beurling--Ahlfors transform is analogous. Using the Cauchy transform $\Cc^-$ we reduce \eqref{def:tBA} to a homogeneous equation be considering $\tilde{F}=F-\Cc^-(G\chi_\Om)$. Then $\tilde{F}$ solves 
\begin{align}\label{def:tBA}
\left\{
    \begin{array}{ll}
    \Dm \tilde{F}(x)=0& x\in \Om,\\
      \gamma_T\tilde{F}(x)=-\gamma_T\Cc^-(G\chi_\Om)& x\in \dv \Om.
      \end{array} \right.
\end{align}
and has the representation $\tilde{F}(x)=-\mathcal{E}^-\circ (I-\mathcal{H}_{TN}^-)\circ \gamma_T\circ \Cc^-(G\chi_\Om)(x)$. Hence 
\begin{align*}
F(x)&=\tilde{F}(x)+\Cc^-(G\chi_\Om)(x)\\
&=-\mathcal{E}^-\circ (I-\mathcal{H}_{TN}^-)\circ \gamma_T\circ \Cc^-(G\chi_\Om)(x)+\Cc^-(G\chi_\Om)(x),
\end{align*}
which gives 
\begin{align*}
\BB_T(G)(x)&=\Dp F(x)=\Dp\Cc^-(G\chi_\Om)(x)-\Dp\circ \mathcal{E}^-\circ (I-\mathcal{H}_{TN}^-)\circ \gamma_T\circ \Cc^-(G\chi_\Om)(x)\\
&=\BB(G\chi_\Om)(x)-\Dp\circ \mathcal{E}^-\circ (I-\mathcal{H}_{TN}^-)\circ \gamma_T\circ \Cc^-(G\chi_\Om)(x).
\end{align*}
We observe that the boundary term $-\Dp\circ \mathcal{E}^-\circ (I-\mathcal{H}_{TN}^-)\circ \gamma_T\circ \Cc^-(G\chi_\Om)(x)$ is a smooth (in fact real analytic) multivector field in $\Om$. 
\end{proof}

\subsection{\sffamily  Reduction of Inhomogeneous Dirac--Beltrami Equation to an Inhomogeneous Hodge--Dirac Equation}

Since $\mathcal{M} (\Dp G+\Dp H)\in L^2$, we have using the tangential Beurling--Ahlfors transform we have 
\begin{align*}
\Dp G(x)=\BB_T(\Dm G)=\BB_T(\mathcal{M} (\Dp G+\Dp H)=\BB_T\circ \mathcal{M} \Dp G+\BB_T\circ \mathcal{M}\Dp H. 
\end{align*}
or equivalently 
\begin{align*}
(I-\BB_T\circ \mathcal{M})(\Dp G)=\BB_T\mathcal{M}(\Dp H). 
\end{align*}
Since $\sup_{x\in \Om} \Vert \mathcal{M}(x)\Vert<1$ by Lemma \ref{lem:EllipticityBound}, and $\Vert S_T\Vert=1$ on $L^2(\Om,\Lambda \R^n)$ by Proposition \ref{prop:BAIso}, we have 
\begin{align*}
\Vert \BB_T\circ \mathcal{M}\Vert_T\leq \Vert S_T\Vert \Vert \mathcal{M}\Vert= \Vert \Vert \mathcal{M}\Vert\Vert_{L^{\infty}(\Om}=M<1. 
\end{align*}
Hence the Neumann series 
\begin{align*}
(I-\BB_T\circ \mathcal{M})^{-1}=\sum_{k=0}^{\infty}(\BB_T\circ\mathcal{M})^k
\end{align*}
converges and the operator $I-\BB_T\circ \mathcal{M}$ is invertible on $L^2(\Om,\Lambda \R^n)$. Thus $G$ solves the inhomogeneous Hodge-Dirac equation 
\begin{align}\label{eq:InHD}
\left\{
    \begin{array}{ll}
     \Dp G(x)=\Phi(x)& x\in \Om,\\
      G_T(x)=0 & x\in \dv \Om.  
      \end{array} \right.
\end{align}
where $\Phi(x)=(I-\BB_T\circ \mathcal{M})^{-1}(\BB_T\mathcal{M}(\Dp H))(x)$. We define the operator 
\begin{align}\label{def:OpB}
\Pi_{\mathcal{M}}:=(I-\BB_T\circ \mathcal{M})^{-1}\circ \BB_T\mathcal{M},
\end{align}
so that $\Phi=\Pi_{\mathcal{M}}H$.

\subsection{\sffamily  Tangential Cauchy Transform}\label{sec:2}

\begin{Def}\label{def:CTan}
Let $\Om$ be a bounded Lipschitz domain in $\R^n$ and consider a solution $F$ of the inhomogeneous Hodge-Dirac equation 
\begin{align}\label{eq:InHD}
\left\{
    \begin{array}{ll}
     \Dp F(x)=\Phi(x)& x\in \Om\\
      F_T(x)=0 & x\in \dv \Om\\
      \gamma_N\Cc^+(\chi_\Om \Phi)\perp \mathcal{H}_T(\Om),&x\in \dv \Om
      \end{array} \right.
\end{align}
with $\Phi\in L^2(\Om,\Lambda \R^n)$. The \emph{tangential Cauchy transform} with respect to the domain $\Om$ is defined to be
\begin{align*}
\mathcal{C}_T^+(\Phi)(x):=F(x).
\end{align*}
Similarly, let $\Om$ be a bounded Lipschitz domain in $\R^n$ and consider a solution $F$ of the inhomogeneous Hodge-Dirac equation 
\begin{align}\label{eq:InHD2}
\left\{
    \begin{array}{ll}
     \Dm F(x)=\Phi(x)& x\in \Om\\
      F_N(x)=0 & x\in \dv \Om\\
     \gamma_T\Cc^-(\chi_\Om \Phi)\perp \mathcal{H}_N(\Om).
      \end{array} \right.
\end{align}
The \emph{normal Cauchy transform} $\Cc_T^-$ of $\Phi\in L^2(\Om,\Lambda\R^n)$ with respect to the domain $\Om$ is defined to be
\begin{align*}
\mathcal{C}_N^+(\Phi)(x):=F(x).
\end{align*}

\end{Def}

It follows immediately from the definitions that we have the relation
\begin{align*}
\BB_T(\Phi)(x)=\Dp \Cc_T^-(\Phi)(x). 
\end{align*}

\begin{Prop}
The tangential Cauchy transform has the integral representation 
\begin{align}
\mathcal{C}_T^+(\Phi)(x)=\mathcal{C}^+(\chi_\Om\Phi)(x)-\int_{\dv \Om}E(y-x)\gp \nu(y)(\gamma_T\Cc^+(\chi_\Om\Phi)(y)+\mathcal{H}_{TN}^+\gamma_T\Cc^+(\chi_\Om\Phi)(y))d\sigma(y). 
\end{align}
\end{Prop}

\begin{proof}
Set $G=F-\mathcal{C}^+(\chi_\Om\Phi)(x)$. Then $G$ solves the homogeneous Hodge-Dirac equation 
\begin{align}\label{eq:HDirac1}
\left\{
    \begin{array}{ll}
     \Dp G(x)=0& x\in \Om,\\
     G_T(x)= -\gamma_T\Cc^+(\chi_\Om\Phi)(x) & x\in \dv \Om,\\
      \gamma_N\Cc^+(\chi_\Om \Phi)\perp \mathcal{H}_T(\Om),& x\in \dv \Om.
      \end{array} \right.
\end{align}
By the definition of the tangential to normal Hilbert transform $G$ is given by 
\begin{align*}
G(x)=-\int_{\dv \Om}E(y-x)\gp \nu(y)(\gamma_T\Cc^+(\chi_\Om\Phi)(y)+\mathcal{H}_{TN}^+\gamma_T\Cc^+(\chi_\Om\Phi)(y))d\sigma(y)
\end{align*}
and hence $F$ is given by 
\begin{align*}
F(x)=\mathcal{C}^+(\chi_\Om\Phi)(x)-\int_{\dv \Om}E(y-x)\gp \nu(y)(\gamma_T\Cc^+(\chi_\Om\Phi)(y)+\mathcal{H}_{TN}^+\gamma_T\Cc^+(\chi_\Om\Phi)(y))d\sigma(y).
\end{align*}
\end{proof}

\subsection{\sffamily Representation Formula for the Solution of the Dirichlet Problem}

Combining the reduction steps in Subsections \ref{sec:1}-\ref{sec:2} we can now give a representation formula for the solution of the Dirichlet problem \eqref{eq:DirPro1} when $\Om$ is a simply connected Lipschitz domain. We have 
\begin{align*}
F&=G+H=\mathcal{C}_T^+\circ \Pi_{\mathcal{M}}\circ \mathcal{E}^-\circ (I+\mathcal{H}_{TN}^-)\phi+\mathcal{E}^-\circ (I+\mathcal{H}_{TN}^-)\phi\\
&=(\mathcal{C}_T^+\circ \Pi_{\mathcal{M}}\Dp H)(x)+H(x).
\end{align*}
Moreover, since \eqref{eq:HDirac1} is uniquely solvable by Theorem 1.7 in \cite{MMMT}, the reduction of \eqref{eq:DiracBeltrami} to  \eqref{eq:HDirac1} shows that \eqref{eq:DiracBeltrami} is uniquely solvable as well. This proves Theorem \ref{thm:Solvability}. 
To conclude we have proven that unique solution $F$ to the half-Dirichlet boundary value problem 
\begin{align*}
\left\{
    \begin{array}{ll}
    \Dm F(x)=\mathcal{M}(x)\Dm F(x), & x\in \Om,\\
      \gamma_TF(x)=\phi(x), & x\in \dv \Om.  
      \end{array} \right.
\end{align*}
is given by 
\begin{align}\label{eq:RepForm}
F(x)=(\mathcal{C}_T^+\circ \Pi_{\mathcal{M}}\Dp H)(x)+H(x)
\end{align}
where we recall that 
\begin{align*}
\Pi_{\mathcal{M}}(f)(x)&=(I-\BB_T\circ \mathcal{M})^{-1}\circ \BB_T\mathcal{M}(f)(x)
\end{align*}
and 
\begin{align*}
H(x)=\mathcal{E}^-\circ (I+\mathcal{H}_{TN}^-)\phi. 
\end{align*}
In particular, this also gives a representation formula for the solution of the Dirichlet problem 
\begin{align*}
\left\{
    \begin{array}{ll}
    \text{div}\, A(x)\nabla u(x)=0, & x\in \Om,\\
      \gamma u(x)=\phi(x), & x\in \dv \Om.  
      \end{array} \right.
\end{align*}
through 
\begin{align*}
u(x)=\langle F(x)\rangle_0=\frac{1}{2}(F(x)+\overline{F(x)})
\end{align*}

This completes the proof of Theorem \ref{thm:RepSol}. 

\begin{rem}
One should note that the only operator in the representation formula \eqref{eq:RepForm} that depends on the coefficient matrix $A(x)$ is the operator $\Pi_{\mathcal{M}}$. All the other operators are either independent of both the domain $\Om$ and $A(x)$, such as $\Cc^+$ and $\mathcal{E}^+$, or only depends on the geometry of $\Om$ such as $\mathcal{H}_{TN}^\pm$. Thus, all information about $A(x)$ is encoded in the operator $\Pi_{\mathcal{M}}$. 
\end{rem}

If we define the operator 
\begin{align*}
\mathcal{T}:=\mathcal{C}_T^+\circ (I-\BB_T\circ \mathcal{M})^{-1}\circ \BB_T\mathcal{M}\circ \Dp
\end{align*}
then the representation formula \eqref{eq:RepForm} can be given the more suggestive form 
\begin{align}\label{eq:Rep2}
F=(I+\mathcal{T})H. 
\end{align}

Since $H$ solves $\Dm H=0$ and solutions of the homogeneous Hodge--Dirac equation have many properties in common with holomorphic functions, we can view \eqref{eq:Rep2} as non-local deformation of $H$. This is somewhat reminiscent of the Stoilow factorization in two dimensions where every solution $f\in W^{1,2}_{loc}(\Om,\C)$ of the Beltrami equation 
\begin{align*}
\overline{\dv}f(z)=\mu(z)\dv f(z)
\end{align*}
for some $\mu\in L^\infty(\Om,\C)$ such that $\Vert \mu\Vert_\infty=k<1$ admits a factorization 
\begin{align*}
f=h\circ g,
\end{align*}
where $h$ is a holomorphic function and $g\in W^{1,1}_{loc}(\Om)$ is a fixed homeomorphic solution of the Beltrami equation, see Theorem 5.5.1. in \cite{AIM}. Of course the Stoilow factorization implies much stronger consequence for the solutions and can be viewed as local deformation of holomorphic functions.

We now generalize the representation formula \eqref{eq:RepForm} to multiply connected domains and inhomogeneous equations. 
\begin{Thm}\label{thm:RepIn}
The solution $F$ of \eqref{eq:MultiDB} has the representation formula
\begin{align}\label{eq:RepMulti}
F=H+\Cc_T^+(I-\BB_T\mathcal{M})^{-1}\BB_T\mathcal{M}\Dp H+\Cc_T^+(I-\BB_T\mathcal{M})^{-1}\BB_T(I+\mathcal{M})\omega_T.
\end{align}
The solution $F$ of \eqref{eq:DBIn} has the representation formula
\begin{align}\label{eq:RepIn}
F=H+\Cc_T^+(I-\BB_T\mathcal{M})^{-1}\BB_T\mathcal{M}\Dp H+\Cc_T^+(I-\BB_T\mathcal{M})^{-1}\BB_T(I+\mathcal{M})G.
\end{align}
\end{Thm}

The proof of the representation formulas \eqref{eq:RepMulti}-\eqref{eq:RepIn} are derived in exactly the same way as \eqref{eq:RepForm} and are therefore omitted.


\subsection{\sffamily Representation formula for solutions in the non-linear uniformly elliptic case}

In the nonlinear case we again prove solvability and representation formulas in a more general setting for fully non-linear Dirac--Beltrami equations which takes values in the the full De Rham complex and need not be derived from an associated second order uniformly quasilinear elliptic equation. 

\begin{Thm}\label{thm:NonLinDBRep}
Let $\Om\subset \R^n$ be a bounded simply connected Lipschitz domain and assume that $\mathcal{M}:\Om\times \Lambda \R^n \to \Lambda \R^n$ satisfy the following.
\begin{itemize}
\item[(i)] $\mathcal{M}(x,0)=0$ for all $x\in \Om$
\item[(ii)] $x\mapsto \mathcal{M}(x,w)$ is measurable in $x$ for all $w\in \Lambda \R^n$. 
\item[(iii)] $\mathcal{M}(x,\cdot)$ is continuous $k$-Lipschitz continuous for some $0\leq k<1$ for a.e. $x\in \Om$.
\end{itemize} 
Then the boundary value problem for the fully non-linear Dirac--Beltrami equation 
\begin{align}\label{eq:BVNonDB}
\left\{
    \begin{array}{ll}
\Dm F(x)=\mathcal{M}(x,\Dp F(x)) &  x\in \Om,\\
          F_T(x)=\phi(x) & x\in \dv \Om
      \end{array} \right. .
\end{align}
with $\phi\in W^{1/2,2}(\dv \Om, \Lambda)$ is uniquely solvable in $W^{1,2}_{d,\delta}(\Om,\Lambda)$ up to addition of constants. Moreover the solution is given by the representation formula 
\begin{align}
F=H+\Cc_T^+(I-\BB_T\mathcal{M})^{-1}(-\BB)_T\Dm H, 
\end{align}
where $H$ is a solution to 
\begin{align*}
\left\{
    \begin{array}{ll}
\Dp H(x)=0 &  x\in \Om,\\
          H_T(x)=\phi(x) & x\in \dv \Om
      \end{array} \right. .
\end{align*}
\end{Thm}

\begin{proof}
We take a slightly different route in the nonlinear case. Let $H\in W^{1,2}(\Om)$ be a solution of 
\begin{align*}
\left\{
    \begin{array}{ll}
\Dp H(x)=0 &  x\in \Om,\\
          H_T(x)=\phi(x) & x\in \dv \Om
      \end{array} \right. .
\end{align*}
Define $G=F-H$. Then 
\begin{align*}
\Dm G(x)=\Dm F(x)-\Dm H(x)=\mathcal{M}(x,\Dp F(x)) -\Dm H(x)=\mathcal{M}(x,\Dp G(x)) -\Dm H(x).
\end{align*}
Consequently $G$ solves the boundary value problem 
\begin{align*}
\left\{
    \begin{array}{ll}
\Dm G(x)=\mathcal{M}(x,\Dp G(x))-\Dm H(x) &  x\in \Om,\\
          G_T(x)=0 & x\in \dv \Om
      \end{array} \right. .
\end{align*}
if and only if $F$ solves \eqref{eq:BVNonDB}. Moreover, we note that both the left and right hand side lie in $L^2$. We may therefore apply the tangential Beurling--Ahlfors transform $\BB_T$ to both sides and get 
\begin{align*}
\left\{
    \begin{array}{ll}
\Dp G(x)=\BB_T\mathcal{M}(x,\Dp G(x))-\BB_T\Dm H(x) &  x\in \Om,\\
          G_T(x)=0 & x\in \dv \Om
      \end{array} \right. .
\end{align*}
Consider the operator $T:L^2(\Om,\Lambda)\to L^2(\Om,\Lambda)$ defined by 
\begin{align*}
T(\Phi)(x):=\BB_T\mathcal{M}(x,\Phi)(x)
\end{align*}
fo $\Phi\in L^2(\Om,\Lambda)$. By assumption on $\mathcal{M}$ and the fact that $\BB_T$ is an isometry on $L^2$ we get for all $\Phi_1,\Phi_2\in L^2(\Om,\Lambda)$
\begin{align*}
\Vert T(\Phi_1)-T(\Phi_2)\Vert^2_{L^2}&=\int_{\Om}\vert \BB_T(\mathcal{M}(x,\Phi_1(x))-\mathcal{M}(x,\Phi_2(x))\vert^2dx\\&=\int_{\Om}\vert \mathcal{M}(x,\Phi_1(x))-\mathcal{M}(x,\Phi_2(x))\vert^2dx\leq  k^2\int_{\Om}\vert \Phi_1(x)-\Phi_2(x)\vert^2dx\\&=k^2\Vert \Phi_1-\Phi_2\Vert^2_{L^2}.
\end{align*}
Thus $T$ is a $k$-contraction on $L^2(\Om,\Lambda)$. It then follows from Banach's fixed point theorem that the operator $I-T:L^2(\Om,\Lambda)\to L^2(\Om,\Lambda)$ is a homeomorphism (\cite[Thm. 2.1 p. 11]{DuGr}). Thus $(I-T)^{-1}$ exists and is continuous. Moreover, by induction 
\begin{align*}
\Vert T^n(\Phi)\Vert^2_{L^2}\leq k^n\Vert \Phi\Vert_{L^2}^2. 
\end{align*}
Hence $\lim_{n\to \infty}T^n\Phi=0$ for every $\Phi\in L^2$. Thus $0$ is the unique fixed point of $T$. This implies that the equation 
\begin{align*}
(I-T)\Phi=\Psi
\end{align*}
for $\Phi$ has the unique solution given by the limit of iterations given by the method of successive approximations generalizing the Neumann series in the linear case. More precisely if we let $\Phi_0=0$, we set  
\begin{align*}
\Phi_1&=T(\Phi_0)+\Psi=\Psi,\\
\Phi_2&=T(\Phi_1)+\Psi=T(\Psi)+\Psi\\
&\vdots\\
\Phi_n&=T(\Phi_{n-1})+\Psi=T(T(\Phi_{n-2})+\Psi)+\Psi=...,
\end{align*}
and 
\begin{align*}
\lim_{n\to \infty}\Phi_n=\Phi. 
\end{align*}
Thus, we see that $G$ solves 
\begin{align*}
\left\{
    \begin{array}{ll}
\Dp G(x)=(I-\BB_T\mathcal{M})^{-1}(-\BB_T)\Dm H(x) &  x\in \Om,\\
          G_T(x)=0 & x\in \dv \Om
      \end{array} \right. .
\end{align*}
From here on we can argue as in the linear case. 
\end{proof}


\section{\sffamily Local Estimates}

\subsection{\sffamily Higher Integrability and Caccioppoli Inequalities}

\begin{Def}[Critical exponents]
Let $p_M$ be defined so that 
\begin{align}\label{def:CriticalP}
M\Vert \BB\Vert_{L^{p_M}(\R^n,\Lambda^{(1,3)})\R^n}=1
\end{align}
We call $p_M$ the critical exponent with respect to $M$. 
\end{Def}

If Iwaniec's conjecture is true, i.e. that 
\begin{align*}
\Vert \BB\Vert_{L^p(\R^n,\Lambda \R^n)}=\max\{(p-1)^{-1},p-1\},
\end{align*}
the the critical exponent is equal to 
\begin{align*}
p_M=1+\frac{1}{M}. 
\end{align*}

We now recall the $L^p$-Gaffney inequality. 
\begin{Thm}[Proposition 4.3 in \cite{Scott}]
Let $F\in C_0^{\infty}(\R^n,\Lambda \R^n)$. Then for all $1<p<n$ there exists a constant $C=C(n,p)$ such that
\begin{align}\label{eq:LpGaff}
\int_{\R^n}\vert \nabla \otimes F(x)\vert^pdx\leq C\int_{\R^n}(\vert dF(x)\vert^p+\vert \delta F(x)\vert^p)dx 
\end{align}
Equivalently, there exists a constant $C'=C'(n,p)$ such that 
\begin{align}\label{eq:LpGaff}
\int_{\R^n}\vert \nabla \otimes F(x)\vert^pdx\leq C'\int_{\R^n}(\vert \Dp F(x)\vert^p+\vert \Dm F(x)\vert^p)dx 
\end{align}
\end{Thm}
We note that since $C^{\infty}_0(\Om,\Lambda \R^n)$ is dense in $W^{1,p}(\R^n,\Lambda \R^n)$, if follows from an approximation argument that \eqref{eq:LpGaff} holds in $W^{1,p}(\R^n,\Lambda\R^n)$ as well. Moreover, if the left hand side of \eqref{eq:LpGaff} is finite so is the right hand side.

\begin{Lem}\label{lem:LocaId}
Let $F\in W^{1,2}_{d,\delta}(\Om,\Lambda^{(0,2)}\R^n)$ be a solution of $\Dm F=\mathcal{M}\Dp F$. Furthermore, let $\mathcal{M}_0$ denote the extension by zero of $\mathcal{M}$ from $\Om$ to all of $\R^n$ and let $\sigma_{\mathcal{M}}(\xi,x)$ denote the symbol of $\mathcal{P}_{\mathcal{M}}=\Dm-\mathcal{M}\Dp$ acting point-wise on multivectors $w\in \Lambda\R^n$ through 
\begin{align*}
\sigma_{M}(\xi,x)w=\xi \gm w-\mathcal{M}(x)\xi\gp w,
\end{align*}
for any $\xi \in \R^n$. Then for any $\eta\in C^\infty_0(\Om)$, $\eta F$ solves the equation 
\begin{align}\label{eq:LocalDB}
\mathcal{P}_{\mathcal{M}}(\eta(x)F(x))=\sigma_{\mathcal{M}}(\nabla \eta(x),x)F(x)
\end{align}
in $\R^n$.
Furthermore, if we let $\mathcal{M}_0$ denote the extension of $\mathcal{M}$ by zero from $\Om$ to $\R^n$ we have the recursion identity 
\begin{align}\label{eq:Recursion}
\eta(x)F(x)=(\Cc^-(I-\mathcal{M}_0(x)\BB)^{-1}\sigma_{\mathcal{M}}(\nabla \eta,x)F)(x)
\end{align}
and the identities
\begin{align}
\label{eq:RecursionD1}
&\eta(x)\Dp F(x)=\BB(I-\mathcal{M}_0(x)\BB)^{-1}\sigma_{\mathcal{M}}(\nabla \eta,x)F)(x)-\nabla \eta(x)\gp F(x),\\
\label{eq:RecursionD2}
&\eta(x)\Dm F(x)=(I-\mathcal{M}_0(x)\BB)^{-1}\sigma_{\mathcal{M}}(\nabla \eta,x)F)(x)-\nabla \eta(x)\gm F(x).
\end{align}
\end{Lem}

\begin{proof}
Set $\Psi=\eta F$. We have
\begin{align}\label{eq:LocalD1}
\Dp \Psi(x)&=\Dp(\eta(x) F(x))=\nabla \eta(x)\gp F(x)+\eta(x) \Dp F(x),\\
\label{eq:LocalD2}
\Dm \Psi(x)&=\Dm(\eta(x) F(x))=\nabla \eta(x)\gm F(x)+\eta(x) \Dm F(x).
\end{align}
Hence 
\begin{align*}
\mathcal{M}\Dp \Psi(x)&=\mathcal{M}(x)(\nabla \eta(x)\gp F(x))+\eta(x) \mathcal{M}(x)\Dp F(x)
\end{align*}
Thus $\Psi$ solves the inhomogeneous Hodge-Dirac equation 
\begin{align*}
\Dm \Psi(x)-\mathcal{M}(x)\Dp \Psi(x)&=\eta(x)(\Dm F(x)-\mathcal{M}(x)\Dp F(x))+\nabla \eta(x)\gm F(x)-\mathcal{M}(x)\nabla \eta(x)\gm F(x)\\
&=\sigma_{M}(\nabla \eta(x),x)F(x)
\end{align*}
in $\R^n$.
Using the Beurling-Ahlfors transform in $\R^n$ we find
\begin{align}\label{eq:Cap1}
\Dm \Psi(x)&=(I-\mathcal{M}_0 \BB)^{-1}\sigma_{\mathcal{M}}(\nabla \eta,x)F(x),\\
\label{eq:Cap2}
\Dp \Psi(x)&=\BB \circ (I-\mathcal{M}_0(x) \BB)^{-1}\sigma_{\mathcal{M}}(\nabla \eta,x)F(x). 
\end{align}
Taking the Cauchy transform gives the recursion identity
\begin{align*}
\eta(x)F(x)=\Psi(x)=\Cc^-(I-\mathcal{M}_0(x))^{-1}\sigma_{\mathcal{M}}(\nabla \eta,x)F(x).
\end{align*}
Finally using the \eqref{eq:Cap1}-\eqref{eq:Cap2} together \eqref{eq:LocalD1}-\eqref{eq:LocalD2} gives the identities \eqref{eq:RecursionD1}-\eqref{eq:RecursionD2} and the lemma is proved. 
 \end{proof}

Using Lemma \ref{lem:LocaId} we get the following version of Meyers' higher integrability theorem.

\begin{Thm}\label{thm:Meyers}
Let $F\in W^{1,2}_{d,\delta}(\Om,\Lambda^{(0,2)} \R^n)$ solve the Dirac--Beltrami equation $\Dm F=\mathcal{M}(x)\Dp F$ in $\Om$. Then for every $p$ such that $p\leq 2^\ast=2n/(n-2)$ and $2\leq p<p_M$, $F\in W^{1,p}_{loc}(\Om,\Lambda \R^n)$ and for every $\eta\in C^{\infty}_0(\Om)$ we have the Caccioppoli inequality 
\begin{align*}
\int_{\R^n}\vert \eta(x)\vert^{p}\vert \nabla \otimes F(x)\vert^{p}dx\leq C\int_{\R^n}\vert \nabla \eta(x)\vert^{p}\vert F(x)\vert^{p}dx.
\end{align*}
for some constant $C$ depending only on $n,p, \Vert \BB\Vert_{p}$ and $M$. 
\end{Thm}

\begin{proof}
By the Sobolev embedding theorem, $F\in L^{p}(\R^n,\Lambda \R^n)$ for all $1< p\leq 2^\ast=2n(n-2)$. Taking any $p$ such that $p\leq 2^\ast=2n/(n-2)$ and $2\leq p<p_M$. Then by the identities \eqref{eq:RecursionD1}-\eqref{eq:RecursionD2} of Lemma \ref{lem:LocaId}

\begin{align*}
&\int_{\Om}\vert \eta(x)\vert^p (\vert \Dp F(x) \vert^p+\vert \Dm F(x) \vert^p)dx\\
&\leq 2^{p-1}(\Vert (I-\mathcal{M}_0 \BB)^{-1}\Vert^p_p +\Vert \BB(I-\mathcal{M}_0 \BB)^{-1}\Vert^p_p)\int_{\Om}\vert \sigma_{\mathcal{M}}(\nabla \eta,x)F(x)\vert^pdx\\
&+2^{p-1}\int_{\Om}\vert\nabla \eta(x)\vert^p\vert F(x)\vert^pdx\\
&\leq C(\Vert\vert\nabla \eta(x)\vert^p\Vert_\infty\int_{\text{supp}(\nabla \eta)}\vert F(x)\vert^pdx
\end{align*}
since $\Vert (I-\mathcal{M}_0 \BB)^{-1}\Vert_p<+\infty$ by the assumption on $p$ for some constant $C=C(n,p,\Vert S\Vert_p,M)$. 

A similar argument shows that for $\Psi=\eta F$,
\begin{align*}
&\int_{\Om}(\vert  \Dp \Psi(x) \vert^p+\vert \Dm \Psi(x) \vert^p)dx\leq C'(\Vert\vert\nabla \eta(x)\vert^p\Vert_\infty\int_{\text{supp}(\nabla \eta)}\vert F(x)\vert^pdx
\end{align*}
for some constant $C'$. 
Thus, by the $L^p$-Gaffney and the inequality above we have 
\begin{align*}
\int_{\R^n}\vert \nabla \otimes \Psi(x)\vert^{p}dx\leq C''\int_{\R^n}\vert \nabla \eta(x)\vert^{p}\vert F(x)\vert^{p}dx
\end{align*}
for some constant depending $C''=C''(p,n,\Vert \BB\Vert_p,M)$. Hence $\Psi\in W^{1,p}_0(\R^n,\Lambda \R^n)$. Using that 
\begin{align*}
\nabla \otimes (\eta(x) F(x))=\nabla \eta(x)\otimes F(x)+\eta(x)\nabla \otimes F(x), 
\end{align*}
we find
\begin{align*}
\int_{\R^n}\vert \eta(x)\vert^{p}\vert \nabla \otimes F(x)\vert^{p}dx\leq (C''+1)\int_{\R^n}\vert \nabla \eta(x)\vert^{p}\vert F(x)\vert^{p}dx.
\end{align*}

\end{proof}

\begin{rem}
That solution $u$ of \eqref{eq:DirPro1} has a higher integrability is well-known and goes back to Meyers in \cite{Mey}. The proof here follows the proof of Theorem 5.4.2 in \cite{AIM} of the Beltrami equation in the plane, originally due to Bojarski. What is perhaps a bit more surprising is that the $L^p$-norms of the Beurling-Ahlfors transform enters into the assumptions also in higher dimension. 
\end{rem}


\subsection{\sffamily Hölder Regularity}

The proof that weak solutions of 
\begin{align}\label{eq:WeakS}
\text{div}\,A(x)\nabla u(x)=0
\end{align}
are Hölder continuous goes back to the classical work of E. De Giorgi in \cite{DeG} using Caccioppoli inequalities and J. Nash in \cite{Nash} using parabolic equations. Later the same result was also proved using Harnack inequalities by J. Moser in \cite{Moser}. Here of course as usual this result holds true under the uniform ellipticity assumption  
\begin{align*}
\lambda \vert v\vert^2\leq \langle A(x)v,v\rangle\leq \Lambda \vert v\vert^2
\end{align*}
for all $v\in \R^n$ and $x\in \Om$ for some domain $\Om \subset \R^n$ when $A$ is symmetric, though the symmetry of $A$ is not needed.

We will now first show that the Hölder continuity of $u$ implies Hölder continuity of the associated field $F$ solving the Dirac--Beltrami equation.

For the readers convenience we recall the definition of the Campanato spaces $\LL^{p,\lambda}(\Om,\Lambda)$. 

\begin{Def}[Camapanto spaces]
Set $\Om_{\rho}(x_0)=\Om\cap B_{\rho}(x_0)$.
\begin{align*}
\LL^{p,\lambda}(\Om,\Lambda):=\bigg\{F\in L^p(\Om,\Lambda): \sup_{\substack{x_0\in \Om\\ \rho>0}}\rho^{-\lambda}\int_{\Om_\rho(x_0)}\vert F(x)-(F)_{x_0,\rho}\vert^pdx<+\infty\bigg\},
\end{align*}
where 
\begin{align*}
(F)_{x_0,\rho}:=\frac{1}{\vert \Om_{x_0}(\rho)\vert}\int_{\Om_{x_0}(\rho)}F(x)dx. 
\end{align*}
The Campanato space $\LL^{\lambda,p}(\Om,\Lambda)$ is equipped with the semi-norm 
\begin{align*}
[F]_{p,\lambda}^p:=\sup_{\substack{x_0\in \Om\\ \rho>0}}\rho^{-\lambda}\int_{\Om_\rho(x_0)}\vert F(x)-(F)_{x_0,\rho}\vert^pdx
\end{align*}
and the norm 
\begin{align*}
\Vert F\Vert_{\LL^{p,\lambda}(\Om,\Lambda)}:=[F]_{p,\lambda}+\Vert F\Vert_{L^p(\Om,\Lambda)}. 
\end{align*}
\end{Def}

We also recall Campanato's theorem
\begin{Thm}[Campanato's theorem]
For $n<\lambda\leq n+p$ and $\alpha=(\lambda-n)/p$ we have $\LL^{p,\Lambda}(\Om,\Lambda)\cong C^{0,\alpha}(\Om,\Lambda)$ and the Hölder semi-norm
\begin{align*}
[F]_{C^{0,\alpha}}:=\sup_{\substack{x,y\in \Om\\ x\neq y}}\frac{\vert F(x)-F(y)\vert}{\vert x-y\vert^\alpha}
\end{align*}
is equivalent to the Campanato semi-norm $[F]_{p,\lambda}$. 
\end{Thm}
For a proof for scalar functions see Theorem 5.5 in \cite{GiMa}. The proof for vector valued functions is the same. 

\begin{Lem}\label{lem:Camp}
Let $F\in W^{1,2}_{d,\delta}(\Om,\Lambda)$ be a solution of $\Dm F=\mathcal{M}\Dp F$. Then for every $B_r(x_0)\Subset \Om$ there exists a constant $C=C(n,\lambda,\Lambda)$ such that 
\begin{align*}
\int_{B_r(x_0)}\vert F(x)-(F)_{x_0,r}\vert^2dx\leq C\int_{B_{2r}(x_0)}\vert u(x)-(u)_{x_0,2r}\vert^2dx. 
\end{align*}
In particular, for any $U\Subset \Om$ we have $ [F]_{2,\lambda}\leq C[u]_{2,\lambda}$ on $U$.  
\end{Lem}

\begin{proof}
By the Poincaré inequality \ref{thm:Poinc},
\begin{align*}
\int_{B_r(x_0)}\vert F(x)-(F)_{x_0,r}\vert^2dx\leq Cr^2\int_{B_r(x_0)}(\vert dF(x)\vert^2+\vert \delta F(x)\vert^2)dx
\end{align*}
By the gauge condition \eqref{eq:GaugeCond}, $dF(x)=du(x)=\nabla u(x)$ and $\delta F(x)=\delta v(x)=A(x)\nabla u(x)$. Hence, by the ellipticity assumption \eqref{eq:EBound1}
\begin{align*}
\vert dF(x)\vert^2+\vert \delta F(x)\vert^2=\vert \nabla u(x))\vert^2+\vert A(x)\nabla u(x)\vert^2\leq (1+\Lambda)\vert \nabla u(x)\vert^2
\end{align*}
By the Caccioppoli inequality in Theorem 4.4 in \cite{GiMa} there exists a constant $C=C(\lambda,\Lambda)$ such that for any $R>r$ and $\xi\in \R$, 
\begin{align*}
\int_{B_r(x_0)}\vert \nabla u(x)\vert^2dx\leq \frac{C}{(R-r)^2}\int_{B_R(x_0)}\vert u(x)-\xi\vert^2dx.
\end{align*}
Taking $\xi=(u)_{x_0,R}$ and $R=2r$ the conclusion follows. 
\end{proof}

\begin{Thm}\label{thm:Hölderloc}
Let $F\in W^{1,2}_{d,\delta}(\Om,\Lambda^(0,2))$ be a solution of $\Dm F=\mathcal{M}\Dp F$. Then $F\in C^{0,\alpha}_{loc}(\Om,\Lambda)$ for some $0<\alpha<1$. 
\end{Thm}

\begin{proof}
By Theorem \eqref{thm:Main1}, $u=\langle F\rangle_0$ is a weak solution  the uniformly elliptic equation 
\begin{align*}
\text{div}\, A(x)\nabla u(x)=0
\end{align*}
where $A(x)=(I-\mathcal{M}(x))(I+\mathcal{M}(x))^{-1}$. Hence, by De Giorgi-Nash-Moser theorem, $u\in C^{0,\alpha}_{loc}(\Om)$. By Campanato's theorem and Lemma \ref{lem:Camp}, $F\in C^{0,\alpha}_{loc}(\Om,\Lambda)$. 
\end{proof}

Though Theorem \ref{thm:Hölderloc} shows that solutions of the Dirac--Beltrami equation are locally Hölder continuous, the proof of this fact reduces to the proof De Giorgi-Nash-Moser theorem in the scalar case. This is somewhat unsatisfactory from the point of view of this paper, which is to derive properties of solutions to the scalar equation \eqref{eq:DirPro1} using the first order Dirac-Beltrami equation for the vector valued field $F$. We therefore give another proof under the assumption that the ellipticity constant $M$ is sufficiently small and the dimension is odd. This proof applies also to solutions of the general Dirac-Beltrami equation, where the field $F$ takes values in the full exterior algebra and is not derived from some scalar second order uniformly elliptic equation. In this case the De Giorgi-Nash-Moser theorem is not applicable.

\begin{Thm}
Assume that $n\geq 3$ is odd. Furthermore, assume that 
\begin{align*}
M\Vert \BB\Vert_{2n}<1. 
\end{align*}
Then any solution $F\in W^{1,2}_{d,\delta}(\Om,\Lambda)$ of $\Dm F(x)=\mathcal{M}(x)\Dm F(x)$ is belongs to $C^{1/2}_{loc}(\Om,\Lambda)$
\end{Thm}

\begin{proof}
Let $n=2m+1$. By the recursion identity \eqref{eq:Recursion}
\begin{align*}
\eta(x)F(x)=(\Cc^-(I-\mathcal{M}_0(x)\BB)^{-1}\sigma_{\mathcal{M}}(\nabla \eta,x)F)(x)
\end{align*}
for any $\eta\in C^\infty_0(\Om)$. By the assumption on $M$, $(I-\mathcal{M}_0(x)\BB)^{-1}$ is invertible on $L^p(\Om,\Lambda)$ for any $2\leq p\leq 2n$. By Theorem \ref{thm:CompactC}
\begin{align*}
\Vert \eta F\Vert_{p^\ast}\leq C\Vert F\Vert_{L^p(\text{supp}(\eta))}
\end{align*}
whenever $1<p<n$. By Appendix \ref{App:2}, this can be iterated $m-1$ times. Thus we may choose $m$ test functions $\eta_j$ such that $\text{supp}(\eta_{j+1})\subset \text{supp}(\eta_{j})$ and $\eta_{j}(x)=1$ for $x\in \text{supp}(\eta_{j+1})$ for $j=1,2,...,m-1$. This gives first the inequality 
\begin{align*}
\Vert \eta_{m-1} F\Vert_{2n}\leq C'\Vert F\Vert_{L^2(\text{supp}(\eta_1))}
\end{align*}
By Proposition \ref{prop:CauchyHölder},
\begin{align*}
\Vert \eta_mF(x)\Vert_{C^{0,1/2}}\leq C''\Vert \sigma_{\mathcal{M}}(\nabla \eta ,x)F\Vert_{L^{2n}}\leq  C'''\Vert F\Vert_{L^2(\text{supp}(\eta_1))}. 
\end{align*}
and the proof is complete. 
\end{proof}


\appendix

\section{\sffamily Auxiliary Algebraic Results}
\subsection{\sffamily Cayley Transform}

\begin{Def}
Let $V$ be a real finite dimensional vector space and let $T\in \LL(V)$ be such that $-1\notin \sigma(T)$. The \emph{Cayley transform} of $T$ is defined by 
\begin{align*}
\Cs(T):=(I-T)\circ (I+T)^{-1}
\end{align*}
\end{Def}

Recall that an operator $T\in \LL(V)$ is called positive if for all $x\neq 0$,
\begin{align*}
\langle Tx,x\rangle >0.
\end{align*}

\begin{Lem}\label{lem:IdentityC}
Assume that $-1\notin \sigma(T)$. Set $y=(I+T)(x)$. Then we have the following identities.
\begin{align}
\label{eq:IdCayley1}
\Cs(T)(y)&=(I-T)(x)\\
\label{eq:IdCayley2}
x&=\frac{1}{2}(y+\Cs(T)(y))\\
\label{eq:IdCayley3}
Tx&=\frac{1}{2}(y-\Cs(T)(y))
\end{align}
\end{Lem}

\begin{Lem}\label{lem:AccOp}
An operator $T\in \LL(V)$ is positive if and only if $\Vert \Cs(T)\Vert<1$. Furthermore, in this case $\Cs$ is an involution. 
\end{Lem}

\begin{proof}
If $T$ is positive then $-1\notin \sigma(T)$ and so $\Cs(T)$ is well-defined. Set $M=\Cs(T)$. For any $x\in V$ different from $0$ we have 
\begin{align*}
\vert x\vert^2-\vert Mx\vert^2&=\langle x,x\rangle -\langle x,M^\ast Mx\rangle=\langle x,(I-M^\ast M)x\rangle\\
&=\{\langle x,M^\ast x\rangle-\langle x,M x\rangle=0\}=\langle x,(I-M)^\ast(I+M)x\rangle\\
&=\langle (I-M)x,(I+M)x\rangle=\{\text{ identity \eqref{eq:IdCayley1}}\}=\langle T(I+M)x,(I+M)x\rangle
&=\langle Ty,y\rangle,
\end{align*}
with $y=(I+M)x$. Hence $\Vert M\Vert<1$ and so $-1\notin \sigma(M)$. Conversely, if $\Vert M\Vert<1$, then 
\begin{align*}
0<(1-\Vert M\Vert)^2\vert x\vert^2=\vert x\vert^2-\vert Mx\vert^2=\langle Ty,y\rangle.
\end{align*}
Hence $T$ is positive and $-1\notin \sigma(T)$.  Moreover, by definition 
\begin{align*}
M(I+T)=I-T
\end{align*}
which implies that $T=(I-M)(I+M)^{-1}=\Cs(M)$. Thus, $\Cs$ is an involution.
\end{proof}

We note that if $T=\frac{1}{2}(T+T^\ast)+\frac{1}{2}(T-T^\ast)$, i.e., $T$ is decomposed into its symmetric and anti-symmetric part, then for all $x\neq 0$, we have 
\begin{align*}
\langle Tx,x \rangle=\frac{1}{2}\langle (T+T^\ast)x,x \rangle,
\end{align*}
since 
\begin{align*}
\langle (T-T^\ast)x,x \rangle=\langle Tx,x \rangle-\langle T^\ast x,x \rangle=\langle Tx,x \rangle-\langle x,Tx \rangle=0. 
\end{align*}
\begin{ex}
Assume that $A=\sigma I$ is an isotropic matrix with $\sigma\neq -1$. Then 
\begin{align*}
\Cs(A)=\frac{1-\sigma}{1+\sigma}I. 
\end{align*}
\end{ex}

\begin{ex}
Assume that $A$ is a self-adjoint matrix with eignvalues $0<\lambda_1<..<\lambda_n$ and that $\lambda_1\leq 1\leq \lambda_n$. By the spectral theorem $A=R^\ast DR$ where $R$ is a rotation and $D$ is a diagonal matrix in the eigenvalues of $A$. Hence 
\begin{align*}
\Cs(A)&=(I-A)(I+A)^{-1}=(I-R^\ast DR)(I+R^\ast DR)^{-1}=R^\ast (I-D)R(R^\ast (I+D)R))^{-1}\\&=R^\ast (I-D)RR^\ast (I+D)^{-1}R=R^\ast(I-D)(I+D)R\\
&=R^\ast\Cs(D)R. 
\end{align*}
Hence 
\begin{align*}
\det(\Cs(A))=\prod_{j=1}^n\frac{1-\lambda_j}{1+\lambda_j}. 
\end{align*}
Thus 
\begin{align*}
\Vert \Cs(A)\Vert&=\Vert \Cs(D)\Vert=\max_{\lambda \in \sigma(A)}\bigg\vert \frac{1-\lambda}{1+\lambda}\bigg\vert 
=\max\bigg\{ \frac{1-\lambda_1}{1+\lambda_1},\frac{\lambda_n-1}{\lambda_n+1}\bigg\}.
\end{align*}
In particular,
\begin{align*}
\Vert \Cs(A)\Vert\leq \max\{\lambda_1^{-1},\lambda_n\}\leq \lambda_1^{-1}+\lambda_n,
\end{align*}
and in the special case when $\lambda_{1}^{-1}=\lambda_n$, then 
\begin{align*}
\Vert \Cs(A)\Vert&=\Vert \Cs(D)\Vert=\max\bigg\{ \frac{\lambda_1^{-1}-1}{1+\lambda_1^{-1}},\frac{\lambda_n-1}{\lambda_n+1}\bigg\}=\frac{\lambda_n-1}{\lambda_n+1}. 
\end{align*}
\end{ex}

The example above can be generalized to normal matrices to yield the following lemma:
\begin{Lem}\label{lem:NormalNorm}
Let $A\in \LL(V)$ be a normal matrix such that $-1\neq \sigma(A)$. Then 
\begin{align}
\Vert \Cs(A)\Vert\leq \max\bigg\{ \frac{1-\sigma_1}{1+\sigma_1},\frac{\sigma_n-1}{\sigma_n+1}\bigg\},
\end{align}
where $\sigma_1\leq \sigma_n$ are the smallest and largest singular value of $A$ respectively. 
\end{Lem}

\begin{proof}
We use the fact that normal matrices are diagonizable by unitary matrices and that we have the relation $\sigma_j=\vert \lambda_j\vert$ between the singular values and eigenvalues of $A$. Furthermore, for a normal matrix $A$ we have $\Vert A\Vert=\sup\{\vert \lambda\vert: \lambda\in \sigma(A)\}. $ The statement then follows from an identical argument as in the example above. 
\end{proof}

\begin{Prop}\label{prop:CayleyExt}
Let $T\in \LL(V)$ and assume that $\Vert T\Vert<1$. Then $\Vert \Lambda T\Vert <1.$
\end{Prop}

\begin{proof}
By polar factorization $T=R\circ S=S\circ R$. Hence $\Vert T\Vert=\Vert S\Vert$. Let $\sigma_{max}$ be the largest singular value of $T$, i.e. the largest eigenvalue of $S$. Then $\Vert T\Vert=\sigma_{max}$. 
Since $\Lambda T=\Lambda (S\circ R)=\Lambda S\circ \Lambda R$, $\Vert T\Vert=\Vert \Lambda S\Vert$. 
By the spectral theorem $S$ has an ON-eigenbasis $\{\xi_j\}_j$ which induces and ON-eigenmultivector basis for $\Lambda S$. Let $\sigma_1\leq ...\leq \sigma_n$ be the eignvalues for $S$ ordered according to size. Since the largest eigenvalue for $\Lambda^kS$ is $\sigma_n\sigma_{n-1}...\sigma_{n-k+1}$ and $\sigma_n\sigma_{n-1}...\sigma_{n-k+1}\leq \sigma_n^k$, we have 
\begin{align*}
\Vert \Lambda T\Vert=\max\{\sigma_n,\sigma_n\sigma_{n-1},....,\sigma_{n}...\sigma_1\}\leq \max\{\sigma_n,\sigma_n^n\}.
\end{align*}
By assumption, $\Vert T\Vert<1$ so $\sigma_n<1$. Hence $\Vert \Lambda T\Vert\leq \sigma_n^n=\Vert T\Vert^n<1$. 
\end{proof}

\begin{Prop}
Let $T\in \LL(V)$ be self-adjoint and positive. Then $\Lambda T$ is self-adjoint and positive.
\end{Prop}

\begin{proof}
The self-adjointness of $\Lambda T$ follows from the identity $\Lambda T^\ast=(\Lambda T)^\ast$. By linearity it is enough to prove that $\langle \Lambda Tw,w\rangle>0$ for all $w\in \Lambda^kV $. By the spectral theorem $T$ has an ON-eigenbasis $\{\xi_j\}_j$ which induces and ON-eigenbasis for $\Lambda^kT$. Hence let $\xi_{i_1}\wedge...\wedge\xi_{i_k}$ be an eigenmultivector for $\Lambda^kV$ corresponding to eigenvalue $\lambda_{i_1}...\lambda_{i_k}$. Then
\begin{align*}
\langle \Lambda T(\xi_{i_1}\wedge...\wedge\xi_{i_k}),\xi_{i_1}\wedge...\wedge\xi_{i_k}\rangle=\lambda_{i_1}...\lambda_{i_k}\langle \xi_{i_1}\wedge...\wedge\xi_{i_k},\xi_{i_1}\wedge...\wedge\xi_{i_k}\rangle >0. 
\end{align*}
\end{proof}

\begin{Prop}
Let $T\in \LL(V)$ and assume that $T$ is invertible, normal (i.e., $T^\ast T=TT^\ast$) and that $-1\notin \sigma(T)$. Then 
\begin{align*}
 \Cs(T^{-1})=-\Cs(T). 
\end{align*}
\end{Prop}

\begin{proof}
First observe that if $x\in \R$ and $x\neq -1$, then 
\begin{align*}
\frac{1-x^{-1}}{1+x^{-1}}=-\frac{1-x}{1+x}.
\end{align*}
Thus, if $T$ is a diagonal map $D$, then $\Cs(D^{-1})=-\Cs(D)$. Since $T$ is normal, $T$ is diagonizable by a unitary map $U$ and so is $T^{-1}$. Hence
\begin{align*}
 \Cs(T^{-1})= \Cs(UD^{-1}U^\ast)=U \Cs(D^{-1})U^\ast=-U \Cs(D)U^\ast=-\Cs(T).
\end{align*}
\end{proof}

\subsection{\sffamily Ellipticity Bounds}

\begin{Lem}\label{lem:EllipticityBound}
Let $V$ a finite dimensional real euclidean vector space and assume that $A\in \LL(V)$ is positive and normal. Assume in addition that $A$ satisfies the inequalities
\begin{align}\label{eq:Eineq1}
\lambda \vert v\vert^2\leq \langle Av,v\rangle, \quad \Vert A\Vert\leq \Lambda.
\end{align}
Then 
\begin{align}\label{eq:Eineq2}
\vert v\vert^2+\vert Av\vert^2\leq 2\bigg(\frac{1+ \mathcal{M}^{2}}{1- \mathcal{M}^{2}}\bigg)\langle Av,v\rangle,
\end{align}
where 
\begin{align*}
\mathcal{M}=\Vert \mathscr{C}(A)\Vert=\max\bigg\{ \frac{1-\lambda}{1+\lambda},\frac{\Lambda-1}{\Lambda+1}\bigg\}.
\end{align*}.

Conversely any $A\in \LL(V)$ that satisfy 
\begin{align}\label{eq:Eineq3}
\vert v\vert^2+\vert Av\vert^2\leq \mathcal{K}\langle Av,v\rangle.
\end{align}
for some $\mathcal{K}\geq 2$ implies that $A$ also satisfies the two sided inequality 
\begin{align}\label{eq:Eineq4}
\mathcal{K}^{-1}\vert v\vert^2\leq \langle Av,v\rangle\leq \mathcal{K}\vert v\vert^2
\end{align}
\end{Lem}

\begin{proof}
First assume that \eqref{eq:Eineq1} holds. We have the identity 
\begin{align*}
I-A=\mathscr{C}(A)\circ (I+A). 
\end{align*}
and by Lemma \ref{lem:AccOp} $\mathcal{M}:=\Vert \mathscr{C}(A)\Vert<1$. By the polarization identity 
\begin{align*}
\langle Av,v\rangle&=\frac{1}{4}(\vert Av+v\vert^2-\vert Av-v\vert^2)\\
&=\frac{1}{4}(\vert Av+v\vert^2-\vert\mathscr{C}(A)\circ (I+A)v\vert^2)\\
&\geq \frac{1}{4}(\vert Av+v\vert^2- \mathcal{M}^{2}\vert v+Av\vert^2)\\
&=\frac{1- \mathcal{M}^{2}}{4}(\vert v\vert^2+\vert Av\vert^2+2\langle Av,v\rangle).
\end{align*}

Hence,
\begin{align*}
\bigg(\frac{1- \mathcal{M}^{2}}{4}\bigg)(\vert v\vert^2+\vert Av\vert^2)\leq \bigg(1-\frac{1- \mathcal{M}^{2}}{2}\bigg)\langle Av,v\rangle,
\end{align*}
or equivalently 
\begin{align}\label{ineq:Ellip}
\vert v\vert^2+\vert Av\vert^2\leq 2\bigg(\frac{1+ \mathcal{M}^{2}}{1- \mathcal{M}^{2}}\bigg)\langle Av,v\rangle. 
\end{align}
Since $\mathcal{M}<1$, it follows that $\frac{1+ \mathcal{M}^{2}}{1- \mathcal{M}^{2}}>1$. By Lemma \ref{lem:NormalNorm} the result follows. 
For the converse direction the assumptions imply 
\begin{align*}
\langle Av,v\rangle\geq \frac{1}{\mathcal{K}}(\vert v\vert^2+\vert Av\vert^2)> \frac{1}{\mathcal{K}}\vert v\vert^2. 
\end{align*}
Hence $A$ is strictly positive and thus invertible. Moreover, the Cauchy-Schwarz inequality together with \eqref{eq:Eineq3} give for every $v\neq 0$ 
\begin{align*}
\frac{\vert v\vert}{\vert Av\vert}+\vert Av\vert\leq \mathcal{K},
\end{align*}
which implies that for every $v$ such that $\vert v\vert=1$
\begin{align*}
\vert Av\vert\leq \mathcal{K}. 
\end{align*}
\end{proof}

\subsection{\sffamily Nonlinear Cayley Transform}

\begin{Lem}\label{lem:NLC1}
Let $\A: \Om\times \R^n \to \R^n$ be a mapping that satisfy the conditions $(i)-(ii)$ in Definition \ref{def:Afield}. Moreover assume that $\A$ is monotone, i.e., that for a.e. $x\in \Om$ and all $\xi_1,\xi_2\in \R^n$
\begin{align*}
\langle \A(x,\xi_1)-\A(x,\xi_2),\xi_1-\xi_1\rangle\geq 0.
\end{align*}
Then the mapping $T(x,\cdot): \R^n \to \R^n$ defined by $T(x,\xi)=\xi+\A(x,\xi)$ is a homeomorphism for a.e. $x\in \Om$. 
\end{Lem}

\begin{proof}
By definition of $\A$, $T$ is continuous for a.e. $x$. By the monotonicity of $\A$ it follows that 
\begin{align*}
\langle T(x,\xi)-T(x,\zeta),\xi-\zeta\rangle =\vert \xi-\zeta \vert^2+\langle \A(x,\xi)-\A(x,\zeta),\xi-\zeta\rangle \geq \vert \xi-\zeta \vert^2.
\end{align*}
Thus $T$ is injective and coercive for a.e. $x\in \Om$. By the Browder-Minty theorem (\cite[Thm. 2.2 p. 39]{Sho}) $T$ is surjective for a.e. $x\in \Om$. Finally, by the invariance of domain theorem it follows that $T$ is a homeomorphism for a.e. $x\in \Om$. 
\end{proof}

By Lemma \ref{lem:NLC1} it follows that the nonlinear Cayley transform $\mathcal{M}(x,\xi)=(I-\A(x,\xi))\circ (I+\A(x,\xi))^{-1}$ is well defined for a.e. $x\in \Om$. 

Moreover, we have
\begin{align*}
\xi-\A(x,\xi)=\mathcal{M}(x,\xi+\A(x,\xi))
\end{align*}
Using that the map $\xi\mapsto \xi +\A(\xi)$ is invertible we can make the change of variables for a.e. $x\in \Om$ 
\begin{align*}
\zeta_x=\xi+\A(x,\xi). 
\end{align*}
Hence we get the following useful identities 
\begin{align}\label{eq:Cid}
\xi&=\frac{1}{2}(\zeta+\mathcal{M}(x,\zeta))\\
\A(x,\xi)&=\frac{1}{2}(\zeta-\mathcal{M}(x,\zeta))
\end{align}

\begin{Lem}\label{lem:NLC2}
Let $\A: \Om\times \R^n \to \R^n$ be a mapping that satisfy the conditions $(i)-(iv)$ in Definition \ref{def:Afield}. Then the Cayley transform $\mathcal{M}(x,\cdot)$ is $\sqrt{\frac{K-2}{K+2}}$-Lipschitz continuous and 
$\mathcal{M}(x,0)=0$ for all $x\in \Om$. 
\end{Lem}

\begin{proof}
Set $\zeta_1=\xi_1 +\A(x,\xi_1)$ and $\zeta_2=\xi_2+\A(x,\xi_2)$.  We have
We have 
\begin{align*}
 \langle \A(x,\xi_1)-\A(x,\xi_2),\xi_1-\xi_2\rangle&=\frac{1}{4}\langle \zeta_1-\mathcal{M}(x,\zeta_1)-\zeta_2+\mathcal{M}(x,\zeta_2), \zeta_1+\mathcal{M}(x,\zeta_1)- \zeta_2-\mathcal{M}(x,\zeta_2)\rangle\\
&=\frac{1}{4}\vert \zeta_1-\zeta_2\vert^2-\frac{1}{4}\vert\mathcal{M}(z,\zeta_1)-\mathcal{M}(x,\zeta_2) \vert^2\\+&\frac{1}{4}\langle \zeta_1-\zeta_2,\mathcal{M}(x,\zeta_1)-\mathcal{M}(x,\zeta_2)\rangle+\frac{1}{4}\langle -\mathcal{M}(x,\zeta_1)+\mathcal{M}(x,\zeta_2),\zeta_1-\zeta_2\rangle\\
&=\frac{1}{4}\vert \zeta_1-\zeta_2\vert^2-\frac{1}{4}\vert\mathcal{M}(x,\zeta_1)-\mathcal{M}(z,\zeta_2) \vert^2.
\end{align*}
and 
\begin{align*}
\vert \xi_1-\xi_2\vert^2+\vert \A(x,\xi_1)-\A(x,\xi_2)\vert^2&=\frac{1}{4}\vert \zeta_1+\mathcal{M}(x,\zeta_1)- \zeta_2-\mathcal{M}(x,\zeta_2)\vert^2+\frac{1}{4}\vert  \zeta_1-\mathcal{M}(x,\zeta_1)-\zeta_2+\mathcal{M}(x,\zeta_2)\vert^2\\
&=\frac{1}{2}\vert \zeta_1-\zeta_2\vert^2+\frac{1}{2}\vert \mathcal{M}(x,\zeta_1)-\mathcal{M}(x,\zeta_2)\vert^2
\end{align*}

By assumption $(iv)$ on $\A$ in Definition \ref{def:Afield} we get the inequality 
\begin{align*}
\frac{1}{2}\vert \zeta_1-\zeta_2\vert^2+\frac{1}{2}\vert \mathcal{M}(x,\zeta_1)-\mathcal{M}(x,\zeta_2)\vert^2\leq \frac{K}{4}\vert \zeta_1-\zeta_2\vert^2-\frac{K}{4}\vert\mathcal{M}(x,\zeta_1)-\mathcal{M}(z,\zeta_2) \vert^2
\end{align*}
which implies 
\begin{align*}
 \mathcal{M}(x,\zeta_1)-\mathcal{M}(x,\zeta_2)\vert\leq \sqrt{\frac{K-2}{K+2}}\vert \zeta_1-\zeta_2\vert.
\end{align*}
By assumption $(iii)$ in Definition \ref{def:Afield} we get that $\mathcal{M}(x,0)=0$ for all $x\in \Om$. 
\end{proof}

\begin{rem}
The proof of Lemma \ref{lem:NLC2} is also given in \cite{ACFJK}. We have included it for the sake of completeness. 
\end{rem}

\section{\sffamily Iteration of Sobolev Conjugate Exponent}\label{App:2}

Define the function 
\begin{align*}
f_n(x):=\frac{nx}{n-x}
\end{align*}
for $x\in [0,n)$. Set $f_n^{k}(x):=\underbrace{f_n\circ f_n\circ...\circ f_n(x)}_{k \text{ times}}$. Proof by induction shows that 
\begin{align*}
f_n^k(x):=\frac{n^kx}{n^k-kn^{k-1}x}=\frac{nx}{n-kx}. 
\end{align*}

This gives 
\begin{align*}
f_n^k(2)=\frac{2n}{n-2k}. 
\end{align*}
as long as $f_n^k(2)\leq n$. The iteration stops when $f_n^k(2)\geq n$.
If $n=2m$ and $k=m-1$ gives 
\begin{align*}
f_{2m}^{m-1}(2)=\frac{4m}{2m-2(m-1)}=2m=n. 
\end{align*}
So in even dimension the iteration stops at $k=m-1$ giving $f_n^k(2)=n$. In odd dimension $n=2m+1$ we get with $k=m$
\begin{align*}
f_{2m+1}^{m}(2)=\frac{2(2m+1)}{2m+1-2m)}=2(2m+1)=2n. 
\end{align*}

\section{\sffamily Auxiliary Analytic Results}

\subsection{\sffamily Integration by Parts Formula}

\begin{Def}
Let $\Om\subset \R^n$ be a $C^2$-domain. Define
\begin{align*}
W^{1,2}_d(\Om,\Lambda)&:=\{F\in \mathscr{D}'(\Om,\Lambda): F,d F\in L^2(\Om,\Lambda)\},\\
W^{1,2}_\delta(\Om,\Lambda)&:=\{F\in  \mathscr{D}'(\Om,\Lambda): F,\delta F\in L^2(\Om,\Lambda)\}.
\end{align*}
\end{Def}

\begin{Thm}[Trace Theorems]
Let $\Om\subset \R^n$ be a $C^2$-domain. There exits a bounded linear operator $\gamma_N: W^{1,2}_\delta(\Om,\Lambda)\to W^{-1/2,2}(\dv \Om,\Lambda)$ such that for every $\Phi\in C^{\infty}(\overline{\Om},\Lambda)$
\begin{align*}
\gamma_N\Phi=\nu(x)\ri \Phi(x)\vert_{\dv \Om}. 
\end{align*}
\end{Thm}

\begin{Thm}[See \cite{Tem} Thm 1.2 p. 7]
\label{thm:IntPart1}
For every $F\in W^{1,2}_\delta(\Om,\Lambda)$ and every $G\in W^{1,2}(\Om)$ the integration we have the integration by parts formula
\begin{align*}
\int_{\Om}\langle F(x),d G(x)\rangle dx+\int_{\Om}\langle \delta F(x),G(x)\rangle dx=\int_{\dv \Om}\langle \gamma_NF(x),\gamma G(x)\rangle d\sigma(x)
\end{align*}
Similarly, for 
$F\in W^{1,2}_d(\Om,\Lambda)$ and every $G\in W^{1,2}(\Om)$ the integration we have the integration by parts formula
\begin{align*}
\int_{\Om}\langle F(x),\delta G(x)\rangle dx+\int_{\Om}\langle dF(x),G(x)\rangle dx=\int_{\dv \Om}\langle \gamma_NF(x),\gamma G(x)\rangle d\sigma(x)
\end{align*}
\end{Thm}

\begin{Prop}\label{prop:IntPart}
Let $F\in W^{1,2}(\Om,\Lambda)$ and let $U\Subset \Om$ be a $C^2$-domain. Let $\nu(x)$ denote the unit outer normal of $\dv U$ at $x$. Then 
\begin{align}\label{eq:IntPart1}
\int_{U}\langle d F(x),\delta F(x)\rangle dx&=\int_{\dv U}\langle \nu(x)\wedge F(x),\delta F(x)\rangle d\sigma(x),\\ \label{eq:IntPart2}
&=\int_{\dv U}\langle d F(x),\nu(x)\ri  F(x)\rangle d\sigma(x).
\end{align}
\end{Prop}

\begin{proof}
Since $\delta^2 F(x)=0$ in the sense of distributions and $\delta F\in L^2(\Om,\Lambda)$ it follows that $\delta F\in W^{1,2}_{\delta}(\Om)$. Moreover, since $\gamma F\in W^{1/2,2}(\dv U,\Lambda)\subset L^2(\dv U,\sigma;\Lambda)$, $\gamma_T(x)F=\nu(x)\wedge F(x)$ for $\sigma$ a.e. $x\in \dv U$. Hence by Theorem \ref{thm:IntPart1} the first identity follows. The other one follows similarly. 
\end{proof}

\subsection{\sffamily Poincaré Inequalities}

\begin{Def}
Let $\Om$ be a smooth bounded domain in $\R^n$. We recall the Hodge decomposition (Theorem 10.5.1 in \cite{IM2})
\begin{align*}
L^2(\Om,\Lambda)=d W_T^{1,2}(\Om,\Lambda)\oplus \delta W_N^{1,2}(\Om,\Lambda)\oplus \mathcal{H}_2(\Om),
\end{align*}
where $ \mathcal{H}_2(\Om):=\{h\in C^\infty(\Om,\Lambda): dh(x)=0,\,\, \delta h(x)=0\}\cap L^2(\Om,\Lambda)$. The harmonic projection operator $\mathbf{H}_\Om^2$ is defined as the orthogonal projection 
\begin{align*}
\mathbf{H}_{\Om}^2:L^2(\Om,\Lambda)\to  \mathcal{H}_2(\Om). 
\end{align*}
\end{Def}

\begin{Thm}[Poincaré inequality on balls]
\label{thm:Poinc}
Let $F\in W^{1,2}(B_r(0),\Lambda)$. Let $\mathbf{H}_r^2=\mathbf{H}_{B_r(0)}^2$ be the harmonic projection on $B_r(0)$. Then the there exists a constant $C=C(n,\mathbb{B}^n)$ such that 
\begin{align*}
\int_{B_r(0)}\vert F(x)-\mathbf{H}_r^2F(x)\vert^2dx\leq Cr^2\int_{B_r(0)}(\vert dF(x)\vert^2+\vert \delta F(x)\vert^2)dx
\end{align*}
where 
\begin{align*}
\mathbf{H}_r^2F(x)=\int_{B_r(0)}F(y)dy.
\end{align*}
\end{Thm}

\subsection{\sffamily Isoperimetric Type Inequalities}

\begin{Thm}[Isoperimetric type inequality]
\label{thm:IsoIneq}
Let $B,E\in L^2(\Om,\R^n)$ where $\Om$ is a bounded domain satisfy 
\begin{align*}
\text{div } B=0,\,\,\,\, \text{curl } E=0
\end{align*}
in the sense of distributions. Then there exists a constant $c_n> 0$ depending only on dimension such that for every $x_0\in \Om$ and every $0<r<R$ such that $B_R(x_0)\Subset \Om$ 
\begin{align}\label{ineq:Jac}
\int_{B_r(x_0)}\langle B(x),E(x)\rangle dx\leq c_n\bigg(\int_{\dv B_r(x_0)} (\vert E(x)\vert^2+\vert B(x)\vert^2)^{(n-1)/n}d\sigma(x)\bigg)^{n/(n-1)}
\end{align}
\end{Thm}

Theorem \ref{thm:IsoIneq} is proven in \cite{CPdiN}, Theorem A. Note that in the case when $F\in W^{1,2}(\Om,\Lambda^{(0,2)})$ and $E=dF=\nabla u$ and $B=\delta F=A(x)\nabla u(x)$, and the inequality takes the form
\begin{align*}
\int_{B_r(x_0)}\langle dF(x),\delta F(x)\rangle dx\leq c_n\bigg(\int_{\dv B_r(x_0)} (\vert dF(x)\vert^2+\vert \delta F(x)\vert^2)^{(n-1)/n}d\sigma(x)\bigg)^{n/(n-1)}. 
\end{align*}

We also have the following Theorem 2.4 from \cite{CPdiN}. 
\begin{Thm}
\label{thm:DivCurl}
Let $B,E\in L^2(\Om,\R^n)$ satisfy 
\begin{align*}
\text{div } B=0,\,\,\,\, \text{curl } E=0
\end{align*}
in $\Om$ in the sense of distributions. Assume in addition that the div-curl couple $(B,E)$ satisfy the inequality 
\begin{align*}
\vert B(x)\vert^2+\vert E(x)\vert^2\leq \mathcal{K}\langle B(x),E(x)\rangle \text{ for a.e. $x\in \Om$},
\end{align*}
for some $\mathcal{K}\geq 2$. Then there exists a constant $\gamma_n\geq 0$ depending only on dimension such that for every $x_0\in \Om$ and every $0<r<R$ such that $B_R(x_0)\Subset \Om$ 
\begin{align}\label{ineq:Jac}
\int_{B_r(x_0)}\langle B(x),E(x)\rangle dx\leq \bigg(\frac{r}{R}\bigg)^{\frac{1}{\gamma_n\mathcal{K}}}\int_{B_R(x_0)}\langle B(x),E(x)\rangle dx
\end{align}

\end{Thm}

{\sc Erik Duse}, Department of Mathematics and Statistics, KTH,  Stockholm, Sweden \texttt{duse@kth.se}


\begin{thebibliography}{10}






\bibitem{ACFJK}
{\sc K. Astala, A. Clop, D. Faraco, J. J{ää}skel{ä}inen and A. Koskis}
Improved H{ö}lder regularity for strongly elliptic PDEs
\emph{Journal de Math{é}matiques Pures et Appliqu{é}es}  Volume 140, August 2020, Pages 230{-}258



\bibitem{AIM}
{\sc K. Astala, T. Iwaniec and G. Martin}
Elliptic Partial Differential Equations and Quasiconformal Mappings in the Plane
\emph{Princeton University Press} (2009) NJ.




\bibitem{AHLMT}
{\sc P. Auscher, S. Hofmann, M. Lacey, A. McIntosh and P. Tchamitchian}
The solution of the Kato square root problem for second order elliptic operators on {$\mathbb{R}^n$}.
\emph{ Ann. of Math.} (2) 156, 2 (2002), 633{-}654.


\bibitem{AAM}
{\sc P. Auscher, A. Axelsson and A. McIntosh }
 Solvability of elliptic systems with square
integrable boundary data
\emph{Ark. Mat.,} {\bf 48} (2010), 253{-}287



\bibitem{AR}
{\sc P. Auscher and A. Axelsson}
Weighted maximal regularity estimates and solvability of non-smooth elliptic systems I
\emph{ Inventiones mathematicae } volume 184, pages47?115(2011)



\bibitem{Ax}
{\sc S. Axler}
Harmonic Functions from a Complex Analysis Viewpoint
\emph{The American Mathematical Monthly }, Apr., 1986, Vol. 93, No. 4 pp. 246-258







\bibitem{BCO}
{\sc J. M. Ball, J.C. Currie and P.J. Olver}
Null Lagrangians, Weak Continuity, and Variational Problems of Arbitrary Order
\emph{Journal of Functional Analysis}, {\bf 41}, 135-174 (1981)








\bibitem{BB}
{\sc W. Ballmann and C. Bär}
Guide to elliptic boundary value problems for Dirac-type operators.
\emph{Progr. Math.}, Arbeitstagung Bonn 2013, 43{--}80, 319 Birkhaüser/Springer, Cham. 2016 








\bibitem{BB2}
{\sc W. Ballmann and C. Bär}
Boundary value problems for elliptic differential operators of first order.
\emph{Surveys in differential geometry.}, Vol. XVII, 1{-}78, Int. Press, Boston, MA, 2012. 




\bibitem{BBR}
{\sc M. Barnabei, A. Brini and G.-C. Rota}
On the Exterior Calculus of Invariant Theory
\emph{Journal of Algebra }, 120-160 (1985)




   
\bibitem{BFG}
{\sc R. Benedettia, R. Frigerioa and R. Ghiloni}
The topology of Helmholtz domains
\emph{Expo. Math. }, 30 (2012) 319{-}375





\bibitem{BL}
{\sc R. Ba{$\text{$\tilde{\text{n}}$}$}uelos and A. Lindeman}
A Martingale Study of the Beurling-Ahlfors Transform in $R^n$
\emph{Journal of Functional Analysis}, {\bf 145}, 224{--}265 (1997).


\bibitem{Dac}
{\sc G. Csató, B. Dacorogna and O. Kneuss}
The Pullback Equation for Differential Forms
{\em Birkhaüser} Progress in Nonlinear Differential Equations and Their Applications, Vol. 83




\bibitem{CDG}
{\sc J. Cantarella, D. DeTurck and H. Gluck}
Vector Calculus and the Topology of Domains in 3-Space
{\em The American Mathematical Monthly,} Vol. 109, No. 5 (May, 2002), pp. 409-442



  


\bibitem{CDGM}
{\sc S. Cappell, D. DeTurck, H. Gluck, and E.Y. Miller}
Cohomology of harmonic forms on Riemannian manifolds
with boundary
{\em Forum Math.} 18 (2006), 923{-}931



\bibitem{CMM}
{\sc R. R. Coifman, A. McIntosh and Y. Meyer}
L{'}int{é}grale de Cauchy {d'}efinit un op{é}rateur born{é} sur L2 pour les courbes lipschitziennes. 
{\em  Ann. of Math.} (2) 116, 2 (1982), 361?387.






\bibitem{CPdiN}
{\sc M. Carozza and A. Passarelli di Napoli}
Isoperimetric type inequality for div-curl fields and Hölder continuity 
{\em J. of Inequal. and Appl.}, 2002, Vol. 7(3) pp. 405-419  






\bibitem{CDT}
{\sc A. Connes, D. Sullivan and N. Teleman}
Quasiconformal 4-manifolds
{\em Topology}, Vol. 33 No. 4 pp 663{-}681. 1994


\bibitem{DeG}
{\sc E. De Giorgi }
Sulla differenziabilit{à} e {l'}analiticit{à} delle estremali degli integrali multipli regulari.
{\em Mem. Acc. Sc. Torino.}, ser. 3, 25-42 1957.



\bibitem{DB}
{\sc R. Delanghe and F. Brackx}
Hypercomplex Function Theory and Hilbert Modules with Reproducing Kernel
{\em Proceedings of the London Mathematical Society}, (1978) s3{-}37: 545{-}576 

\bibitem{Ding}
{\sc Z. Ding }
A Proof of the Trace Theorem of Sobolev Spaces on Lipschitz Domains
{\em Proceedings of the AMS}, Volume 124, Number 2, February 1996



\bibitem{DS}
{\sc S.K. Donaldson and D.P. Sullivan}
Quasiconformal 4-manifolds
{\em Acta Math.,}, 163 (1989), 181-252.

 

\bibitem{DuGr}
{\sc J. Dugundji and A. Granas }
Fixed point theory
{\em Springer Monographs in Mathematics}, 2003 NY.


 








\bibitem{FJR}
{\sc  E. B. Fabes, M. Jodeit Jr. and N. M. Rivi{è}re}
Potential techniques for boundary values problems on {$C^1$-}domains.
{\em Acta Math.}, Volume 141 (1978), 165{-}186.









\bibitem{IM1}
{\sc T. Iwaniec and G. Martin}
Quasiregular mappings in even dimensions
\emph{Acta Math.}, 170 (1993), 29-81 


\bibitem{IM2}
{\sc T. Iwaniec and G. Martin}
Geometric Function Theory and Non-Linear Analysis
{\em Oxford Mathematical Monographs} 2001 

\bibitem{IMS}
{\sc T. Iwaniec M. Mitrea and C. Scott}
Boundary value estimates for harmonic forms
\emph{Proc. of the AMS}, Volume 124, Number 5, May 1996




\bibitem{IL}
{\sc T. Iwaniec and A. Lutoborski}
Integral Estimates for Null Lagrangians
{\em Arch. Rational Mech. Anal.} 125 (1993) 25-79 


\bibitem{IS}
{\sc T. Iwaniec and C. Sbordone}
Quasiharmonic Fields 
{\em Ann. I. H. Poincar{é}-AN} 18,5 (2001) 519-572 


\bibitem{GdiN}
{\sc F. Giannetti and A. Passarelli di Napoli}
Isoperimetric type inequality for Differential Forms on Manifolds
{\em Indiana University Mathematics Journal.},  Vol. 45, No. 5 (2005)


\bibitem{GiMa}
{\sc M. Giaquinta and L. Martnazzi}
An Introduction to the Regularity Theory for Elliptic Systems, Harmonic Maps and Minimal Graphs
\emph{Edizoni Della Normale} 2012






\bibitem{GM}
{\sc J. Gilbert and M. Murray}
Clifford algebras and Dirac operators in harmonic analysis 
{\em Cambridge studies in advanced mathematics} 26 Cambridge University Press, 1991 Cambridge


\bibitem{GW}
{\sc R. Goodman and N. Wallach}
Symmetry, Representations and Invariants
\emph{Graduate Texts in Mathematics, Springer} 2009 NY.


\bibitem{HMOMPET}
{\sc S. Hofmann, E. Marmolejo{-}Olea, M. Mitrea, S. P{é}rez-Esteva and M. Taylor }
Hardy spaces, singular integrals and the geometry of euclidean domains of locally finite perineter
\emph{Geom. Funct. Anal. } Vol. 19 (2009) 842{-}882





\bibitem{HLMcIZ}
{\sc J. Hogan, C. Li, A. McIntosh and K. Zwang}
Global higher integrability of jacobians on bounded domains
\emph{Annales de {l'I}. H. P., }section C, tome 17, no 2 (2000), p. 193-217



\bibitem{HW}
{\sc G. Hsiao and W. L. Wendland}
Boundary Integral Equations 
\emph{Applied Mathematical Sciences, Springer} Vol. 164 2008.







\bibitem{H}
{\sc T. Hytönen}
On the Norm of the Beurling-Ahlfors Operator in Several Dimensions
\emph{Canad. Math. Bull.}, Vol. 54 (1) 2011 pp. 113{-}125.




\bibitem{LM}
{\sc B. Lawson and M. Michelson}
Spin Geometry
\emph{Princeton University Press}, NJ 1989.




\bibitem{LMcIS}
{\sc C. Li, A. McIntosh and S. Semmes}
Convolution Singular Integrals on Lipschitz Surfaces
\emph{Journal of the American Mathematical Society}, Vol. 5, N. 3 July 1992. 


\bibitem{LS}
{\sc D. Lundholm and L. Svensson}
Clifford algebra, geometric algebra, and applications
\emph{arXiv:0907.5356} 




\bibitem{Mar}
{\sc M. Mitrea}
Half-Dirichlet problems for Dirac operators in Lipschitz domains
{\em Adv. Appl. Clifford Algebras} {\bf 11} (1), 122{-}135 (2001)


\bibitem{McL}
{\sc W. McLean}
Strongly Elliptic Systems and Boundary Integral Equations
{\em Cambridge University Press} 2000. 



\bibitem{Mey}
{\sc N. G. Meyers}
An {$L^p$}-estimate for the gradient of solutions of second order elliptic divergence equations
\emph{Annali della Scuola Normale Superiore di Pisa, Classe di Scienze {$3^e$} s{é}rie}, tomme 17, {$n^{o}3$} (1963), p. 189-206 




\bibitem{MMMT}
{\sc D. Mitrea, I. Mitrea, M. Mitrea and M. Taylor}
The Hodge--Laplacian. Boundary Values on Riemannian Manifolds
{\em De Gruyter}   Studies in Mathematics 64



\bibitem{Moser}
{\sc J. Moser}
A New Proof of de Giorgi's Theorem Concerning the Regularity Problem for Elliptic Differential Equations.
\emph{Communications on Pure and Applied Mathematics}, Vol. XIII, 457-468 (1960).






\bibitem{Mu}
{\sc S. Müller}
Higher integrability of determinants and weak convergence in {$L^1$}
\emph{J. Reine Angew. Math.} 412 (1990) 20-34.


\bibitem{Nash}
{\sc J. Nash}
Continuity of solutions of parabolic and elliptic equations
\emph{Amer. J. Math.}  80 (1958), 931{-}954



\bibitem{Nolder}
{\sc C. A. Nolder}
Hardy-Litllewood Theorems for A-Harmonic Tensors.
\emph{ILLINOIS JOURNAL OF MATHEMATICS}   Volume 43, Number 4, Winter 1999




\bibitem{NV}
{\sc F. Nazarov and A. Volberg}
Heating of The Ahlfors-Beurling Operator, and Estimates of its Norm
\emph{St. Petersburg Math. J.}, Vol. 15 (2004), No. 4, Pages 563{-}573. 









\bibitem{PSW}
{\sc S. Petermichl, L. Slavin and B.D. Wick}
New estimates for the Beurling-Ahlfors operator on differential forms. 
\emph{J. Operator Theory}, 65 (2011), no.2, 307{-}324.





\bibitem{R}
{\sc A. Rosén}
Geometric Multivector Analysis
\emph{Birkhaüser} 2019







\bibitem{Scott}
{\sc C. Scott}
Lp Theory of Differential Forms on Manifolds
{\em Transactions of the American Mathematical Society}, Vol. 347, No. 6 (Jun., 1995), pp. 2075-2096


\bibitem{Sho}
{\sc R. E. Showalter}
Monotone Operators in Banach Space and Nonlinear Partial Differential Equations
{\em Mathematical Surveys and Monographs}, AMS Vol. 49 1997 Rhode Island. 




\bibitem{Sw}
{\sc G. Schwarz}
Hodge Decompositions - A method for solving boundary value problems
\emph{Springer Lecture Notes in Mathematics} (1995) 1607





\bibitem{S}
{\sc E. Stein}
Singular Integrals and Differentiability Properties of Functions
\emph{Princeton University Press}


\bibitem{SW}
{\sc E. Stein and G. Weiss}
On the theory of harmonic functions in several variables I. The theory of {$H^p$-} spaces. 
\emph{Acta Math.} Volume 103, Number 1-2 (1960), 25-62.




\bibitem{Tat}
{\sc L. Tatar}
The Compensated Compactness Method Applied to Systems of Conservation Laws 
{\em  NATO Science Series C: Systems of Nonlinear Partial Differential Equations} p. 263{-}285 1983. 





\bibitem{Tem}
{\sc R. Temam}
Navier-Stokes Equations 
Theorey and Numerical Analysis
{\em AMS Chelsea Publishing } 2001 Providence, Rhode Island. 


\bibitem{To}
{\sc X. Tolsa}
Jump formulas for singular integrals and layer potentials on rectifiable sets
{\em Proc. Amer. Math. Soc.} 148 (2020), 4755-4767


\bibitem{Witten}
{\sc E. Witten}
Supersymmetry and Morse Theory
{\em J. DIFFERENTIAL GEOMETRY} 17 (1982) 661-692







\bibitem{Wu1}
{\sc S. Wu}
Well-Posedness in Sobolev Spaces of the Full Water Wave Problem in 3D.
{\em Journal of the American Mathematical Society}, Volume 12, Number 2, April 1999, p. 445-495.


\bibitem{Wu2}
{\sc S. Wu}
Global well-posedness of the 3{-}D full water wave problem.
{\em Invent. math.}, (2011) 184: 125-220




\end{thebibliography}
\end{document}